\documentclass[10pt]{amsart}

\usepackage{amsmath}
\usepackage{amsfonts}
\usepackage{amsthm}
\usepackage{amssymb}

\usepackage{hyperref}

\usepackage{xypic}
\usepackage[active]{srcltx}
\usepackage{graphicx}

\usepackage{pgf}
\usepackage{tikz}
\usetikzlibrary{positioning,automata,shapes}

\usepackage{color}

\usepackage{enumitem}

\newcommand{\pv}[1]{\ensuremath{{\mathsf{#1}}}}
\newcommand{\Om}[2]{\ensuremath{\widehat F_{\pv #2}{(#1)}}}

\newcommand{\te}[1]{\mathsf t_{#1}}
\newcommand{\be}[1]{\mathsf i_{#1}}

\newcommand{\J}{\ensuremath{\mathcal{J}}}
\newcommand{\D}{\ensuremath{\mathcal{D}}}
\newcommand{\R}{\ensuremath{\mathcal{R}}}
\renewcommand{\L}{\ensuremath{\mathcal{L}}}
\renewcommand{\H}{\ensuremath{\mathcal{H}}}
\newcommand{\K}{\ensuremath{\mathcal{K}}}

\newcommand{\Kar}{\ensuremath{\mathbb{K}}}

\newcommand{\Mor}{\ensuremath{\operatorname{Mor}}}
\newcommand{\Obj}{\ensuremath{\operatorname{Obj}}}

\newcommand{\NN}{\mathbb N}
\newcommand{\ZZ}{\mathbb Z}

\newcommand{\ci}[1]{\ensuremath{{\mathcal {#1}}}}
\newcommand{\Mir}{\ensuremath{{\operatorname{\ci {M}}}}}
\newcommand{\Sha}{\ensuremath{{\operatorname{{Sha}}}}}

\newcommand{\Lo}[1]{\ensuremath{\ci L\pv {#1}}}

\newcommand{\Cl}[1]{\ensuremath{\ci #1}}

\newcommand{\dia}{\ensuremath{\diamond}}

\newcommand\malcev{\mathop{\raise1pt\hbox{\footnotesize$\bigcirc$\kern-8pt\raise1pt
      \hbox{\tiny$m$}\kern1pt}}}

\newcommand{\li}[1]{\ensuremath{\overleftarrow{#1}.\overrightarrow{#1}}}

\newcommand{\ori}[1]{\ensuremath{\overrightarrow{#1}}}
\newcommand{\ole}[1]{\ensuremath{\overleftarrow{#1}}}

\newtheorem{Thm}{Theorem}[section]

\newtheorem{Prop}[Thm]{Proposition}
\newtheorem{Lemma}[Thm]{Lemma}
\newtheorem{Cor}[Thm]{Corollary}

{\theoremstyle{definition}
\newtheorem{Fact}[Thm]{Fact}
}

{\theoremstyle{remark}
\newtheorem{Example}[Thm]{Example}}
{\theoremstyle{definition}
}
{\theoremstyle{definition}
\newtheorem{Def}[Thm]{Definition}}

{
  \theoremstyle{definition}
\newtheorem{Rmk}[Thm]{Remark}
}

\numberwithin{equation}{section}

\begin{document}

\title[The Karoubi envelope of the mirage of a subshift]{The Karoubi envelope of the mirage of a subshift}

\thanks{%
  The work of A.~Costa
  was carried out in part at City College of New York, CUNY, whose hospitality is gratefully acknowledged, with the support of the FCT sabbatical scholarship SFRH/BSAB/150401/2019,
  and it was partially supported by the Centre for Mathematics of the University of Coimbra - UIDB/00324/2020, funded by the Portuguese Government through FCT/MCTES}

\author[A.~Costa]{Alfredo Costa}

\address{University of Coimbra, CMUC, Department of Mathematics,
  Apartado 3008, EC Santa Cruz,
  3001-501 Coimbra, Portugal.}
\email{amgc@mat.uc.pt}

\author[B.~Steinberg]{Benjamin Steinberg}

\address{Department of Mathematics, City College of New York, Convent Avenue at 138th Street, New York, New York 10031, USA.}

\email{bsteinberg@ccny.cuny.edu}



\begin{abstract}
  We study a correspondence associating to each subshift $\Cl X$
  of $A^{\ZZ}$ a subcategory
  of the Karoubi envelope of the free profinite semigroup generated by $A$.
  The objects of
  this category
  are the idempotents in the mirage
  of~$\Cl X$, that is,
  in the set of pseudowords whose finite factors are blocks of~$\Cl X$.
  The natural equivalence class of
  the category
  is
  shown to be invariant under flow equivalence.
  As a corollary of our proof, we deduce the
  flow invariance of the profinite group that Almeida associated to each irreducible subshift.
  We also show, in a functorial manner, that the isomorphism class of
  the category
  is invariant under conjugacy.
  Finally, we see that the zeta function of~$\Cl X$
  is naturally encoded in
  the category.
  These results hold,  with obvious translations, for
  relatively free profinite semigroups over many pseudovarieties,
  including all of the form $\overline{\pv H}$, with $\pv H$
  a pseudovariety of groups.
\end{abstract}

\maketitle

\section{Introduction}
\label{sec:introduction}

Relatively free profinite semigroups and their elements, pseudowords,
play an important role in finite semigroup theory.
Around 2003, Almeida established
the following connection between them and
symbolic dynamics~\cite{Almeida:2003b}: in the $A$-generated relatively free profinite semigroup $\Om AV$, where $\pv V$ is a semigroup pseudovariety containing $\Lo {Sl}$, associate to each subshift $\Cl X$ of $A^{\ZZ}$ the topological closure in $\Om AV$ of the set $L(\Cl X)$ of finite blocks of $\Cl X$.
This connection proved to be very useful
for a better understanding of structural aspects
of $\Om AV$, even in the most difficult case where $\pv V$ is
the pseudovariety $\pv S$ of all finite semigroups. One of the most relevant aspects of this line of research concerned the
case of irreducible subshifts. When $\Cl X$ is irreducible,
the union of the $\J$-classes intersecting the topological closure $\overline {L(\Cl X)}\subseteq \Om AV$ contains a minimum $\Cl J$-class $J_{\pv V}(\Cl X)$,
which is a regular $\J$-class of $\Om AV$.
If $\pv V=\pv V\ast\pv D$, the corresponding Sch\"utzenberger group $G_{\pv V}(\Cl X)$,
the profinite group isomorphic to all maximal subgroups of $J_{\pv V}(\Cl X)$,
is invariant under \emph{conjugacy}~\cite{ACosta:2006}, the name given to the isomorphism
relation between topological dynamical systems.
The conjugacy invariance of $G_{\overline{\pv H}}(\Cl X)$ was crucial
to the proof in~\cite{ACosta&Steinberg:2011} that if $\pv H$ is an extension-closed pseudovariety
of groups containing infinitely many groups of prime order, and if $L(\Cl X)$ is recognized by a semigroup
of $\overline{\pv H}$, then the maximal subgroups of $J_{\overline{\pv H}}(\Cl X)$ are free pro-$\pv H$ groups of countable rank, unless $\Cl X$
is periodic, in which case $G_{\overline{\pv H}}(\Cl X)$ is free pro-aperiodic.
The profinite group $G_{\pv S}(\Cl X)$ was also identified in many instances where $\Cl X$ is minimal~\cite{Almeida:2004a,Almeida&ACosta:2013,Almeida&ACosta:2016b}, in the process being shown to sometimes not be free, although it is always projective accordingly to~\cite{Rhodes&Steinberg:2008}.
In this paper, we add information about the dynamical
meaning of $G_{\pv V}(\Cl X)$, as briefly contextualized in the following paragraphs.

Some techniques used in~\cite{ACosta:2006}
to prove the conjugacy invariance of $G_{\pv V}(\Cl X)$
were adapted in the same paper in order to obtain conjugacy invariants
encoded in the syntactic semigroup $S(\Cl X)$ of the language $L(\Cl X)$,
when $\Cl X$ is sofic (that is, when $L(\Cl X)$ is rational).
These syntactic invariants
where shown in~\cite{ACosta&Steinberg:2016}
to be invariants with respect to another
relation of significant importance in symbolic dynamics, \emph{flow equivalence}
(cf.~\cite[Section~13.6]{Lind&Marcus:1996}),
the relation, coarser than conjugacy, identifying subshifts
with suitably equivalent suspension flows (or mapping tori).
This was done by showing that those invariants
are encoded in the Karoubi envelope of $S(\Cl X)$,
a small category whose equivalence class was shown in~\cite{ACosta&Steinberg:2016} to be
a flow invariant (even if $\Cl X$ is not sofic),
and in fact, as also proved there, the best possible syntactic flow equivalence invariant for sofic systems.

The importance for semigroup theory of the Karoubi envelope $\Kar(S)$
of a semigroup $S$ became clear with Tilson's seminal paper \cite{Tilson:1987}
(there, it is denoted~$S_E$).
Inspired by~\cite{ACosta&Steinberg:2016}, we now consider the Karoubi envelope of $\Om AV$  in relation with the subshift $\Cl X$.
In fact, we view $\Kar(\Om AV)$
as a compact topological category.
In the exploration of the connections
between symbolic dynamics and free profinite semigroups,
the convenience of considering
the set of pseudowords of $\Om AV$ whose finite factors are elements of $L(\Cl X)$ soon became apparent (here, as before, $\pv V\supseteq\Lo {Sl}$)~\cite{ACosta:2006,Almeida&ACosta:2007a}.
This set, the \emph{mirage} of $\Cl X$ in $\Om AV$,
denoted $\Mir_{\pv V}(\Cl X)$, always contains $\overline{L(\Cl X)}$, but
it may contain elements not in $\overline{L(\Cl X)}$.
The arrows $(e,u,f)$ in $\Kar(\Om AV)$
such that $u\in \Mir_{\pv V}(\Cl X)$
form a compact subcategory of $\Kar(\Om AV)$,
which we call the Karoubi envelope
of the mirage of $\Cl X$ (with respect to $\pv V$),
and denote by $\Kar(\Mir_{\pv V}(\Cl X))$.

 In this paper, we show that the correspondence
 $\Cl X\mapsto \Kar(\Mir_{\pv V}(\Cl X))$
 establishes a functor from
 the category of symbolic dynamical systems to that of compact zero-dimensional  categories,
 whenever $\pv V=\pv V*\pv D$ and $\pv V\supseteq \Lo{Sl}$.
 From this functor, we get for free
 a new proof that the profinite group 
 $G_{\pv V}(\Cl X)$ is a conjugacy invariant, when $\Cl X$ is irreducible,
 $\pv V=\pv V*\pv D$ and $\pv V\supseteq \Lo {Sl}$.
 Under the additional mild assumption that~$\pv V$ is monoidal,
 we show that the natural equivalence class of $\Kar_{\pv V}(\Cl X)$
 is actually invariant under flow equivalence, deducing from
 that, in the irreducible case, the invariance under flow equivalence of the profinite group $G_{\pv V}(\Cl X)$. 

When $\Cl X$ is irreducible,
the mirage $\Mir_{\pv V}(\Cl X)$ contains a minimum $\J$-class
$\widetilde J_{\pv V}(\Cl X)$, which is regular
and therefore possesses a profinite
Sch\"utzenberger group $\widetilde G_{\pv V}(\Cl X)$
isomorphic to its maximal subgroups.
We deduce, just as for $G_{\pv V}(\Cl X)$,
that $\widetilde G_{\pv V}(\Cl X)$ is a conjugacy invariant when
$\Lo {Sl}\subseteq \pv V=\pv V*\pv D$, and that $\widetilde G_{\pv V}(\Cl X)$
is a flow equivalence invariant
when $\pv V$ is also monoidal.
It may be interesting and challenging to investigate the group
$\widetilde G_{\pv V}(\Cl X)$, which remains largely unknown when $\Cl X$ is not minimal
(in the minimal case the equality $G_{\pv V}(\Cl X)=\widetilde G_{\pv V}(\Cl X)$ holds).

The structure of $\Kar(\Mir_{\pv V}(\Cl X))$
says more about $\Cl X$ than $G_{\pv V}(\Cl X)$ or~$\widetilde G_{\pv V}(\Cl X)$.
Indeed, we show that the zeta function of
$\Cl X$ is
encoded in $\Kar(\Mir_{\pv V}(\Cl X))$, by showing
that the periodic points correspond to the objects of $\Kar(\Mir_{\pv V}(\Cl X))$ with a finite isomorphism class.
This is done by identifying the regular $\J$-classes of $\Om AV$
containing a finite number of $\ci H$-classes.

This introduction is followed by two sections
of preliminaries, about symbolic dynamics and free profinite semigroups.
Section~\ref{sec:pseud-block-codes}
gives tools for
the establishment,
in Section~\ref{sec:funct-corr-from},
of the above mentioned functor between subshifts and compact categories. The results about flow equivalence are treated in Section~\ref{sec:flow-equivalence},
(with an appendix at the end of the paper, concerning one technical consequence).
Finally, Section~\ref{sec:relat-with-zeta} deals with the connections
with the zeta function.

\section{Symbolic dynamics }
\label{sec:symbolic-dynamics}

In this section we provide a brief introduction
to symbolic dynamics. For a very developed introduction,
see the book~\cite{Lind&Marcus:1996}. In the context of this paper, the short text~\cite{ACosta:2018} might also be useful.

It is helpful to begin by recalling some terminology and notation about free semigroups. In this paper, an \emph{alphabet} will always be a finite nonempty set.
The elements of the alphabet~$A$ are the \emph{letters} of $A$.
A \emph{word} over $A$ is a finite nonempty sequence of letters of $A$.
The words over $A$ form the semigroup $A^+$, for the operation of concatenation
of words.
The free monoid $A^*$ is obtained from $A^+$ by adjoining the empty sequence
(the \emph{empty word}, here denoted by the symbol $\varepsilon$), which is the neutral element of $A^*$ for the concatenation operation.
As it is usual in the literature, the length of a word $u$ is denoted by $|u|$.

\subsection{The category of symbolic dynamical systems}
Let $A$ be an alphabet. Endow $A$ with the discrete topology, and
$A^{\ZZ}$ with the corresponding product topology.
Note that, by Tychonoff's theorem and our convention that all alphabets are finite sets, the space $A^{\ZZ}$ is compact. We assume
  that compact topological spaces are Hausdorff.
The \emph{shift map}
$\sigma_A\colon A^{\ZZ}\to A^{\ZZ}$
is the mapping defined by
\begin{equation*}
  \sigma_A((x_i)_{i\in\ZZ})=(x_{i+1})_{i\in\ZZ}.
\end{equation*}
A \emph{symbolic dynamical system}, also called a \emph{subshift}, or just
a \emph{shift}, of $A^{\ZZ}$ is a nonempty closed subset~$\Cl X$ of $A^{\ZZ}$ such that
$\sigma_A(\Cl X)=\Cl X$.
The subshifts are the objects of the category of symbolic dynamical systems. In this category, a morphism between a subshift of $A^{\ZZ}$ and a subshift of
$B^{\ZZ}$ is a continuous mapping $\varphi\colon \Cl X\to\Cl Y$ such that the diagram
\begin{equation*}
  \xymatrix{
    \Cl X\ar[r]^{\varphi}\ar[d]_{\sigma_A}&\Cl Y\ar[d]^{\sigma_B}\\
    \Cl X\ar[r]^{\varphi}&\Cl Y.
  }
\end{equation*}
commutes. In the category of symbolic dynamical systems, an isomorphism is usually called a \emph{conjugacy}, and two isomorphic subshifts are said to
be \emph{conjugate}.

A \emph{block} of a subshift $\Cl X$ is a nonempty word $u$
appearing in some element of~$\Cl X$, that is, a word
$u$ such that for some $x\in\Cl X$ and some integers $i\leq j$, the equality $u=x_ix_{i+1}\ldots x_{j-1}x_{j}$ holds. The word
$u$ may then be denoted by $x_{[i,j]}$.
We denote the set of blocks of $\Cl X$ by $L(\Cl X)$.
One has $\Cl X\subseteq \Cl Y$ if and only if $L(\Cl X)\subseteq L(\Cl Y)$,
for all subshifts $\Cl X$ and $\Cl Y$ of $A^{\ZZ}$.

The notion of block
is the groundwork for a form of producing morphisms between subshifts,
which we next describe.
For alphabets~$A$ and~$B$, and a positive integer~$N$,
take a map $\Phi\colon A^{N}\to B$, where $N$ is some positive integer.
Let~$m$ and~$n$ be nonnegative integers such that $N=m+n+1$.
In the context of this paper,
such a map is called a \emph{block map}. The integer $N$ is the \emph{window size}
of the block map.
Consider the mapping $\varphi\colon A^{\ZZ}\to B^{\ZZ}$
defined by the correspondence
\begin{equation*}
  \varphi((x_i)_{i\in\ZZ})=(\Phi(x_{[i-m,i+n]}))_{i\in\ZZ}.
\end{equation*}
We say that $\varphi$ is the \emph{sliding block code}
from $A^{\ZZ}$ to $B^{\ZZ}$ with block map $\Phi$, \emph{memory}~$m$ and \emph{anticipation}~$n$.
More generally, if the subshifts $\Cl X\subseteq A^{\ZZ}$ and
$\Cl Y\subseteq B^{\ZZ}$
are such that $\varphi(\Cl X)\subseteq \Cl Y$,
then the induced restriction
$\varphi\colon \Cl X\to \Cl Y$
is also called a sliding block code, from $\Cl X$ to $\Cl Y$,
with memory $m$ and anticipation $n$.
Note that $\varphi\colon \Cl X\to \Cl Y$
is determined by the restriction of $\Phi$ to the set of words of $L(\Cl X)$
with length $N$. 

We are now ready to state a fundamental result
of symbolic dynamics, the Curtis--Hedlund--Lyndon theorem~\cite{Hedlund:1969},
fully characterizing the morphisms of subshifts (cf.~\cite[Theorem 6.2.9]{Lind&Marcus:1996}).

\begin{Thm}\label{t:morphisms-are-sliding-block-codes}
  The morphisms between subshifts are precisely the sliding block codes.
\end{Thm}

Let us say that a block map $\Psi\colon A^N\to B$
is a \emph{central block map} if $N$ is odd.
If $N=2k+1$, then
we say that $k$ is the \emph{wing} of $\Psi$.
Given a sliding block code $\psi\colon \Cl X\to\Cl Y$,
a \emph{central block map of $\psi$} is a central block map
$\Psi\colon A^{2k+1}\to B$
for which $\psi$ has $\Psi$ as a block
map with both  memory and anticipation equal to~ $k$.

\begin{Fact}\label{fact:central-sliding-block-codes}
Every sliding block code $\psi\colon\Cl X\to\Cl Y$
has a central block map.
\end{Fact}

\begin{proof}
  If $\psi\colon\Cl X\to\Cl Y$
is a sliding block code with block map $\Phi\colon A^{m+n+1}\to B$, memory $m$ and anticipation $n$,
and letting $k=\max\{m,n\}$,
then the
map $\Psi\colon A^{2k+1}\to B$ defined by
\begin{equation*}
  \Psi(a_{-k}a_{-k+1}\cdots a_{-1}a_0a_1\cdots a_{k-1}a_k)
  =
  \Phi(a_{-m}a_{-m+1}\cdots a_{-1}a_0a_1\cdots a_{n-1}a_n),
\end{equation*}
where $a_i\in A$ for all $i\in\{-k,-k+1,\ldots,k-1,k\}$,
is such that $\psi$ has $\Psi$
as block map with memory $k$ and anticipation $k$.
\end{proof}

A \emph{$1$-code} is a block map having a central block map of window size~$1$
(that is, wing~$0$). A $1$-conjugacy is a $1$-code that is a conjugacy.

\begin{Rmk}
  The composition of two $1$-codes is a $1$-code.
\end{Rmk}

With the help of Theorem~\ref{t:morphisms-are-sliding-block-codes},
one gets the next useful result. In the diagram included, the
double arrow represents an isomorphism, a convention reprised throughout the paper.

\begin{Prop}[{cf.~\cite[Proposition~1.5.12]{Lind&Marcus:1996}}]\label{p:decomposition-of-morphisms}
  If $\varphi\colon \Cl X\to \Cl Y$ is a morphism of subshifts,
  then there are $1$-codes $\alpha\colon\Cl Z\to\Cl X$
  and $\beta\colon \Cl Z\to\Cl Y$, for some subshift $\Cl Z$,
  such that $\alpha$ is a conjugacy and the diagram
  \begin{equation*}
    \xymatrix{
      &\Cl Z\ar[rd]^{\beta}\ar@{=>}[ld]_{\alpha}&\\
      \Cl X\ar[rr]^{\varphi}&&\Cl Y
    }
  \end{equation*}
  commutes, that is, $\varphi=\beta\circ\alpha^{-1}$.
\end{Prop}

\subsection{Classification of subshifts}

We review some important classes of subshifts.

A subshift $\Cl X$ of $A^{\ZZ}$
is \emph{irreducible} if there is $x\in A^{\ZZ}$ with positive dense orbit, that is, such that $\{\sigma_A^n(x)\mid n\geq 1 \}$ is dense in~$\Cl X$.
Clearly, being irreducible is a property invariant under conjugacy.
Next is a convenient
characterization in terms of words.
Say that a subset $K$ of a semigroup~$S$ is \emph{irreducible} if, for every $u,v\in K$, there is $w\in S$ such that $uwv\in S$. It turns out that
a subshift $\Cl X$ of $A^{\ZZ}$ is irreducible if and only if $L(\Cl X)$
is irreducible in the semigroup $A^+$.

A subshift $\Cl X$ of $A^{\ZZ}$ is~\emph{sofic} when its elements are the labels of the bi-infinite paths in a fixed graph with edges labeled by letters of $A$. A sofic subshift is irreducible if and only such a graph can be chosen to be strongly connected.

\begin{Example}\label{eg:even-subshift}
  The \emph{even subshift} is the irreducible sofic subshift $\Cl X$ of $\{a,b\}^{\ZZ}$ with
  presentation given by the labeled graph in Figure~\ref{fig:even-shift}.
       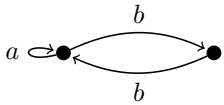
\begin{figure}[h]
     \centering\vspace{-0.2 cm}
     \begin{tikzpicture}[shorten >=1pt, node distance=2cm, on grid,initial text=,semithick]
       \tikzstyle{state}=[draw,circle,minimum size=5pt,inner sep=1,fill]
  \node[state]   (1)                {};
  \node[state]   (2) [right=of 1]   {};
  \path[->]  (1)   edge  [loop left]     node {$a$} ()
             (1)   edge [bend left=25] node [above] {$b$} (2)
             (2)   edge [bend left=25] node [below] {$b$} (1);
\end{tikzpicture}\vspace{-0.5 cm}
     \caption{The even subshift.}
     \label{fig:even-shift}
   \end{figure}
That is, when $u$ is a word over~$\{a,b\}$, one has $u\in L(\Cl X)$ if and only if $ab^na$ is not a factor of $u$
for some odd~$n$.
\end{Example}

A subshift conjugate to a sofic subshift is also sofic.
Within sofic shifts, the more salient class closed under
conjugacy is that of \emph{finite type shifts}. These are the subshifts
conjugate with \emph{edge shifts}, the latter being the subshifts presented by a labeled graph  where distinct edges have distinct labels. 
The most famous open problem of symbolic dynamics is to know if we can always decide if two given edge shifts are conjugate or not.

A conjugacy-closed class quite distinct from sofic shifts,
that has received a lot of attention in the literature
(see~\cite{Fogg:2002,Lothaire:2001}), is that of \emph{minimal subshifts}:
the subshift $\Cl X$  is minimal if, whenever $\Cl Y$ is a subshift, the inclusion  $\Cl Y\subseteq \Cl X$ implies $\Cl Y=\Cl X$.
All minimal subshifts are irreducible.

\section{Free profinite semigroups}

We assume knowledge about basic features of semigroups,
like Green's relations (a short introduction may be found in~\cite[Appendix A]{Rhodes&Steinberg:2009qt}).
In this section we quickly review some aspects of profinite semigroup theory. One of our purposes is to fix notation.  For a more paused but short introduction to the subject, see for example~\cite{Almeida:2003cshort}. The book~\cite{Rhodes&Steinberg:2009qt} is also an updated guiding reference.
We finish this section reviewing
some connections with symbolic dynamics.

\subsection{Languages and pseudovarieties}

A subset of the free semigroup~$A^+$ is called
a \emph{language} of~$A^+$. A language $L$ of $A^+$ is said to be
\emph{recognized} by a finite semigroup~$S$ if there is a
homomorphism $\varphi\colon A^+\to S$ such that
$L=\varphi^{-1}(\varphi(L))$.
Without giving details, we recall the well-known fact, not difficult to prove, that a language $L$ is recognizable in this algebraic sense if and only if
it is recognized by some finite automaton.
Hence,
  a subshift $\Cl X$ is sofic if and only if $L(\Cl X)$ is recognizable.

A \emph{pseudovariety of semigroups} is a class of finite semigroups
closed under taking subsemigroups, homomorphic images, and finitary products.
The intersection of pseudovarieties is clearly a pseudovariety,
and so we may talk of the pseudovariety generated by a class of semigroups.
In Section~\ref{sec:flow-equivalence} we shall have to restrict ourselves
to monoidal pseudovarieties, the semigroup
pseudovarieties generated by a class of finite monoids.
Here are some pseudovarieties of semigroups, relevant for this paper, with only the last three examples not being monoidal:
  \begin{itemize}
  \item The pseudovariety $\pv S$ of all finite semigroups.
    \item The pseudovariety $\pv I$ of one-element semigroups.
  \item The pseudovariety $\pv G$ of all finite groups.
   \item The pseudovariety $\pv A$ of  finite \emph{aperiodic} semigroups,
     that is, semigroups all of whose subgroups (i.e.,
     subsemigroups with group structure) are trivial.
   \item The pseudovariety $\overline{\pv H}$
     of finite semigroups whose maximal subgroups
     belong to the pseudovariety of groups $\pv H$.
  \item The pseudovariety $\pv {Sl}$
    of \emph{semilattices}, that is, commutative semigroups
    all of whose elements are idempotent.
  \item Given a pseudovariety $\pv V$, the pseudovariety
    $\Lo V$ of semigroups $S$ such that, for every idempotent $e$ of $S$,
    the subsemigroup $eSe$ belongs to $\pv V$.
  \item The pseudovariety $\pv N$ of finite \emph{nilpotent} semigroups,
    which are the finite semigroups with a zero element $0$
    such that $S^n=\{0\}$ for some $n\geq 1$.
  \item The pseudovariety $\pv D$ of finite  semigroups such that $Se=\{e\}$ for every idempotent $e$ if $S$.  
  \end{itemize}

One of the main interests of semigroup pseudovarieties is that
quite often one decides if a recognizable language $L$ satisfies a certain combinatorial property by deciding if $L$
is recognized by a semigroup from a certain pseudovariety $\pv V$.
Sometimes, these pseudovarieties are expressed
as the result of operations on other pseudovarieties.
An important example is the \emph{semidirect product} $\pv V\ast\pv W$
of two pseudovarieties $\pv V$ and $\pv W$,
the least semigroup pseudovariety containing the semidirect products of elements of $\pv V$ with  elements of~$\pv W$.
This is an associative operation on the lattice of pseudovarieties.
Another important operation, non-associative, is the \emph{Mal'cev product} $\pv V\malcev \pv W$,
briefly mentioned in one example later on,
and which is the pseudovariety generated by finite semigroups $S$
 for which there is a homomorphism $\varphi:S\to T$
 with $T\in\pv W$ and $\varphi^{-1}(e)\in \pv V$
 for every idempotent $e$ of~$T$.
The interested reader is referred
to~\cite{Rhodes&Steinberg:2009qt} for more information on these operations.

\begin{Example}
  A language $L$ of $A^+$ is said to be \emph{locally testable} if it is a finite Boolean combination of languages of the form $uA^*$, $A^*u$ and $A^*uA^*$, where $u$ denotes a (non-fixed) word of $A^+$.
  One of the first successes of finite semigroup theory was the proof
that being locally testable is a decidable property by showing
that a language is locally testable if and only if
it is recognized by a semigroup in~$\Lo {Sl}$~\cite{Brzozowski&Simon:1973,McNaughton:1974,Zalcstein:1973a,Zalcstein:1972}.
In terms of pseudovarieties, this amounts to the equality
$\Lo{Sl}=\pv {Sl}*\pv D$.
If $\Cl X$ is a subshift of $A^{\ZZ}$ of finite type, then $L(\Cl X)$ is locally testable: indeed, it is of the form $L(\Cl X)=A^+\setminus A^*WA^*$ for some finite set $W$ of words.
  Conversely, if $\Cl X$ is irreducible and
  $L(\Cl X)$ is locally testable, then
  $\Cl X$ is of finite type (see~\cite{ACosta:2007a} for a proof).
\end{Example}

\subsection{Relatively free profinite semigroups}

A \emph{compact semigroup} is a semigroup endowed with a topology for which the semigroup operation is continuous. We view finite semigroups as compact semigroups with the discrete topology.

In general, a pseudovariety of semigroups $\pv V$ is too small to contain
free objects. An
approach
commonly followed
is to find room for free objects by
considering the inverse limits of semigroups of $\pv V$, viewed as compact semigroups. These semigroups are the \emph{pro-$\pv V$ semigroups}.
Note that the semigroups from $\pv V$ are pro-$\pv V$. Conversely, finite
pro-$\pv V$ semigroups must belong to~$\pv V$. When dealing with the pseudovariety
$\pv S$ of all finite semigroups,
one uses the terminology \emph{profinite} instead of pro-$\pv S$.

If $A$ is an alphabet, then
the natural inverse limit defined by the finite quotients of $A^+$ that belong to $\pv V$ is a pro-$\pv V$ semigroup, denoted by $\Om AV$.
Our assumption that all alphabets are finite guarantees that the topology
of $\Om AV$ is metrizable.

The least closed subsemigroup of $\Om AV$
containing the image of the generating map
$\iota\colon A\to\Om AV$ is $\Om AV$. The pro-$\pv V$ semigroup $\Om AV$ is the free object
generated by $A$ in the category of pro-$\pv V$ semigroups,
as the map $\iota\colon A\to \Om AV$ satisfies the following universal property:
for every map $\varphi\colon A\to S$ into a pro-$\pv V$ semigroup, there is a
unique continuous semigroup homomorphism $\hat\varphi\colon \Om AV\to S$ such that $\hat\varphi\circ \iota=\varphi$. Hence, we say that
$\Om AV$ is the \emph{free pro-$\pv V$ semigroup generated by $A$}, or that it is the \emph{free profinite semigroup relative to $\pv V$ generated by $A$}.

Let $\pv V$ be a pseudovariety of semigroups containing
the pseudovariety $\pv N$ of finite nilpotent semigroups.
Then the unique extension of $\iota\colon A\to\Om AV$ to
a semigroup homomorphism $A^+\to\Om AV$
is an injective map, and it is from this viewpoint
that we consider $\iota$ as the inclusion and $A^+$ as a subsemigroup of $\Om AV$. 
One should bear in mind that $A^+$ is dense in~$\Om AV$. Moreover,
the hypothesis $\pv N\subseteq\pv V$
guarantees that the elements of $A^+$ are isolated in $\Om AV$.
Hence, one may view the elements
of $\Om AV$ as generalizations of finite words, for which reason we call them \emph{pseudowords}, and
we are justified to say that the elements of $A^+$ are the \emph{finite}
pseudowords of $\Om AV$, while
those of $\Om AV\setminus A^+$ are the \emph{infinite} pseudowords of $\Om AV$.

The following theorem gives us a glimpse of the reasons
why relatively free profinite semigroups and pseudowords are useful.
It essentially says that $\Om AV$ is the Stone dual of the Boolean algebra of languages recognized by semigroups of~$\pv V\supseteq \pv N$.

 \begin{Thm}[{cf.~\cite[Theorem 3.6.1]{Almeida:1994a}}]\label{t:rational-open}
   Let $\pv V$ be a pseudovariety of semigroups containing $\pv N$.
   Then a language $L\subseteq A^+$ is recognized by a semigroup of $\pv V$
   if and only if its topological closure $\overline{L}$ in
  $\Om AV$ is open, if and only if $L=K\cap A^+$
   for some clopen subset $K$ of $\Om AV$.
 \end{Thm}

Given a semigroup $S$, we denote by $S^I$ the monoid
$S\uplus \{I\}$ extending the semigroup operation of $S$ by adjoining an identity $I$. For example, $A^*$ is (isomorphic to) the monoid $(A^+)^I$.
If $S$ is a compact semigroup, then we view $S^I$
as a compact monoid extending $S$, by letting $I$ be an isolated point.
If $\varphi\colon S\to T$ is a function between semigroups,
then its extension $S^I\to T^I$
that maps $I$ to $I$, may still be denoted by~$\varphi$, in the absence of confusion.

\subsection{Pseudowords defined by subshifts}
\label{sec:conn-with-symb}

We briefly review some data relating relatively free profinite semigroups
with symbolic dynamics,
in part already met in Section~\ref{sec:introduction}, most of which is explained in~\cite[Section 3.2]{ACosta:2006} or~\cite{Almeida&ACosta:2007a}. Fix a semigroup pseudovariety $\pv V$ containing~$\Lo{Sl}$. 
The \emph{mirage} of a subshift $\Cl X\subseteq A^{\ZZ}$
is the set $\Mir_{\pv V}(\Cl X)$ of elements of $\Om AV$ whose
finite factors are in~$L(\Cl X)$.
It helps to also consider the set $\Mir_{\pv V,k}(\Cl X)$ of elements of $\Om AV$ whose
finite factors of length at most $k$ belong to~$L(\Cl X)$.
One clearly has
$\Mir_{\pv V}(\Cl X)=\bigcap_{k\geq 1}\Mir_{\pv V,k}(\Cl X)$.

\begin{Rmk}\label{rmk:mvk-is-clopen}
The set $\Mir_{\pv V,k}(\Cl X)$
is the finite intersection of subsets of $\Om AV$ of the form $\Om AV\setminus\overline{A^*uA^*}$, with $u\in A^+\setminus L(\Cl X)$ having length at most $k$.
Hence, $\Mir_{\pv V,k}(\Cl X)$ is clopen, in view of Theorem~\ref{t:rational-open},
as the locally testable language $A^*uA^*$
is recognized by a semigroup of~$\Lo {Sl}$.  
\end{Rmk}

A subset $K$ of a semigroup $S$ is said to be
\emph{factorial} if every factor of an element of $K$ belongs to $K$,
and is said to be \emph{prolongable} with respect to a
subset $A$ of~$S$ if
\mbox{$uA\cap K\neq\emptyset$} and $Au\cap K\neq\emptyset$
for each $u\in K$.
The languages of the form~$L(\Cl X)$,
with~$\Cl X$
a subshift of $A^{\ZZ}$, are precisely the nonempty languages of $A^+$
that are factorial and prolongable with respect to $A$. With routine topological arguments, one easily deduces that $\overline{L(\Cl X)}$,
$\Mir_{\pv V,k}(\Cl X)$ and $\Mir_{\pv V}(\Cl X)$
are prolongable subsets of $\Om AV$, with respect to $A$.
Note also that each of these sets
contains infinite pseudowords, for example, every accumulation point
of a sequence of words in $L(\Cl X)$ with increasing~length.

Again applying standard topological arguments,
one sees that the inclusion
$\overline{L(\Cl X)}\subseteq \Mir_{\pv V}(\Cl X)$
holds. This inclusion may be strict. In fact, it is clear that $\Mir_{\pv V,k}(\Cl X)$ and $\Mir_{\pv V}(\Cl X)$ are factorial, but the next example
shows that $\overline{L(\Cl X)}$ may not be factorial,
as seen in Example~\ref{eg:an-example-of-non-factoriality}, taken from~\cite{ACosta:2007t}. In that example, we use the notation $s^\omega$,
standard in (pro)finite semigroup theory, for the unique idempotent
in the closed subsemigroup of $S$ generated by $s$, where
$s$ is an element in a compact semigroup $S$. If $S$ is profinite, one has $s=\lim s^{n!}$.

\begin{Example}\label{eg:an-example-of-non-factoriality}
  Let $A=\{a,b,c,d\}$
  and consider the sofic subshift $\Cl X$ of $A^{\ZZ}$ presented in
  Figure~\ref{fig:not-factorial}.
     \begin{figure}[h]
       \centering
       \vspace{-0.3 cm}
     \begin{tikzpicture}[shorten >=1pt, node distance=2cm, on grid,initial text=,semithick]
       \tikzstyle{state}=[draw,circle,minimum size=5pt,inner sep=1,fill]
  \node[state]   (1)                {};
  \node[state]   (2) [right=of 1]   {};
  \node[state]   (3) [right=of 2]   {};
  \path[->]  (1)   edge  [loop left]     node {$a$} ()
             (2)   edge  [loop below]    node {$a$} ()
             (3)   edge  [loop right]    node {$a$} ()
             (1)   edge  node [below] {$b$} (2)
             (2)   edge  node [below] {$c$} (3)
             (3)   edge  [bend right=20] node [above] {$d$} (1);
\end{tikzpicture}\vspace{-0.5 cm}
     \caption{An irreducible sofic subshift.}
     \label{fig:not-factorial}
   \end{figure}
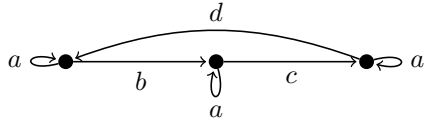
      In $\Om A{{\Lo{Sl}}}$,
   the pseudoword $v=a^{\omega}ba^{\omega}ca^{\omega}$
   belongs to $\overline{L(\Cl X)}$,
   since one clearly has $a^*ba^*ca^*\subseteq L(\Cl X)$.
   Moreover,
   in $\Om A{{\Lo{Sl}}}$
   we have $a^\omega cv=v$
   and so $cv$ is a factor of $v$. The topological closure
   of the locally testable language $K=cA^*\cap A^*bA^*\cap A^+\setminus A^*dA^*$
   is a clopen neighborhood of $cv$ (cf.~Theorem~\ref{t:rational-open}).
   Therefore, if we had $cv\in\overline{L(\Cl X)}$,
   then we would have $L(\Cl X)\cap K\neq\emptyset$,
   which is false.
\end{Example}

On the other hand, if $\pv V=\pv A\malcev\pv V$ (for example, if
$\pv V=\overline{\pv{H}}$), then $\overline{L(\Cl X)}$
is factorial~\cite{Almeida&ACosta:2007a}.
For arbitrary $\pv V$,  consider another set, the \emph{shadow of $\Cl X$}, denoted by $\Sha_{\pv V}(\Cl X)$, defined as the union of the $\J$-classes of $\Om AV$ intersecting $\overline{L(\Cl X)}$.
Note that $\Sha_{\pv V}(\Cl X)=\overline{L(\Cl X)}$ if $L(\Cl X)$ is factorial.
One has $\Sha_{\pv V}(\Cl X)\subseteq \Mir_{\pv V}(\Cl X)$,
with equality if $\Cl X$ is of finite type. The equality also holds
if $\Cl X$ is minimal, a fact recorded in Theorem~\ref{t:bijection-minimal-shifts-maximal-j-classes} below.

We already mentioned that $\Cl X$ is irreducible if and only if $L(\Cl X)$ is an irreducible subset of $A^+$.
From that, again with routine topological arguments,
one deduces that
if $\Cl X$ is irreducible then $\overline{L(\Cl X)}$,
$\Sha_{\pv V}(\Cl X)$ and $\Mir_{\pv V}(\Cl X)$ are irreducible.
If $K$ is a nonempty closed irreducible factorial subset of a compact semigroup, then
it contains a $\J$-minimum $\J$-class, which is regular, as seen in~\cite{ACosta&Steinberg:2011}. All elements of $K$ are then factors
of all elements of such $\J$-class.
Therefore, if~$\Cl X$ is irreducible,
$\Sha_{\pv V}(\Cl X)$ contains a $\J$-minimum $\J$-class $J_{\pv V}(\Cl X)$
and $\Mir_{\pv V}(\Cl X)$ contains a $\J$-minimum $\J$-class $\widetilde J_{\pv V}(\Cl X)$, both regular $\J$-classes. 

A \emph{$\J$-maximal infinite element}
of $\Om AV$ is an element $u$ of $\Om AV$ such that
\mbox{$u\leq_{\J}v$} implies $v\in A^+$.

\begin{Rmk}\label{rmk:always-a-factor}
  Every infinite pseudoword
has some infinite idempotent as a factor~\cite[Corollary 5.6.2]{Almeida:1994a},
and so every $\J$-maximal infinite element of $\Om AV$
is regular.  Moreover, every infinite pseudoword $w$ has
some $\J$-maximal infinite element
as a factor, by Zorn's Lemma,
because, by compactness, every $\leq_\J$-chain of infinite pseudowords
that are factors of $w$ clusters to an infinite pseudoword which is also a factor of $w$.
\end{Rmk}

A \emph{$\J$-maximal infinite $\J$-class} of $\Om AV$
is a $\J$-class consisting of $\J$-maximal infinite elements of $\Om AV$.

\begin{Thm}\label{t:bijection-minimal-shifts-maximal-j-classes}
  Let $\pv V$ be a pseudovariety of semigroups
  containing $\Lo{Sl}$.
  The correspondence $\Cl X\mapsto J_{\pv V}(\Cl X)$
  is a bijection from
  the set of minimal subshifts of $A^{\ZZ}$ to the set of
  $\J$-maximal infinite classes of $\Om AV$.
  Moreover,
  for every minimal subshift~$\Cl X$, the equalities
  $\Sha_{\pv V}(\Cl X)=L(\Cl X)\cup J_{\pv V}(\Cl X)=\Mir_{\pv V}(\Cl X)$  
  hold.
\end{Thm}

Theorem~\ref{t:bijection-minimal-shifts-maximal-j-classes}
is from~\cite{Almeida:2004a}.
Another proof, substantially different, is given in~\cite{Almeida&ACosta:2007a}.
The following related proposition will be used in Section~\ref{sec:relat-with-zeta}.

\begin{Prop}
   Let $\Cl X$ be a subshift of $A^{\ZZ}$.
   Consider a pseudovariety of semigroups containing $\Lo {Sl}$.
   The $\J$-maximal infinite elements
   of $\Om AV$ contained in $\Mir_{\pv V}(\Cl X)$ are
   the $\J$-maximal infinite elements of $\Om AV$ contained
   in $\Sha_{\pv V}(\Cl X)$.
 \end{Prop}

 \begin{proof}
   Since the factorial set $\Mir_{\pv V}(\Cl X)$
   contains infinite pseudowords,
   we may take some $\J$-maximal infinite element $w$
   of $\Om AV$ belonging to $\Mir_{\pv V}(\Cl X)$ (cf.~Remark~\ref{rmk:always-a-factor},).
   By Theorem~\ref{t:bijection-minimal-shifts-maximal-j-classes}, there is a
   minimal subshift $\Cl Y$
   such that $w\in\Mir_{\pv V}(\Cl Y)$ and all elements of $L(\Cl Y)$
   are finite factors of $w$.
   By the definition of $\Mir_{\pv V}(\Cl X)$, we then have
   $L(\Cl Y)\subseteq L(\Cl X)$,
   whence
   $\Sha_{\pv V}(\Cl Y)\subseteq\Sha_{\pv V}(\Cl X)$.
   Looking again at Theorem~\ref{t:bijection-minimal-shifts-maximal-j-classes},
   one sees that $\Mir_{\pv V}(\Cl Y)=\Sha_{\pv V}(\Cl Y)$.
   Therefore, we have $w\in \Sha_{\pv V}(\Cl X)$.
 \end{proof}

\section{Pseudoword block codes}
\label{sec:pseud-block-codes}

In this section we present a technique
emulating for pseudowords
the sliding block code process used for bi-infinite sequences. This will permit to build in Section~\ref{sec:funct-corr-from} the functors mentioned
in Section~\ref{sec:introduction}.
This technique was applied in~\cite{ACosta:2006}, explicitly for
free profinite semigroups over $\pv S$, implicitly for free profinite semigroups
over pseudovarieties $\pv V$ such that $\pv V=\pv V\ast\pv D$
and $\pv V\supseteq \Lo {Sl}$. In Theorem~\ref{t:block-coding-pseudovarieties}
we see that these pseudovarieties
give the exact scope of validity of this technique. While
the facts in Theorem~\ref{t:block-coding-pseudovarieties} are not original, they are dispersed in the literature and may not be easily accessible
(for example, that
all pseudovarieties $\pv V$ for which the technique holds satisfy
$\pv V=\pv V\ast\pv D$ is, as far as we know, only explicitly mentioned, \emph{en passant}, in the thesis~\cite{ACosta:2007t}, written in Portuguese).

\subsection{Word and pseudoword block codes}

We use the following convenient notation:
given a word $u$ of length $n\geq 1$, over the alphabet~$A$, if $u=a_1a_2\cdots a_n$, with $a_i\in A$ for each $i\in\{1,\ldots,n\}$,
we represent by $u_{[p,q]}$ the word $a_pa_{p+1}\cdots a_{q-1}a_q$,
whenever $1\leq p\leq q\leq n$.
If $1\leq k\leq n$, then we define $\be k(u)=u_{[1,k]}$
and $\te k(u)=u_{[n-k+1,n]}$,
that is, $\be k(u)$ and $\te k(u)$
are respectively the unique prefix and the unique suffix of~$u$
with length $k$.
If $k>n$, then we let $\be k(u)=u=\te k(u)$.
Moreover, for $k=0$, we make $\be 0(u)=\varepsilon=\te 0(u)$.

If $\pv V$ contains $\Lo I$, then the maps $u\mapsto \te k(u)$
and $u\mapsto \be k(u)$, with $u\in A^*$,
admit a unique continuous extension
to maps $\be k\colon \Om AV^I\to A^*$
and $\te k\colon \Om AV^I\to A^*$, respectively, where we consider the discrete topology on $A^*$~(take~\cite[Sections 3.7 and 5.2]{Almeida:1994a} as reference, with~\cite[Section 2.5]{Almeida&ACosta:2007a} as a possible helpful text). Hence, for every pseudoword
$u\in \Om AV\setminus A^+$, the word $\be k(u)$ (respectively, $\te k(u)$)
is the unique prefix (respectively, suffix)
of $u$ which is a word of length $k$.

Given a block map $\Psi\colon A^N\to B$, we are interested in the map $\overline{\Psi}\colon A^*\to B^*$
defined as follows: if $u$ is a word of $A^*$ of length at most $N-1$ then
$\overline{\Psi}(u)=1$, and if $u=a_1\cdots a_M$ is a word of length $M\geq N$,  with $a_i\in A$ for all $i\in\{1,\ldots,M\}$,
  then we have
  \begin{equation}\label{eq:definition-of-sliding-word-code}
    \overline{\Psi}(u)=
    \Psi(u_{[1,N]})\cdot \Psi(u_{[2,N+1]})
    \cdot \Psi(u_{[3,N+2]})\cdots\Psi(u_{[M-N+1,M]}).
  \end{equation}

      \begin{Example}\label{eg:word-code-versus-block-code}
        Let $\Psi$ be a central block map $A^{2k+1}\to B$.
        Consider the sliding block code
        $\psi\colon A^{\ZZ}\to B^{\ZZ}$ having
        $\Psi$ as a central block map.
        Let $x\in A^{\ZZ}$, and $y=\psi(x)$. Then, for all $i\in \ZZ$, we have
    \begin{equation*}
     \Psi(x_{[i-k,i+k]})=y_i
\end{equation*}
and so, applying formula~\eqref{eq:definition-of-sliding-word-code},
we obtain
\begin{equation*}
  \overline{\Psi}(x_{[i-k,j+k]})=y_{[i,j]}
\end{equation*}
whenever $i,j\in\ZZ$ are such that $i\leq j$.
  \end{Example}
  
  Intuitively, what $\overline{\Psi}$ does is to ``encode''
  the word $u$ into a new word $\overline{\Psi}(u)$ of $B^*$,
  by ``reading'' the consecutive
  factors on length $N$ and assigning the corresponding letters from~$B$.
  Loosely speaking, we are coding words as we code
  elements of a subshift via block maps, for which reason we say that
  $\overline{\Psi}$ is a \emph{word block code}.  
  Theorem~\ref{t:block-coding-pseudovarieties}
  below characterizes the pseudovarieties for which
  we can extend this process in the most natural way, to what we shall call~\emph{pseudoword block codes}.
 
  In preparation
  for Theorem~\ref{t:block-coding-pseudovarieties}, we introduce some notation.
  For each alphabet~$A$ and positive integer $N$, we denote by $A^{(<N)}$
  the set of elements of~$A^*$
  with length at most $N-1$.
  We will sometimes view the set $A^N$, of words of~$A^+$ with length~$N$,
  as an alphabet of its own. Viewed as an alphabet, $A^N$
  may be denoted~$A_N$, to facilitate the understanding of the context in which
  the elements of~$A^N$ are being~seen.
 
  For the special case where $\Psi$ is the identity map $A^N\to A_N$,
  we use the notation~$\Upsilon_N$ for the corresponding  word code $\overline{\Psi}$.
  In the literature (eg.~\cite{Almeida:1994a,Almeida&Klima:2020a,Pin&Weil:2002b}), the map $\Upsilon_N$ is sometimes denoted by $\Phi_{N-1}$ or $\sigma_{N-1}$.
  These two notations are somewhat
  unfortunate in the context of this paper,
  the latter because of the standard notation for the shift map, the former
  because it is also usual, in the symbolic dynamics literature,
  to use the letter $\Phi$ to denote arbitrary block maps (see eg.~\cite{Lind&Marcus:1996}).

  In this section we work with pseudovarieties
  $\pv V$ satisfying $\pv V=\pv V\ast\pv D$ and $\Lo I\subseteq \pv V$.
  After the next theorem, we deal with them
  using their characterization
  in the theorem, without needing the original definition in terms of semidirect~products.
  
  \begin{Thm}\label{t:block-coding-pseudovarieties}
    Let $\pv V$ be a pseudovariety of semigroups
    containing $\Lo I$. The following conditions are equivalent:
    \begin{enumerate}
    \item $\pv V=\pv V\ast\pv D$;\label{item:block-coding-pseudovarieties-1}
    \item for every alphabet $A$ and every positive integer $N$,
      the word block code $\Upsilon_N\colon A^*\to (A_N)^*$
      admits a unique extension to a continuous mapping
         $\Upsilon_N^{\pv V}\colon\Om AV^I\to\Om {A_N}V^I$;\label{item:block-coding-pseudovarieties-2}     
       \item for every alphabet $A$, positive integer $N$, and
         block map $\Psi\colon A^N\to B$,
      the word block code $\overline{\Psi}\colon A^*\to B^*$
      admits a unique extension
      to a continuous mapping $\overline{\Psi}_{\pv V}\colon\Om AV^I\to\Om {B}V^I$.\label{item:block-coding-pseudovarieties-3}
    \end{enumerate}
    Moreover, assuming the equivalent
    conditions~\eqref{item:block-coding-pseudovarieties-1}-\eqref{item:block-coding-pseudovarieties-3}, and denoting by $\lambda_{\pv V}$
    the unique continuous homomorphism $\Om {A_N}V\to \Om AV$
    such that $\lambda_{\pv V}(u)=\be 1(u)$ for each $u\in A^N$:
    \begin{enumerate}[resume]
    \item we have the equality
      \begin{equation}\label{eq:block-coding-pseudovarieties}
        \lambda_{\pv V}(\Upsilon_N^{\pv V}(uv))=u
      \end{equation}
      for every
      $u\in \Om AV$ and $v\in A^{N-1}$,
      so that in particular $\Upsilon_N^{\pv V}$ is injective
      on~$\Om AV\setminus A^{(<N)}$.\label{item:block-coding-pseudovarieties-4}
    \end{enumerate}
  \end{Thm}
 
  Theorem~\ref{t:block-coding-pseudovarieties}
  derives from~\cite[Chapter 10]{Almeida:1994a},
  and some parts are more or less explicitly stated there.
  In the paper~\cite{Almeida&Klima:2020a}
  and in the thesis~\cite{ACosta:2007t} (written in Portuguese)
  more details are given for other parts.
  The following proof is for
  the reader's convenience, so that the proof can be found in one location.

  \begin{proof}[Proof of Theorem~\ref{t:block-coding-pseudovarieties}]
    Throughout the proof, we refer to the pseudovariety $\pv D_k$
    of finite semigroups such that $St=\{t\}$ whenever $t\in S^k$,
    where $k$ is a positive integer. By convention, one has $\pv D_0=\pv I$.
    In fact, the equality $\pv D=\bigcup_{k\geq 1}\pv D_k$ holds (cf.~\cite[Sections 10.4 and 10.6]{Almeida:1994a}), and $\pv V\ast\pv D=\bigcup_{k\geq 1}\pv V\ast\pv D_k$. We proceed in several steps.

    \eqref{item:block-coding-pseudovarieties-1}
    $\Rightarrow$
    \eqref{item:block-coding-pseudovarieties-2}: 
     The validity of this implication when $\pv V=\pv S$
     is Lemma 10.6.11 from~\cite{Almeida:1994a}.

         For each alphabet $X$, denote by
    $p_X^{\pv V}$ the unique continuous
    homomorphism from $\Om XS$ onto the $X$-generated
    profinite semigroup $\Om XV$ that extends the identity on $X$. As before, we also use the notation
    $p_X^{\pv V}$ for the extension
    $\Om XS^I\to\Om XV^I$ mapping $I$ to $I$.
    In a somewhat different language, Theorem 10.6.12 from~\cite{Almeida:1994a} affirms in particular that if $\pv W$ is a pseudovariety
     strictly containing $\Lo I$
     and such that $\pv W=\pv W\ast\pv D_{N-1}$,
     then one has
     \begin{equation}\label{eq:equivalence-to-commute}
       p_A^{\pv W\ast\pv D_{N-1}}(u)=p_A^{\pv W\ast\pv D_{N-1}}(v)
       \Leftrightarrow
       \begin{cases}
       p_{A^N}^{\pv W}(\Upsilon_N^{\pv S}(u))=
       p_{A^N}^{\pv W}(\Upsilon_N^{\pv S}(v))\\
       \be {N-1}(u)=\be {N-1}(v)\\
       \te {N-1}(u)=\te {N-1}(v)
       \end{cases}
     \end{equation}
     for all $u,v\in\Om AS$.
     But this equivalence is also valid when $\pv V=\Lo I$,
     because, in what is a well-know property of pseudowords
     (see \cite[Section 2.3]{Costa:2001a} for example\footnote{We give \cite[Section 2.3]{Costa:2001a} as a reference for this property of $\Lo I$
       for the sake of better readability, but the property was known before: in the language of pseudowords, it is implicit in~\cite[Section 3.7]{Almeida:1994a}, and in fact it amounts to the fact that $\Lo I$ is the join of $\pv D$ and its dual $\pv K$, a fact already appearing in~\cite{Eilenberg:1976}.}), when $u,v\in\Om AV$ one has $p_A^{\Lo I}(u)=p_A^{\Lo I}(v)$
     if and only if
     $i_k(u)=i_k(v)$
     and
     \mbox{$t_k(u)=t_k(v)$}
     for every positive integer $k$ (in particular,
     either $u=v$ or $u,v$ are infinite pseudowords
     with the same finite prefixes and the same finite suffixes).
     So, we may in fact suppose that $\pv V\supseteq \Lo I$.
     Taking $\pv W=\pv V$,
     and
     since $\pv V=\pv V\ast\pv D=\pv V\ast\pv D\ast\pv D_{N-1}=\pv V\ast\pv D_{N-1}$ (as $\pv D=\pv D*\pv D_{N-1}$, cf.~\cite[Sections 10.4 and 10.6]{Almeida:1994a}), we obtain the implication
     \begin{equation*}
       p_A^{\pv V}(u)=p_A^{\pv V}(v)
     \Rightarrow
     p_{A^N}^{\pv V}(\Upsilon_N^{\pv S}(u))=
       p_{A^N}^{\pv V}(\Upsilon_N^{\pv S}(v))
     \end{equation*}
     and so we may define a (unique) map
     $\Upsilon_N^{\pv V}\colon\Om AV\to \Om {A^N}V^I$
     for which the diagram
               \begin{equation}\label{eq:coding-at-different-levels}
         \begin{split}  
         \xymatrix@C=2cm{
         \Om AS\ar[r]^{\Upsilon_N^{\pv S}}
         \ar[d]_{p_A^{\pv V}}
         &\Om {A_N}S^I
         \ar[d]^{p_{A_N}^{\pv V}}
         \\
         \Om AV\ar[r]^{\Upsilon_N^{\pv V}}&
         \Om {A_N}V^I
         \\
       }
       \end{split}
     \end{equation}
     commutes. Finally, because the other maps in the diagram
     are continuous maps between compact spaces,
     one sees that $\Upsilon_N^{\pv V}$ is also continuous\footnote{The
       arguments used in the proof of this implication are basically
       the same that were used in the proof of~\cite[Lemma 2.2]{Almeida&Klima:2020a}, but there one finds the assumption that $\pv V$
       contains $\pv{Sl}$ to guarantee that $\pv V$ does contain nontrivial monoids and
       therefore is according to the statement in~\cite[Theorem 10.6.12]{Almeida:1994a}. As seen in our recapitulation of those arguments, such assumption
       is unnecessary.}.

    \eqref{item:block-coding-pseudovarieties-2}
    $\Rightarrow$
    \eqref{item:block-coding-pseudovarieties-4}:
    Consider words $v,w\in A^+$ with length $N-1$,
    and letters $a,b\in A$ such that $av=wb$.
    By definition, we have $\lambda_{\pv V}(\Upsilon_N^{\pv V}(av))=a$.
    This provides the base step for the following inductive
    argument to show the equality~\eqref{eq:block-coding-pseudovarieties} for words,
    inducting on the length of words.
    If $z\in A^+$ has length $M\geq N$, then, according
    to formula~\eqref{eq:definition-of-sliding-word-code} applied to the case where $\overline{\Psi}$ acts in $A^N$ as the identity, we have
    \begin{equation}\label{eq:detail}
      \Upsilon_N^{\pv V}(z)
      =\Upsilon_N^{\pv V}(z_{[1,M-1]})\cdot\Upsilon_N^{\pv V}(z_{[M-N+1,M]}).
    \end{equation}
    Therefore,
    for every $u\in A^+$, by putting $z=uav=uwb$ in~\eqref{eq:detail}, one has the equality  $\Upsilon_N^{\pv V}(uav)=\Upsilon_N^{\pv V}(uw)\cdot \Upsilon_N^{\pv V}(av)$, so that    
    \begin{equation*}
      \lambda_{\pv V}(\Upsilon_N^{\pv V}(uav))
      =\lambda_{\pv V}(\Upsilon_N^{\pv V}(uw))\cdot \lambda_{\pv V}(\Upsilon_N^{\pv V}(av))
      =\lambda_{\pv V}(\Upsilon_N^{\pv V}(uw))\cdot a=ua,
    \end{equation*}    
    where in the last equality we use the induction hypothesis.
    Since $A^+$ is dense in $\Om AV$,
    and $\lambda_{\pv V}\circ \Upsilon_N^{\pv V}$
    is continuous on $A^+\setminus A^{(<N)}$,
    we immediately extend the scope of
    equality \eqref{eq:block-coding-pseudovarieties}
    to every   $u\in \Om AV$ and $v\in A^{N-1}$.

    Let $w\in\Om AV\setminus A^{(<N)}$.
    Then $w=uv$, for some $u\in\Om AV^I$ and $v\in A^N$.
    Let $(w_{k})_k$ be a sequence of words of length at least $N$
    converging to $w$.
    By the continuity of $\te N$,
    we have $v=\te N(w_k)=\te N(w)$ for all sufficiently large $k$.
    If we see $v$ as letter of $A_N$,
    then, by the continuity of $\Upsilon_N^{\pv V}$,
    for all sufficiently large $k$ we have
    $v=\te 1(\Upsilon_N^{\pv V}(w_k))=\te 1(\Upsilon_N^{\pv V}(w))$.
    Therefore, if $w$ and $w'$ are elements of
    $\Om AV\setminus A^{(<N)}$ such that
    $\Upsilon_N^{\pv V}(w)=\Upsilon_N^{\pv V}(w')$,
    then $\te N(w)$ and $\te N(w')$ are the same word $v$.
    In particular, $\te {N-1}(w)$ and $\te {N-1}(w')$ are
    both equal to the word $\tilde v=\te {N-1}(v)$.
    On the other hand, we have factorizations
    $w=u\tilde v$ and $w'=u'\tilde v$, for some $u,u'\in\Om AV$.
    Since $u=\lambda_\pv V(\Upsilon_N^{\pv V}(u\tilde v))
    =\lambda_\pv V(\Upsilon_N^{\pv V}(u'\tilde v))=u'$,
    we conclude that $w=w'$ and that $\Upsilon_N^{\pv V}$ is injective.
    
    \eqref{item:block-coding-pseudovarieties-2}
    $\Rightarrow$
    \eqref{item:block-coding-pseudovarieties-3}:
    Let $\widehat{\Psi}$ be the unique continuous
    homomorphism $\Om {A_N}V\to \Om BV$
    such that $\widehat{\Psi}(u)=\Psi(u)$
    for every $u\in A_N$.
    We  also work with the extension
    $\Om {A_N}V^I\to \Om BV^I$,
    still denoted $\widehat{\Psi}$, mapping $I$ to $I$.    
    By the hypothesis that \eqref{item:block-coding-pseudovarieties-2}
    holds, the composition $\overline{\Psi}_{\pv V}=\widehat\Psi_{\pv V}\circ \Upsilon_N^{\pv V}$ is a continuous mapping from $\Om AV^I$ into $\Om BV^I$.
    One sees straightforwardly 
    by induction on the length of $u\in A^+$
    that $\overline{\Psi}_{\pv V}(u)=\overline{\Psi}(u)$
    for every $u\in A^+$.
    Since $A^*$ in dense in $\Om AV^I$,
    the mapping $\overline{\Psi}_{\pv V}$
    is the unique continuous extension of
    $\overline{\Psi}\colon A^*\to B^*$
    to a mapping $\Om AV^I\to\Om {B}V^I$.
        
    Observing that \eqref{item:block-coding-pseudovarieties-3}
    $\Rightarrow$
    (\ref{item:block-coding-pseudovarieties-2}) is trivial,
    it remains to check \eqref{item:block-coding-pseudovarieties-2}
    $\Rightarrow$
    \eqref{item:block-coding-pseudovarieties-1}.
    We shall
    use the following facts,
    valid for all pseudovarieties $\pv V$ and $\pv W$:
    \begin{itemize}
    \item  if $\pv W\supseteq \pv V$, then
    the kernel of $\pv W$ is contained in the kernel of $\pv V$ (this is because
    $\Om AV$ is a pro-$\pv W$ semigroup when $\pv V\subseteq \pv W$);
   \item     $p_X^{\pv W}$ and $p_X^{\pv V}$ have the same kernel if and only if
    $\pv V=\pv W$   (this is just a reformulation of the fact that $\pv V$ and $\pv W$ are equal if and only if they satisfy the same ``pseudoidentities'', see for example~\cite{Almeida:2003cshort} for details if necessary). 
    \end{itemize}
    Observe Diagram~\eqref{eq:coding-at-different-levels},
    which, under our assumption that \eqref{item:block-coding-pseudovarieties-2} holds, is commutative:
    indeed, the restrictions to $A^+$ of  the continuous mappings
    $p_{A_N}^{\pv V}\circ\Upsilon_N^{\pv S}$
    and $\Upsilon_N^{\pv V}\circ p_A^{\pv V}$
    clearly coincide, and $A^+$ is dense in $\Om AV$.
     As already seen in the proof
     of the implication
     \eqref{item:block-coding-pseudovarieties-1}
    $\Rightarrow$
    \eqref{item:block-coding-pseudovarieties-2},
    if $\pv W$ is a pseudovariety
    containing $\Lo I$, then the
    equivalence \eqref{eq:equivalence-to-commute}
     holds for all $u,v\in\Om AV$.
    Taking $\pv W=\pv V$,
    and using the commutativity
    of Diagram~\eqref{eq:coding-at-different-levels},
    we then get
    the equivalence
    \begin{equation}\label{eq:towards-isomorphism}
      p_A^{\pv V\ast\pv D_{N-1}}(u)=p_A^{\pv V\ast\pv D_{N-1}}(v)
      \Leftrightarrow
      \begin{cases}
      \Upsilon_{N}^{\pv V}[p_{A}^{\pv V}(u)]
      =
      \Upsilon_{N}^{\pv V}[p_{A}^{\pv V}(v)]\\
      \be{N-1}(u)=\be{N-1}(v)\\
      \be{N-1}(u)=\be{N-1}(v)
      \end{cases}
    \end{equation}
    for all $u,v\in\Om AV$.
    We claim that in fact we have
        \begin{equation*}
      p_A^{\pv V\ast\pv D_{N-1}}(u)=p_A^{\pv V\ast\pv D_{N-1}}(v)
      \Leftrightarrow
        p_{A}^{\pv V}(u) = p_{A}^{\pv V}(v).
      \end{equation*}
      Since $\pv V*\pv D_{N-1}$ contains $\pv V$,
      the direct implication is immediate. Conversely,
      suppose that $u,v\in\Om AV$ are such that $p_{A}^{\pv V}(u) = p_{A}^{\pv V}(v)$.
     Since $\pv V$ contains $\Lo I$,
     this implies $p_{A}^{\Lo I}(u) = p_{A}^{\Lo I}(v)$,
     which is the same as having $\be k(u)=\be k(v)$
     and $\te k(u)=\te k(v)$
     for every positive integer $k$.
     It then follows from~\eqref{eq:towards-isomorphism}
     that $p_A^{\pv V\ast\pv D_{N-1}}(u)=p_A^{\pv V\ast\pv D_{N-1}}(v)$.
    Therefore, $p_A^\pv V$ and $p_A^{\pv V\ast\pv D_{N-1}}$
    have the same kernel, whence $\pv V=\pv V\ast\pv D_{N-1}$.
    As this is true for all $N\geq 1$, we conclude that $\pv V=\pv V\ast\pv D$.
  \end{proof}

  Given a block map $\Psi\colon A^N\to B$,
  we say that the map $\overline{\Psi}_{\pv V}\colon \Om AV^I\to\Om BV^I$,
  introduced in Theorem~\ref{t:block-coding-pseudovarieties},
  is a \emph{pseudoword block code}. 

  The next corollary is in~\cite{Almeida&Klima:2020a}, with
  the additional hypothesis that $\pv V$ is monoidal.
  
  \begin{Cor}\label{c:cancelation-property}
    Let $\pv V$
    be a pseudovariety of semigroups
    such that $\pv V=\pv V\ast\pv D$ and $\Lo I\subseteq\pv V$.
    If $u$, $v$ are words of $A^*$ with the same length, and
    $\pi,\rho\in\Om AV$
    are such that $\pi u=\rho v$ or that $u\pi =v\rho$, then
    $u=v$ and $\pi=\rho$.
  \end{Cor}
 
  \begin{proof}
    Let $n$ be the length of $u$ and $v$.
    If $\pi u=\rho v$,
    then $u=\te n(\pi u)=\te n(\rho v)=v$,
    and the equality $\pi=\rho$ follows from
    equality~\eqref{eq:block-coding-pseudovarieties} in Theorem~\ref{t:block-coding-pseudovarieties} (with~\mbox{$N=n+1$}).
    The case $u\pi=v\rho$ is treated similarly, using the dual
    of  equality~\eqref{eq:block-coding-pseudovarieties}.
  \end{proof}

      If we are in the conditions of Corollary~\ref{c:cancelation-property},
      then, for each positive integer $N$,
      and pseudoword $u\in \Om AV\setminus A^+$,
    we denote by $\be N(u)^{-1}\cdot u$
    the unique pseudoword
    $u'$ such that $u=\be N(u)\cdot u'$.
    Similarly, we denote by $u\cdot \te N(u)^{-1}$
    the unique pseudoword
    $u''$ such that $u=u''\cdot\te N(u)$.

  \begin{Lemma}\label{l:continuity-of-cancelation}
    The maps $u\mapsto \be N(u)^{-1}\cdot u$
    and $u\mapsto  u\cdot \te N(u)^{-1}$
    are continuous on the space $\Om AV\setminus A^+$.
  \end{Lemma}

  \begin{proof}
    Suppose that $(u_n)_n$ converges to $u$ in $\Om AV\setminus A^+$.
    As $\be N$ is continuous,
    we have $\be N(u_n)=\be N(u)$ for all large enough $n$.
    Therefore, we have $u_n=\be N(u)v_n$ for all large enough $n$,
    where $v_n=\be N(u_n)^{-1}\cdot u_n$. Every accumulation
    point $v$ of $(v_n)_n$ is such that $u=\be N(u)v$,
    that is, $v=\be N(u)u^{-1}$.
    Since we are dealing with a compact space, this means
    that $\be N(u_n)^{-1}\cdot u_n\to \be N(u)^{-1}\cdot u$.
    Similarly, we have $u_n\cdot \te N(u_n)^{-1}\to u\cdot \te N(u)^{-1}$.
  \end{proof}

\subsection{Some properties of pseudoword block codes}

We now introduce some useful properties of pseudoword block codes.
Until the end of this section, we work with a fixed pseudovariety
of semigroups such that $\pv V=\pv V\ast\pv D$ and $\Lo I\subseteq\pv V$.
In the absence
  of confusion, we may denote a pseudoword block code 
  $\overline{\Psi}_{\pv V}$, from $\Om AV^I$ to $\Om BV^I$,
  simply by $\overline{\Psi}$, dropping the subscript $\pv V$.

\begin{Lemma}\label{l:composition-of-block-maps}
  Consider a pseudovariety of semigroups $\pv V$
  such that $\pv V=\pv V\ast\pv D$ and $\Lo I\subseteq\pv V$.    
  Let $\varphi\colon\Cl X\to\Cl Y$
  be a morphism of subshifts with central block map
  $\Phi\colon A^{2k+1}\to B$.
  Take a morphism of subshifts $\psi\colon\Cl Y\to\Cl Z$
  with central block map $\Psi\colon B^{2l+1}\to C$.
  Then the map $\Lambda\colon A^{2k+2l+1}\to C$
  defined by $\Lambda(u)=\Psi\circ\overline{\Phi}(u)$
  is a central block map for $\psi\circ\varphi$.
  Moreover, the equality
  \begin{equation}\label{eq:composition-of-block-maps}
    \overline{\Lambda}(u)=\overline{\Psi}\circ\overline{\Phi}(u)
  \end{equation}
  holds for every $u\in\Om AV$.
\end{Lemma}

\begin{proof}
  Let $x\in\Cl X$, $y=\varphi(x)$ and $z=\psi(y)$.
  Note that (cf.~Example~\ref{eg:word-code-versus-block-code})
  for each $i\in \ZZ$, we have
  \begin{equation*}
    z_i=\Psi([y_{[i-l,i+l]}])=\Psi\Biggl(\,\prod_{j\in [i-l,i+l]}\Phi(x_{[j-k,j+k]})\,\Biggr)=\Psi(\overline{\Phi}(x_{[i-l-k,i+l+k]})),
  \end{equation*}
  and so $\Lambda$ is indeed a central block map for $\psi\circ\varphi$.

  We may in particular suppose that $\Cl X=A^{\ZZ}$, $\Cl Y=B^{\ZZ}$
  and $\Cl Z=C^{\ZZ}$.
  Let $u$ be a word of $A^+$ of length
  $n\geq 2k+2l+1$. Then $u=x_{[1,n]}$ for some $x\in \Cl X$.
  As we already checked that
  $\Lambda$ is a central block map for $\varphi\circ\psi$,
  we know that, for $y=\varphi(x)$ and $z=\psi(y)$,
  the following chain of equalities holds:
  \begin{equation*}
    \overline{\Lambda}(x_{[1,n]})
    =z_{[1+k+l,n-k-l]}=\overline{\Psi}(y_{[1+k,n-k]})
    =\overline{\Psi}(\overline{\Phi}(x_{[1,n]})).
  \end{equation*}
  Hence, equality~\eqref{eq:composition-of-block-maps} holds
  for every $u\in A^+$ of length at least $2k+2l+1$.
  It also holds if $u$ has  smaller length: in that case,
  we have $\overline{\Lambda}(u)=\varepsilon=\overline{\Psi}(\overline{\Phi}(u))$
  (for the latter equality, note that the length of $\overline{\Phi}(u)$
  will be smaller than $2l+1$).
  As $\overline{\Lambda}$, $\overline{\Psi}$
  and $\overline{\Phi}$ are continuous in $\Om AV^I=\overline{A^*}$,
  it follows that~\eqref{eq:composition-of-block-maps}
  holds for every $u\in\Om AV^I$.
\end{proof}

\begin{Prop}\label{p:it-is-not-a-homomorphism-but-has-nice-properties}
  Take a pseudovariety of semigroups $\pv V$
  such that $\pv V=\pv V\ast\pv D$ and $\Lo I\subseteq\pv V$.
  Consider a block map $\Psi\colon A^N\to B$.
    For all pseudowords~$u$ and~$v$
    of~$\Om AV$, we have
    \begin{equation}\label{eq:it-is-not-a-homomorphism-but-has-nice-properties-1-0}
      \overline{\Psi}(uv)=
      \overline{\Psi}(u\,\be {N-1}(v))
      \cdot \overline{\Psi}(v)
      =\overline{\Psi}(u)\cdot \overline{\Psi}(\te {N-1}(u)v).
    \end{equation}
    If, moreover, $N=2k+1$, then
    \begin{equation}\label{eq:it-is-not-a-homomorphism-but-has-nice-properties-1}
      \overline{\Psi}(uv)=
      \overline{\Psi}(u\,\be{k}(v))
      \cdot
      \overline{\Phi}(\te k(u)\,v)
    \end{equation}
    holds.
  \end{Prop}

  For $\overline{\Psi}=\Upsilon_N$,
  the property in~\eqref{eq:it-is-not-a-homomorphism-but-has-nice-properties-1-0} is~\cite[Exercise 10.6.6]{Almeida:1994a}.
  In its entirety,
  Proposition~\ref{p:it-is-not-a-homomorphism-but-has-nice-properties}
  is proved in the thesis~\cite{ACosta:2007t}.
  For the reader's convenience, a short proof is given here,
  which seems more transparent than that in~\cite{ACosta:2007t}.

   \begin{proof}    
     By the continuity of $\overline{\Psi}$, and since $A^+$
     is dense in $\Om AV$, it suffices to check~\eqref{eq:it-is-not-a-homomorphism-but-has-nice-properties-1-0}  and~\eqref{eq:it-is-not-a-homomorphism-but-has-nice-properties-1} for elements of $A^+$.
     We only do it for~\eqref{eq:it-is-not-a-homomorphism-but-has-nice-properties-1}, as~\eqref{eq:it-is-not-a-homomorphism-but-has-nice-properties-1-0} may be treated similarly.
     Let $u,v\in A^+$ be words of lengths
     $n$ and $m$, respectively.
     If $n\leq k$, then $\te k(u)=u$
     and $|u\,\be k(v)|<2k+1$,
     whence $\overline{\Psi}(u\,\be k(v))=\varepsilon$
     and the equality holds trivially.
     Similarly for the case $m\leq k$.
     Finally, suppose that $n,m>k$.
      Consider the sliding block
     code $\psi\colon A^{\ZZ}\to B^{\ZZ}$ having $\Psi$
     as a central block map. 
     Then, we may choose some $x\in A^{\ZZ}$
     such that $u=x_{[1,n]}$ and $v=x_{[n+1,n+m]}$.
     Let $y=\psi(x)$.
     Since $n+m>2k+1$, we have
     \begin{align*}
       \overline{\Psi}(uv)
       =\overline{\Psi}(x_{[1,n+m]})=y_{[1+k,n+m-k]}&=y_{[1+k,n]}\cdot y_{[n+1,n+m-k]}\\
       &=\overline{\Psi}(x_{[1,n+k]})\cdot\overline{\Psi}(x_{[n+1-k,n+m]})\\
       &=\overline{\Psi}(u\,\be k(v))
       \cdot
       \overline{\Psi}(\te k(u)\,v),
     \end{align*}
establishing the equality~\eqref{eq:it-is-not-a-homomorphism-but-has-nice-properties-1}.
\end{proof}

Pseudoword block codes behave well with
respect to the sets $\overline{L(\Cl X)}$
and $\Mir_{\pv V}(\Cl X)$, in the sense of the next proposition.
Note the assumption that $\pv V$ contains $\Lo{Sl}$ is
necessary because in the proof we need to guarantee that
$\Mir_{\pv V,k}(\Cl X)$ is clopen, for every positive integer~$k$.
Recall that, under the hypothesis $\pv V=\pv V\ast\pv D$,
the inclusion $\Lo{Sl}\subseteq\pv V$ is equivalent to $\pv {Sl}\subseteq\pv V$,
since $\Lo {Sl}=\Lo{Sl}*\pv D$.

\begin{Prop}\label{p:image-of-phi}
Consider a pseudovariety of semigroups $\pv V$
  such that $\pv V=\pv V\ast\pv D$ and $\Lo {Sl}\subseteq\pv V$.
       Let $\varphi\colon\Cl X\to\Cl Y$ be a sliding block code of subshifts,
       with central block map $\Phi$. 
       The inclusions
       $\overline{\Phi}\Bigl(\overline{L(\Cl X)}\Bigr) \subseteq
       \overline{L(\Cl Y)} \cup\{\varepsilon\}$, $\overline{\Phi}\bigl(\Mir_{\pv V}(\Cl X)\bigr)\subseteq\Mir_{\pv V}(\Cl Y)\cup\{\varepsilon\}$ and       $\overline{\Phi}\bigl(\Sha_{\pv V}(\Cl X)\bigr) \subseteq
     \Sha_{\pv V}(\Cl Y)\cup\{\varepsilon\}$, hold.
   \end{Prop}

   We omit the routine proof of Proposition~\ref{p:image-of-phi},
   appearing in~\cite[Lemma~3.2]{ACosta:2006} under the assumption $\pv V=\pv S$, irrelevant for the proof given there. Proposition~\ref{p:image-of-phi}
   is also proved in the thesis~\cite{ACosta:2007t} (with exactly the same hypothesis as here). We just underline that the last of the three inclusions
   is a direct consequence of the first inclusion
   and of the implication $u\leq_{\J}v\Rightarrow\overline{\Phi}(u)\leq_{\J} \overline{\Phi}(v)$, justified by Proposition~\ref{p:it-is-not-a-homomorphism-but-has-nice-properties}. More precisely, if $\Phi$ has window size $N$,
   and $u=xvy$, then
   \begin{equation*}
     \overline{\Phi}(u)
     =\overline{\Phi}(x\,\be {N-1}(vy))\cdot\overline{\Phi}(vy)=
     \overline{\Phi}(x\,\be {N-1}(vy))\cdot\overline{\Phi}(v)\cdot\overline{\Phi}(\te {N-1}(v)y)\leq_{\J}\Phi(v).
   \end{equation*}

\section{A functorial correspondence from subshifts to categories}
\label{sec:funct-corr-from}   

By a \emph{compact category} we mean a small category
$C$ such that:
\begin{enumerate}
\item the set $\Obj(C)$ of objects of $C$ and the set $\Mor(C)$ of arrows (i.e., morphisms) of $C$
  are both compact topological spaces;
\item both incidence maps, respectively assigning the domain
  $d(e)=x$ and the co-domain $r(e)=y$ to each arrow $e\colon x\to y$,
  are continuous maps from the space of arrows to the space of objects;
\item the map $x\mapsto 1_x$ is continuous, where $1_x$ denotes the identity at $x$;
\item the map $(s,t)\mapsto st$ defined on the set of composable
arrows is continuous.
\end{enumerate}
The morphisms between compact categories are the functors
that restrict to continuous mappings between the corresponding
spaces of objects and arrows.

For each semigroup $S$, we denote by $E(S)$ the
set of idempotents of $S$.
 The \emph{Karoubi envelope} of a semigroup $S$ is a small category $\Kar(S)$, whose objects are the idempotents of $S$,
 and whose morphisms $f\to e$ are triples $(e,s,f)\in E(S)\times S\times E(S)$
 such that $s=esf$, with composition $(e,s,f)(f,t,g)=(e,st,g)$.
 The identity morphism $1_e$ at object $e$ is the triple $(e,e,e)$.
 If $S$ is a compact semigroup, then $E(S)$ is a nonempty compact
 subspace~\cite[Theorem 3.5]{Carruth&Hildebrant&Koch:1983},
 and $\Kar(S)$ becomes a compact category, if we consider the space of morphisms endowed with the topology
 induced from the product space $E(S)\times S\times E(S)$.

 \begin{Rmk}
   Each continuous homomorphism $\varphi\colon T\to R$
   of compact semigroups
   induces a continuous functor
   $\Kar(\varphi)\colon\Kar(T)\to\Kar(R)$,
   with $\Kar(\varphi)(e)=\varphi(e)$
   when $e\in E(T)$,
   and $\Kar(\varphi)(e,s,f)=(\varphi(e),\varphi(s),\varphi(f))$
   when $(e,s,f)$ is an arrow of~$\Kar(T)$.
   Moreover,
   an inverse limit $S=\varprojlim S_i$
   of finite semigroups induces the equality $\Kar(S)=\varprojlim\Kar(S_i)$.
      A compact category is \emph{profinite}
   when it is the inverse limit of an inverse system of
   finite categories. Hence, the Karoubi envelope of a profinite semigroup is a
   profinite category. We shall not need this fact, but one should have it in mind, as we will work with Karoubi envelopes of (free) profinite semigroups.
 \end{Rmk}
 
 For later reference, we collect a couple of simple
 facts about the Karoubi envelope of a semigroup. 
 For each idempotent $e$ of a semigroup $S$,
 let $G_e$ be the $\H$-class of~$e$, that is, $G_e$ is the group
 of units of the monoid $eSe$.
 Recall that $G_e$ is a compact/profinite group if $S$ is compact/profinite.
 
\begin{Prop}\label{p:local-isomorphism}
  If $e$ is an idempotent of the compact semigroup $S$,
  then the group of
  automorphisms of $e$ in $\Kar(S)$ is a compact group isomorphic to $G_e$.
\end{Prop}

\begin{proof}
  The map $(e,s,e)\mapsto s$ is an isomorphism
  from the group of automorphisms of $e$ in $\Kar(S)$
  onto $G_e$ (see for example~\cite{ACosta&Steinberg:2015}).
  This map is clearly continuous.
\end{proof}

In a category, an object $c$ is a retract of an object $d$, denoted
$c\prec d$,
if there are arrows $\varphi\colon c\to d$ and $\psi\colon d\to c$
with $\psi\circ\varphi=1_c$. The relation $\prec$ is a partial order.

\begin{Prop}\label{p:retraction}
  Let $e,f$ be idempotents of the semigroup $S$.
  Then~$e\leq_\J f$ if and only if~$e\prec f$.
\end{Prop}

\begin{proof}
  If $e=xfy$, with $x,y\in S^I$, then $1_e=(e,exf,f)(f,fye,e)$,
  establishing the ``only if'' part.
  Conversely, if $1_e=(e,s,f)(f,t,e)$, then $e=st=sft$.
\end{proof}

If $F$ is a closed factorial subset of $S$, we denote by $\Kar(F)$
the subgraph of $\Kar(S)$ whose edges are the morphisms
$(e,s,t)$ such that $s\in F$, and whose objects are the idempotents
of $S$ belonging to $F$. The graph $\Kar(F)$ may not be a subcategory.

\begin{Example}
  Let $\Cl X$ be the even subshift from Example~\ref{eg:even-subshift}.
  Then $s=(a^\omega,a^\omega b^\omega,b^\omega)$
  and $t=(b^\omega,b^{\omega+1}a^\omega,a^\omega)$
  belong to $\Kar(\Sha_{\pv V}(\Cl X))$,
  but not $st=(a^\omega,a^\omega b^{\omega+1}a^\omega,a^\omega)$.
  Indeed, $a^\omega b^{\omega+1}a^\omega$ belongs
  to $K=\overline{a^+(b^2)^*ba^+}$,
  and so, since $K\cap L(\Cl X)=\emptyset$ and $K$ is open,
  one has $a^\omega b^{\omega+1}a^\omega\notin\overline{L(\Cl X)}$.
\end{Example}

Let $\Cl X$ be a subshift of $A^{\ZZ}$
and let $\pv V$ be a pseudovariety containing $\Lo {Sl}$.
Suppose that $(e,u,f)$ and $(f,v,g)$ are arrows of $\Kar(\Om AV)$
with $u,v\in\Mir_{\pv V}(\Cl X)$.
Then, we have $uv=ufv\in \Mir_{\pv V}(\Cl X)$: indeed,
in what has some similarity with properties of ordinary words,
a finite factor in a product $w_1\cdots w_n$ of infinite pseudowords over $\pv V\supseteq\Lo {Sl}$ is either a factor of
some $w_i$, or a product of a suffix of $w_i$
and a prefix of $w_{i+1}$, for some $i$ (see~\cite[Lemma $8.2$]{Almeida&Volkov:2006} for a formal statement and proof), so that a finite factor of $ufv$ is either a factor of $uf=u$ or of $v=fv$.
Hence $\Kar(\Mir_{\pv V}(\Cl X))$ is a compact subcategory of~$\Kar(\Om AV)$.

In this section, $\pv V$ is always a pseudovariety of semigroups
containing $\Lo {Sl}$ and such that $\pv V=\pv V\ast\pv D$.

  \begin{Lemma}\label{l:coding-of-an-idempotent}
    Consider a central block map~$\Phi\colon A^{2k+1}\to B$.
    For every idempotent $e$ of $\Om AV$, the
    pseudoword $\overline{\Phi}(\te k(e)\cdot e\cdot \be k(e))$ is an idempotent of $\Om BV$. 
  \end{Lemma}

  \begin{proof}
    Applying the property expressed in equality~\eqref{eq:it-is-not-a-homomorphism-but-has-nice-properties-1},
    from Proposition~\ref{p:it-is-not-a-homomorphism-but-has-nice-properties},
    to the pseudowords $u=\te k(e)\cdot e$
    and $v=e\cdot \be k(e)$,
    we obtain
    \begin{equation*}
      \overline{\Phi}(\te k(e)\cdot e\cdot \be k(e))
      \cdot
      \overline{\Phi}(\te k(e)\cdot e\cdot \be k(e))
      =
      \overline{\Phi}(\te k(e)\cdot e\cdot e\cdot \be k(e))
      =
      \overline{\Phi}(\te k(e)\cdot e\cdot \be k(e)),
    \end{equation*}
    showing that $\overline{\Phi}(\te k(e)\cdot e\cdot \be k(e))$ is an idempotent.
  \end{proof}

  In the setting of Lemma~\ref{l:coding-of-an-idempotent},
  we denote the idempotent $\Phi(\te k(e)\cdot e\cdot \be k(e))$
  by~$\Phi_\Kar(e)$.

  \begin{Prop}\label{p:functor-defined-by-block-map}
     Consider a central block map~$\Phi\colon A^{2k+1}\to B$.
    The following data defines
  a functor:
  \begin{equation*}
       \Phi_{\Kar_{\pv V}}\colon\begin{array}[t]{rcl}
   \Kar(\Om AV)&\to&\Kar(\Om BV)\vspace{0.2 cm}\\
   e&\mapsto&
   \Phi_{\Kar_{\pv V}}(e)\vspace{0.1 cm}\\
   (e,u,f)
   &\mapsto
   &\Bigl(\Phi_{\Kar_{\pv V}}(e),\overline{\Phi}(\te k(e)\,u\,\be k(f)),
   \Phi_{\Kar_{\pv V}}(f)\Bigr)
  \end{array}
\end{equation*}
  \end{Prop}

  \begin{proof}
    Thanks to Lemma~\ref{l:coding-of-an-idempotent},
    we already know that this correspondence is correctly defined on objects.
    Since $u=euf$, we have $\be k(u)=\be k(e)$,
    $\te k(u)=\te k(f)$, and therefore, applying~\eqref{eq:it-is-not-a-homomorphism-but-has-nice-properties-1}, we get the following chain of equalities:
    \begin{align*}
      \overline{\Phi}(\te k(e)\,u\,\be k(f))&=
      \overline{\Phi}(\te k(e)\,e\cdot u\,\be k(f))\\
      &=
      \overline{\Phi}(\te k(e)\,e\,\be k(e))\cdot
      \overline{\Phi}(\te k(e)\,u\,\be k(f))
      \\
      &=
      \overline{\Phi}(\te k(e)\,e\,\be k(e))\cdot
      \overline{\Phi}(\te k(e)\,u\cdot f\,\be k(f))
      \\
      &=
      \overline{\Phi}(\te k(e)\,e\,\be k(e))\cdot
      \overline{\Phi}(\te k(e)\,u\,\be k(f))\cdot
      \overline{\Phi}(\te k(f)\,f\,\be k(f)).
    \end{align*}
    This shows that $\Phi_{\Kar_{\pv V}}$ is a morphism of graphs.
    Similarly, if $(e,u,f)$ and $(f,v,g)$
    are two composable arrows of $\Kar(\Om AV)$,
    by applying again~\eqref{eq:it-is-not-a-homomorphism-but-has-nice-properties-1} we get
    \begin{equation*}
      \overline{\Phi}(\te k(e)\,u\,\be k(f))
      \cdot
      \overline{\Phi}(\te k(f)\,v\,\be k(g))
      =
      \overline{\Phi}(\te k(e)\,uv\,\be k(g)),
    \end{equation*}
    thus establishing that $\Phi_{\Kar_{\pv V}}$ is a functor.
  \end{proof}
  
  For $u\in\Om AV$ and idempotents $e,f$
  with $u=euf$, one has $u\mathrel{\J}\te k(e)\,u\,\be k(f)$,
  thus $u\in \Mir_{\pv V}(\Cl X)$
  if and only if $\te k(e)\,u\,\be k(f)\in\Mir_{\pv V}(\Cl X)$.
  Therefore, by Proposition~\ref{p:image-of-phi},
  the functor $\Phi_{\Kar_{\pv V}}$ restricts to a functor
  $\Kar(\Mir_{\pv V}(\Cl X))\to \Kar(\Mir_{\pv V}(\Cl Y))$
  whenever $\Phi$ is a central block map
  of a sliding block code $\varphi\colon\Cl X\to\Cl Y$.
  We proceed to show that this restriction depends on $\varphi$ only.
  
   \begin{Lemma}\label{l:independence-of-block-map}
     Consider a sliding block code
     $\varphi\colon \Cl X\to\Cl Y$,
     where $\Cl X$ is a subshift of $A^{\ZZ}$
     and $\Cl Y$ is a subshift of $B^{\ZZ}$.
     Suppose that $\Phi$ and $\Psi$ are central block maps of $\varphi$,
     with wings~$l$ and~$k$, respectively,  and suppose that $k\geq l$.
     Take $v\in\Om AV$ and
     words $u,w$ of $A^*$ with length~$k$
     and such that every factor of $uvw$ with length $2k+1$ belongs
     to~$L(\Cl X)$.
     Then the equality
     \begin{equation*}
       \overline{\Psi}(uvw)=\overline{\Phi}(\te l(u)\,v\,\be l(w))
     \end{equation*}
     holds.
   \end{Lemma}

   \begin{proof}
     As reasoned in previous proofs, it suffices
     to consider the case $v\in A^+$.
     Suppose first that $v\in A$.
     Then, $uvw$ is a word of length $2k+1$,
     and so $uvw=x_{[-k,k]}$ for some $x\in \Cl X$.
     Then (cf.~Example~\ref{eg:word-code-versus-block-code}), we have
     \begin{equation*}
       \overline{\Psi}(uvw)=(\varphi(x))_0=\overline{\Phi}(x_{[-l,l]})
       =\overline{\Phi}(\te l(u)\,v\,\be l(w)),
     \end{equation*}
     settling the case where $v$ is a letter.
     Let $\varphi'$ (respectively, $\varphi''$)
     be the sliding block code $A^{\ZZ}\to B^{\ZZ}$
     having $\Phi$ (respectively, $\Psi$) as a central block map.
     Let $x\in A^{\ZZ}$ and $n\geq 2k+1$ be such that
     $x_{[1,n]}=uvw$. Take $y=\varphi'(x)$ and $z=\varphi''(x)$.
     Assuming $1+k\leq i\leq n-k$,
     the word $x_{[i-k,i+k]}$
     belongs to $L(\Cl X)$, as it has length $2k+1$.
     By the already settled case, we know that
     \begin{equation*}
       y_i=\Psi(x_{[i-k,i+k]})=\overline{\Phi}(x_{[i-l,i+l]})=z_i,
     \end{equation*}
     whenever $1+k\leq i\leq n-k$.
     Finally, we have
     \begin{align*}
       \overline{\Phi}(\te l(u)\,v\,\be l(w))
       =\overline{\Phi}(x_{[1+(k-l),n-(k-l)]})
       &=z_{[1+k,n-k]}\\
       &=y_{[1+k,n-k]}
       =\overline{\Psi}(x_{[1,n]})
       =\overline{\Psi}(uvw),
     \end{align*}
     establishing the result.
   \end{proof}
   
   \begin{Cor}\label{c:independence-of-block-map}
     Consider a sliding block code
     $\varphi\colon \Cl X\to\Cl Y$,
     with $\Cl X\subseteq A^{\ZZ}$
     and~\mbox{$\Cl Y\subseteq B^{\ZZ}$}.
     If $\Phi$ and $\Psi$ are central block maps of $\varphi$,
     then the restrictions
     of $\Phi_{\Kar_{\pv V}}$ and $\Psi_{\Kar_{\pv V}}$
     to $\Kar(\Mir_{\pv V}(\Cl X))$ are equal.
   \end{Cor}

   \begin{proof}
     Suppose that the wings of $\Phi$ and $\Psi$
     are respectively $l$ and $k$.
     Let $u$ be an element of $\Mir_{\pv V}(\Cl X)$ such that
     $u=euf$ for some idempotents~$e$ and $f$ of~$\Om AV$.
     Then we have
     $ \overline{\Psi}(\te k(e)\,u\,\be k(f))
     =\overline{\Phi}(\te l(e)\,u\,\be l(f))$
     by Lemma~\ref{l:independence-of-block-map}.
   \end{proof}

   \begin{Def}
     Let $\varphi\colon\ci X\to\ci Y$ be a sliding block code, with $\Cl X\subseteq A^{\ZZ}$
     and~\mbox{$\Cl Y\subseteq B^{\ZZ}$}.
     We define the functor
     $\varphi_{\Kar_{\pv V}}\colon\Kar(\Mir_{\pv V}(\ci X))\to\Kar(\Mir_{\pv V}(\ci Y))$
     as being the restriction
     to $\Kar(\Mir_{\pv V}(\ci X))$
     of the functor 
     $\Phi_{\Kar_{\pv V}}\colon \Kar(\Om AV)\to\Kar(\Om BV)$,
     whenever $\Phi$ is a central block map of $\varphi$ (remember that $\Phi_{\Kar_{\pv V}}\bigl(\Kar(\Mir_{\pv V}(\ci X))\bigr)$
     is indeed contained in $\Kar(\Mir_{\pv V}(\ci Y))$,
     as observed in the paragraph before Lemma~\ref{l:independence-of-block-map}).
     By~Corollary~\ref{c:independence-of-block-map},
     the map $\varphi_{\Kar_{\pv V}}$ does not depend on the choice
     of~$\Phi$.
   \end{Def}

   \begin{Rmk}
     Let $\Cl X$ be a subshift of
     $A^{\ZZ}$.
     The identity $1_A\colon A\to A$
     is a central block map of the identity $1_{\Cl X}\colon\Cl X\to\Cl X$,
     and so the formula
     $(1_{\Cl X})_{\Kar_{\pv V}}=1_{\Mir_{\pv V}(\Cl X)}$ holds.
   \end{Rmk}

   \begin{Rmk}\label{rmk:restriction-phik-to-shadow}
  By Proposition~\ref{p:image-of-phi},
$\varphi_{\Kar_{\pv V}}(\Sha_{\pv V}(\Cl X))\subseteq \Sha_{\pv V}(\Cl Y)$
for every morphism $\varphi\colon\Cl X\to\Cl Y$.
\end{Rmk}

The next proposition
is the first step
to show that the correspondence $\varphi\mapsto\varphi_{\Kar_{\pv V}}$
defines a functor.

   \begin{Prop}\label{p:1-code-case}
     Let $\varphi\colon\Cl X\to\Cl Z$
     and $\psi\colon\Cl Z\to\Cl Y$
     be sliding block codes such that
     either~$\varphi$ or~$\psi$
     is a $1$-code.
     Then $\psi_{\Kar_{\pv V}}\circ\varphi_{\Kar_{\pv V}}=(\psi\circ\varphi)_{\Kar_{\pv V}}$.
\end{Prop}

   \begin{proof}
     Note that it suffices to prove
     that the restrictions of
     $\psi_{\Kar_{\pv V}}\circ\varphi_{\Kar_{\pv V}}$ and $(\psi\circ\varphi)_{\Kar_{\pv V}}$
     to the set of morphisms of $\Kar(\Mir_{\pv V}(\Cl X))$ are equal.
     Let $(e,u,f)$ be a morphism of $\Kar(\Mir_{\pv V}(\Cl X))$.
      Suppose that
      $\Phi$ and $\Psi$ are central block maps
      of $\varphi$ and $\psi$, respectively.

      Suppose first that $\Psi$ has wing $0$,
      and let $k$ be the wing of $\Phi$.      
      Then $\psi_{\Kar_{\pv V}}\circ\varphi_{\Kar_{\pv V}}(e,u,f)$ is equal to the triple
      \begin{equation*}
      \Bigl(\overline{\Psi}\circ \bar{\Phi}[\te k(e)\,e\,\be k(e)],
      \overline{\Psi}\circ \overline{\Phi}[\te k(e)\,euf\,\be k(f)],
      \overline{\Psi}\circ \overline{\Phi}[\te k(f)\,f\,\be k(f)]\Bigr).
      \end{equation*}
      By Lemma~\ref{l:composition-of-block-maps},
      the latter is equal to $(\psi\circ\varphi)_{\Kar_{\pv V}}(e,u,f)$.

It remains to consider the case in
which $\Phi$ has wing $0$. Let $k$ be the wing of $\Psi$.
Then
      \begin{align*}
      &\psi_{\Kar_{\pv V}}\circ\varphi_{\Kar_{\pv V}}(e,u,f)=\\
      &
      =\Psi_{\Kar_{\pv V}}\Bigl(\overline{\Phi}(e),
      \overline{\Phi}(u),
      \overline{\Phi}(f)\Bigr)\\
      &=\Bigl(\overline{\Psi}[\te k(\overline{\Phi}(e))
      \cdot
      \overline{\Phi}(e)\cdot
      \be k(\overline{\Phi}(e))],\,
      \overline{\Psi}[\te k(\overline{\Phi}(e))\cdot \overline{\Phi}(u)\cdot
      \be k(\overline{\Phi}(f))],\,
\overline{\Psi}[\te k(\overline{\Phi}(f))\cdot \overline{\Phi}(f)\cdot
\be k(\overline{\Phi}(f))]
      \Bigr)\\
&=\Bigl(\overline{\Psi}\circ \overline{\Phi}[\te k(e)\cdot e\cdot \be k(e)],\,
      \overline{\Psi}\circ \overline{\Phi}[\te k(e)\cdot u\cdot \be k(f)],\,
      \overline{\Psi}\circ \overline{\Phi}[\te k(f)\cdot f\cdot\be k(f)]\Bigr),
      \end{align*}
where the last equality holds  because $\overline{\Phi}$ is a homomorphism.
Again by Lemma~\ref{l:composition-of-block-maps},
we conclude that
$\psi_{\Kar_{\pv V}}\circ\varphi_{\Kar_{\pv V}}(e,u,f)=(\psi\circ\varphi)_{\Kar_{\pv V}}(e,u,f)$.
\end{proof}

\begin{Prop}\label{p:inverse-functoriality}
  If the sliding block code $\varphi\colon\ci X\to\ci Y$ is a conjugacy of subshifts,
  then the functor $\varphi_{\Kar_{\pv V}}\colon\Kar(\Mir_{\pv V}(\Cl X))\to\Kar(\Mir_{\pv V}(\Cl Y))$
  is an isomorphism of compact categories, and the equality
  \begin{equation*}
    (\varphi^{-1})_{\Kar_{\pv V}}=(\varphi_{\Kar_{\pv V}})^{-1}
  \end{equation*}
  holds.
\end{Prop}

\begin{proof}
  By Proposition~\ref{p:decomposition-of-morphisms}, if the sliding block code $\varphi$ is a conjugacy, then
  there are $1$-conjugacies~$\alpha$ and~$\beta$ such that the following diagram commutes:
    \begin{equation*}
    \xymatrix{
      &\Cl Z\ar[rd]^{\beta}\ar@{=>}[ld]_{\alpha}&\\
      \Cl X\ar[rr]^{\varphi}&&\Cl Y
    }
  \end{equation*}
  Using several times Proposition~\ref{p:1-code-case},
  we deduce the following chain of equalities:
  \begin{align*}
    1_{\Kar(\Mir_{\pv V}(\Cl X))}&=(1_{\Cl X})_{\Kar_{\pv V}}\\
    &=(\alpha\circ\alpha^{-1})_{\Kar_{\pv V}}\\
    &=\alpha_{\Kar_{\pv V}}\circ (\alpha^{-1})_{\Kar_{\pv V}}\qquad\text{(since $\alpha$ is a $1$-code)}\\
    &=\alpha_{\Kar_{\pv V}}\circ 1_{\Kar(\Mir_{\pv V}(\Cl Z))} \circ(\alpha^{-1})_{\Kar_{\pv V}}\\
    &=\alpha_{\Kar_{\pv V}}\circ (1_{\Cl Z})_{\Kar_{\pv V}} \circ(\alpha^{-1})_{\Kar_{\pv V}}\\
    &=\alpha_{\Kar_{\pv V}}\circ (\beta^{-1}\circ\beta)_{\Kar_{\pv V}} \circ(\alpha^{-1})_{\Kar_{\pv V}}\\
    &=\alpha_{\Kar_{\pv V}}\circ (\beta^{-1})_{\Kar_{\pv V}}\circ\beta_{\Kar_{\pv V}} \circ (\alpha^{-1})_{\Kar_{\pv V}}\qquad\text{(since $\beta$ is a $1$-code)}\\
    &=(\alpha\circ \beta^{-1})_{\Kar_{\pv V}}\circ(\beta\circ\alpha^{-1})_{\Kar_{\pv V}}\qquad\text{(because $\beta$ and $\alpha$ are $1$-codes)}\\
    &=(\varphi^{-1})_{\Kar_{\pv V}}\circ\varphi_{\Kar_{\pv V}}.
  \end{align*}
  Similarly, the next chain of equalities is valid:
  \begin{align*}
    1_{\Kar(\Mir_{\pv V}(\Cl Y))}&=(1_{\Cl Y})_{\Kar_{\pv V}}\\
    &=(\beta\circ\beta^{-1})_{\Kar_{\pv V}}\\
    &=\beta_{\Kar_{\pv V}}\circ (\beta^{-1})_{\Kar_{\pv V}}\qquad\text{(since $\beta$ is a $1$-code)}\\
    &=\beta_{\Kar_{\pv V}}\circ 1_{\Kar(\Mir_{\pv V}(\Cl Z))} \circ(\beta^{-1})_{\Kar_{\pv V}}\\
    &=\beta_{\Kar_{\pv V}}\circ (1_{\Cl Z})_{\Kar_{\pv V}} \circ(\beta^{-1})_{\Kar_{\pv V}}\\
    &=\beta_{\Kar_{\pv V}}\circ (\alpha^{-1}\circ\alpha)_{\Kar_{\pv V}} \circ(\beta^{-1})_{\Kar_{\pv V}}\\
    &=\beta_{\Kar_{\pv V}}\circ (\alpha^{-1})_{\Kar_{\pv V}}\circ \alpha_{\Kar_{\pv V}}\circ(\beta^{-1})_{\Kar_{\pv V}}\qquad\text{(since $\alpha$ is a $1$-code)}\\
    &=(\beta\circ\alpha^{-1})_{\Kar_{\pv V}}\circ(\alpha\circ\beta^{-1})_{\Kar_{\pv V}}\qquad\text{(because $\alpha$ and $\beta$ are $1$-codes)}\\
    &=\varphi_{\Kar_{\pv V}}\circ(\varphi^{-1})_{\Kar_{\pv V}}.
  \end{align*}
  Therefore, the proposition holds.  
\end{proof}

The following statement is now obvious.

\begin{Cor}\label{c:splitting-category-free-profinite-is-conj-invariant}
  Let $\pv V$ be a pseudovariety
  of semigroups containing $\Lo {Sl}$
  and such that $\pv V=\pv V\ast\pv D$.
  The compact category $\Kar(\Mir_{\pv V}(\Cl X))$
  is a conjugacy invariant, up to isomorphism of compact categories.
\end{Cor}

Assuming $\Cl X$ is irreducible, we may consider the $\J$-minimum $\J$-class
$J_{\pv V}(\Cl X)$ of $\Sha_{\pv V}(\Cl X)$
and the $\J$-minimum $\J$-class
$\widetilde J_{\pv V}(\Cl X)$ of $\Mir_{\pv V}(\Cl X)$.
To such $\Cl X$ we associate two profinite groups: the Sch\"utzenberger group
$G_{\pv V}(\Cl X)$
of $J_{\pv V}(\Cl X)$ and the Sch\"utzenberger group
$\widetilde G_{\pv V}(\Cl X)$
of $\widetilde J_{\pv V}(\Cl X)$,
respectively
isomorphic to the maximal subgroups
of $J_{\pv V}(\Cl X)$ and to the maximal subgroups of $\widetilde J_{\pv V}(\Cl X)$. The following straightforward consequence
of
Proposition~\ref{p:inverse-functoriality}
was first established in~\cite{ACosta:2006}.

 \begin{Cor}\label{c:invariance-of-GX}
   Let $\pv V$ be a pseudovariety
   of semigroups containing $\Lo {Sl}$
   and such that $\pv V=\pv V\ast\pv D$.   
   Suppose that the subshift $\Cl X$ is irreducible.
  Then the profinite groups
  $G_{\pv V}(\Cl X)$ and $\widetilde G_{\pv V}(\Cl X)$ are conjugacy invariants,
  up to isomorphism of profinite groups.
\end{Cor}

\begin{proof}
      The idempotents of $\widetilde J_{\pv V}(\Cl X)$ are
    the minimal objects of $\Kar(\Mir_{\pv V}(\Cl X))$
    with respect to the retraction order $\prec$, by Proposition~\ref{p:retraction}.
    The partial order $\prec$ is clearly preserved by isomorphisms
    of categories.
    Hence,
    in view of Proposition~\ref{p:inverse-functoriality},
    if $\varphi\colon\Cl X\to\Cl Y$  is a conjugacy,
    then each idempotent~$e$ in $\widetilde J_{\pv V}(\Cl X)$
    is mapped via $\varphi_{\Kar_{\pv V}}$
    to an idempotent $f=\varphi_{\Kar_{\pv V}}(e)$ in $\widetilde J_{\pv V}(\Cl Y)$.
    In particular, the profinite groups $G_e$ and $G_f$
    are isomorphic, by Proposition~\ref{p:local-isomorphism},
    establishing the conjugacy invariance of $\widetilde G_{\pv V}(\Cl X)$.

    As $\varphi_{\Kar_{\pv V}}(\Sha_{\pv V}(\Cl X))\subseteq \Sha_{\pv V}(\Cl Y)$
    and $\varphi_{\Kar_{\pv V}}^{-1}(\Sha_{\pv V}(\Cl Y))\subseteq \Sha_{\pv V}(\Cl X)$
    (Remark~\ref{rmk:restriction-phik-to-shadow}),
    the arguments in the previous paragraph also
    yeld the conjugacy invariance of the profinite group $G_{\pv V}(\Cl X)$.
\end{proof}

We now establish the functoriality
of the correspondence \mbox{$\Cl X\mapsto\Kar(\Mir_{\pv V}(\Cl X))$}.

\begin{Thm}\label{t:1-conjugacy-case}
     Let $\pv V$ be a pseudovariety
   of semigroups containing $\Lo {Sl}$
   and such that $\pv V=\pv V\ast\pv D$.   
     The following data defines a functor from
     the category of shifts to the category of compact categories.
     \begin{equation*}
       \xymatrix{
         \Cl X\ar[r]\ar[d]_{\varphi}&\Kar(\Mir_{\pv V}(\Cl X))\ar[d]^{\varphi_{\Kar_{\pv V}}}\\
         \Cl Y\ar[r]&\Kar(\Mir_{\pv V}(\Cl Y))
         }
     \end{equation*}
   \end{Thm}

   \begin{proof}
     Consider composable sliding block codes
     $\varphi\colon \Cl X_1\to\Cl X_2$
     and
     \mbox{$\psi\colon \Cl X_2\to\Cl X_3$}.
     Then, by Proposition~\ref{p:decomposition-of-morphisms}, we      
     may build a commutative diagram of sliding block codes,
     displayed in Figure~\ref{fig:functoriality}, such that
     all the maps not in the base of the outer triangle
     (the maps $\alpha$, $\beta$, $\gamma$, $\delta$, $\mu$ and $\nu$)
     are $1$-codes, with $\alpha$, $\gamma$ and $\mu$
     being conjugacies.
     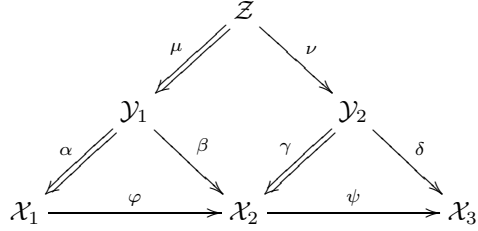
\begin{figure}[h]
       \centering
            \begin{equation*}
     \xymatrix{
       &&\Cl Z\ar@{=>}[ld]_\mu\ar[rd]^{\nu} &&\\
       &\Cl Y_1\ar@{=>}[ld]_{\alpha}
       \ar[rd]^{\beta}
       &&\Cl Y_2\ar[rd]^{\delta}\ar@{=>}[ld]_{\gamma}&\\
      \Cl X_1\ar[rr]^\varphi
      &&\Cl X_2\ar[rr]^\psi
      &&\Cl X_3
     }
   \end{equation*}
       \caption{Triangle with base $\psi\circ\varphi$.}
       \label{fig:functoriality}
     \end{figure}

     Applying several times Proposition~\ref{p:1-code-case}
     and~\ref{p:inverse-functoriality}, we deduce the following chain of
     equalities:     
     \begin{align*}
       (\psi\circ\varphi)_{\Kar_{\pv V}}
       &=((\delta\circ\nu)\circ (\alpha\circ\mu)^{-1})_{\Kar_{\pv V}}\\
       &=(\delta\circ\nu)_{\Kar_{\pv V}}\circ ((\alpha\circ\mu)^{-1})_{\Kar_{\pv V}}
       \qquad\text{(by Proposition~\ref{p:1-code-case},
         since $\delta\circ\nu$ is a $1$-code)}
       \\
       &=(\delta\circ\nu)_{\Kar_{\pv V}}\circ ((\alpha\circ\mu)_{\Kar_{\pv V}})^{-1}
       \qquad\text{(by Proposition~\ref{p:inverse-functoriality})}\\
       &=\delta_{\Kar_{\pv V}}\circ\nu_{\Kar_{\pv V}}\circ (\alpha_{\Kar_{\pv V}}\circ\mu_{\Kar_{\pv V}})^{-1}
       \qquad\text{(by Proposition~\ref{p:1-code-case}, as
         $\delta$, $\nu$, $\alpha$ and
         $\mu$ are $1$-codes)}
       \\
       &=\delta_{\Kar_{\pv V}}\circ\nu_{\Kar_{\pv V}}\circ (\mu_{\Kar_{\pv V}})^{-1}\circ(\alpha_{\Kar_{\pv V}})^{-1}\\
       &=\delta_{\Kar_{\pv V}}\circ\nu_{\Kar_{\pv V}}\circ (\mu^{-1})_{\Kar_{\pv V}}\circ(\alpha^{-1})_{\Kar_{\pv V}}\qquad\text{(by Proposition~\ref{p:inverse-functoriality})}\\
       &=\delta_{\Kar_{\pv V}}\circ(\nu\circ \mu^{-1})_{\Kar_{\pv V}}\circ(\alpha^{-1})_{\Kar_{\pv V}}\qquad\text{(by Proposition~\ref{p:1-code-case}, as $\nu$ is a $1$-code)}\\
       &=\delta_{\Kar_{\pv V}}\circ(\gamma^{-1}\circ \beta)_{\Kar_{\pv V}}\circ(\alpha^{-1})_{\Kar_{\pv V}}\\
       &=\delta_{\Kar_{\pv V}}\circ(\gamma^{-1})_{\Kar_{\pv V}}\circ \beta_{\Kar_{\pv V}}\circ(\alpha^{-1})_{\Kar_{\pv V}}
       \qquad\text{(by Proposition~\ref{p:1-code-case},
         as $\beta$ is a $1$-code)}\\
       &=(\delta\circ\gamma^{-1})_{\Kar_{\pv V}}\circ (\beta\circ\alpha^{-1})_{\Kar_{\pv V}}
       \qquad\text{(by Proposition~\ref{p:1-code-case},
         as $\delta$ and $\beta$ are $1$-codes)}\\
       &=\psi_{\Kar_{\pv V}}\circ \varphi_{\Kar_{\pv V}},\\
     \end{align*}     
   thus establishing the result.    
 \end{proof}

 Our proof of Theorem~\ref{t:1-conjugacy-case}
 depends on the Curtis--Hedlund--Lyndon theorem (Theorem~\ref{t:morphisms-are-sliding-block-codes}). It may be interesting
 to obtain a more direct proof, not depending on the use of block maps. We leave that as an open problem.
 
\section{Flow equivalence}
\label{sec:flow-equivalence}

We turn our attention to flow equivalence, having~\cite[Section 13.7]{Lind&Marcus:1996} and \cite{Beal&Berstel&Eilers&Perrin:2010arxiv} as guiding references.
Two discrete-time dynamical systems are \emph{flow equivalent} if their
suspension flows (or mapping tori) are conjugate modulo a time change.
Parry and Sullivan showed that within the class of subshifts, flow equivalence
is the equivalence relation between subshifts generated by conjugacy
and~\emph{symbol expansion}~\cite{Parry&Sullivan:1975},
described next. Fix an alphabet $A$ and a letter $\alpha$ of $A$.
Let $\dia$ be a letter not in~$A$,
and let $B=A\cup\{\dia\}$.
The \emph{symbol expansion of $A$ associated to $\alpha$}
is the homomorphism $\ci E\colon A^+\to B^+$
such that $\ci E(\alpha)=\alpha\dia$ and $\ci E(a)=a$ for all
$a\in A\setminus\{\alpha\}$.
The \emph{symbol expansion of a subshift \Cl X of $A^{\ZZ}$ relative to $\alpha$}
is the least subshift $\Cl X'$ of $B^{\ZZ}$ such that
$L(\Cl X')$
contains $\ci E(L(\Cl X))$.
A \emph{symbol expansion of $\Cl X$} is a symbol expansion of~$\Cl X$
relative to some letter.

\begin{Rmk}\label{r:image-of-E}
Using induction on the length of
words, one verifies that
\begin{equation*}
  \ci E(A^+)=B^+
  \setminus
  \Bigl(\dia B^\ast
  \cup
  B^\ast \alpha
  \cup
  \bigcup_{x\in A}\!{B^\ast \alpha xB^\ast}
  \cup
  \bigcup_{x\in B\setminus \{\alpha\}}\!\!\!\!{B^\ast x\dia B^\ast}
  \Bigr).
\end{equation*}
In particular, one sees that $\ci E(A^+)$ is a locally testable language.
\end{Rmk}

Throughout this section, as in Section~\ref{sec:funct-corr-from},
$\pv V$ will always be a pseudovariety of semigroups
containing $\Lo {Sl}$
such that $\pv V=\pv V\ast\pv D$,
but (unlike  Section~\ref{sec:funct-corr-from}) with the additional requirement that $\pv V$ is monoidal.
It is folklore that if  $\pv V$ is monoidal and contains~$\pv{Sl}$,
then $S\in\pv V$ if and only if $S^I\in\pv V$~\cite{Eilenberg:1976}.
From that it follows that $\Om AV^I$ is pro-$\pv V$ whenever $\pv V$ is monoidal and contains $\pv {Sl}$,
a property that we shall need.

Let us return to the symbol expansion homomorphism $\ci E\colon A^+\to B^+$
introduced in the first paragraph of this section.
The unique extension of $\ci E$ to a continuous homomorphism
$\Om AV\to \Om BV$ will also be denoted by $\ci E$.  We let $\ci E(I)=I=\varepsilon$. Because $\Om AV^I$ is a pro-$\pv V$ semigroup, as we are assuming $\pv V$ to be monoidal, we may consider the unique continuous semigroup homomorphism $\ci C\colon \Om BV\to \Om AV^I$ such that
$\ci C(\dia)=\varepsilon=I$ and $\ci C(a)=a$ for all $a\in A$.
    The notation $\mathcal C$ is used because its restriction
    to $B^+$ is said to be a \emph{symbol contraction}.
    Note that $\ci C\circ\ci E(u)=u$ for all $u\in \Om AV$, since this is
  clearly true for finite words and $\ci C\circ\ci E$ is continuous.
  In particular, $\ci E$ is injective, and we may use
  the notation $\ci E^{-1}$ for the restriction
  of $\ci C$ to $\ci E(\Om AV)$. Observe that $\ci E(\Om AV)$
  is clopen by Remark~\ref{r:image-of-E}.

\begin{Lemma}\label{l:symbol-exp-3}
  Let $v\in \Om AV$. The following properties hold:
  \begin{enumerate}
  \item For $x,y,u\in \Om BV^I$, if
$x\cdot \ci E(v)\cdot y=\ci E(u)$ then
$x,y\in \ci E(\Om AV^I)$ and
$u=\ci E^{-1}(x)v\ci E^{-1}(y)$.
\item If $\ci E(v)\in \overline{L(\Cl X')}$, then we have $v\in \overline{L(\Cl X)}$.
  \end{enumerate}
\end{Lemma}

\begin{proof}
  The case where $x,y,v,u$
  are finite words is Lemma 12.3 in~\cite{ACosta&Steinberg:2016},
  following very easily from Remark~\ref{r:image-of-E}.

  For the general case, suppose that
  $x\cdot \ci E(v)\cdot y=\ci E(u)$
  and let $(x_n)_n$, $(v_n)_n$, $(y_n)_n$
  be sequences of words respectively converging
  to $x$, $v$ and $y$.
  Since $\ci E(\Om AV)$ is open,
  for all large enough $n$
  there is a word $u_n$
  such that $x_n\cdot \ci E(v_n)\cdot y_n=\ci E(u_n)$.
  Taking subsequences, we may as well suppose that
  $u_n\to u$.
  By the special case for words,
  we have $x_n,y_n\in \ci E(A^*)$
  and $u_n=\ci E^{-1}(x_n)v_n\ci E^{-1}(y_n)$ (bear in mind that $\ci E$ is injective).
  Because the mapping~$\ci E^{-1}$, from the closed space
  $\ci E(\Om AV^I)$ to $\Om AV^I$, is continuous,
  we deduce that $u=\ci E^{-1}(x)v\ci E^{-1}(y)$.

  Suppose now that $\ci E(v)\in \overline{L(\Cl X')}$.
  There is a sequence $(u_n)_n$ of elements of $L(\Cl X')$
  converging to $\ci E(v)$.
  Because $\ci E(\Om AV)$ is clopen,
  and by compactness, by taking subsequences we may in
  fact suppose that $u_n=\ci E(w_n)\in L(\Cl X')$
  for a sequence $(w_n)_n$ converging to some $w\in\Om AV$.
  Again by the case for words, we get $w_n\in L(\Cl X)$,
  thus $w\in\overline{L(\Cl X)}$.
  Since $\ci E$ is continuous, we have $\ci E(v)=\lim u_n=\ci E(w)$,
  whence $v=w\in \overline{L(\Cl X)}$, as $\ci E$ is injective.  
\end{proof}

In what follows, $\Cl X$ is a subshift of $A^{\ZZ}$.
We begin to record that the shadow of~$\Cl X$ is preserved by the symbol expansion.

\begin{Lemma}\label{l:expansion-shadow-is-also-preserved}
   The inclusion $\ci E(\Sha_{\pv V}(\Cl X))\subseteq \Sha_{\pv V}(\Cl X')$
  holds.
\end{Lemma}

\begin{proof}
  This is immediate, since $\ci E(L(\Cl X))\subseteq L(\Cl X')$
  by the definition of $\Cl X'$
  and because $\ci E\colon \Om AV\to\Om BV$
  is a continuous homomorphism.
\end{proof}

We next prove that the mirage is also preserved by the symbol expansion.

  \begin{Lemma}\label{l:E-preserves-the-mirage}
    The inclusion
    \begin{equation*}
      \ci E(\Mir_{\pv V}(\Cl X))\subseteq \Mir_{\pv V}(\Cl X')
    \end{equation*}
    holds. More precisely, one has
    \begin{equation*}
    \ci E(\Mir_{\pv V,k}(\Cl X))\subseteq \Mir_{\pv V,k}(\Cl X')  
    \end{equation*}
    for every positive integer~$k$.
  \end{Lemma}

  \begin{proof}
    Clearly, it suffices to show the second
    inclusion, as $\Mir_{\pv V}(\Cl Z)=\bigcap_{k\geq 1}\Mir_{\pv V,k}(\Cl Z)$
    for every subshift $\Cl Z$.
    
  Let $u$ be an element of $\Mir_{\pv V,k}(\Cl X)$. Suppose that
  $w\in B^+$ is a finite factor of $\ci E(u)$ with length at most $k$.
  Let $(u_n)_n$ be a sequence of elements of $A^+$ converging to~$u$. Then
  $u_n\in\Mir_{\pv V,k}(\Cl X)$
  for all sufficiently large $n$, as $\Mir_{\pv V,k}(\Cl X)$
  is a (clopen) neighborhood of $u$ (cf.~Remark~\ref{rmk:mvk-is-clopen}).
  On the other hand, $\overline{B^*wB^*}$
  is a clopen of $\Om BV$ containing $\ci E(u)$.
  Since $\ci E$ is continuous,
  we also have $\ci E(u_n)\in B^*wB^*$
  for all sufficiently large $n$.
  Therefore, we may take some word $u_m$
  in the intersection $\Mir_{\pv V,k}(\Cl X)\cap \ci E^{-1}(B^*wB^*)$.
  Since $w$ is a factor of $\ci E(u_m)$,
  in view of the equality in Remark~\ref{r:image-of-E},
  we see that 
  there are words $p,q\in B^\ast$,
  $x\in\{\alpha,\varepsilon\}$ and $y\in\{\dia,\varepsilon\}$
  such that $\ci E(u_m)=pxwyq$ and
  $xwy$ belongs to the image $\mathrm{Im} {\ci E}$, the possibilities
  for~$x$ and~$y$ depending on whether $w$ starts with $\dia$ or not, and
  whether $w$ ends with $\alpha$ or~not.
  By Lemma~\ref{l:symbol-exp-3}, the words  $p$ and $q$ also belong to $\mathrm{Im} {\ci E}$ and
  \begin{equation}\label{eq:E-preserves-the-mirage}
    u_m=\ci E^{-1}(p)\cdot \ci E^{-1}(xwy)\cdot \ci E^{-1}(q).
  \end{equation}
  Moreover, if
  $x=\alpha$, then $w$ starts with the letter $\dia$.
  Hence, if $x=\alpha$ or $y=\dia$,
  then $xwy$ has at least one occurrence
  of the letter $\dia$, and it has at least two occurrences if
  $x=\alpha$ and $y=\dia$. Therefore,
  for whatever possibility for $x\in\{\alpha,\varepsilon\}$ and $y\in\{\dia,\varepsilon\}$,
  it follows from the definition of the symbol
  contraction~$\ci C$ that $|\ci E^{-1}(xwy)|\leq |w|\leq k$. 
  Since $u_m\in \Mir_{\pv V,k}(\Cl X)$ and~\eqref{eq:E-preserves-the-mirage}
  holds, it follows that $\ci E^{-1}(xwy)\in L(\Cl X)$. Therefore,
  $xwy\in L(\Cl X')$, and
  so $w\in L(\Cl X')$. This proves that $\ci E(u)\in \Mir_{\pv V,k}(\Cl X')$.
\end{proof}

Concerning the contraction homomorphism, we first note the following fact.

\begin{Lemma}\label{l:contraction-shadow-is-also-preserved}
  The inclusion $\ci C(L(\Cl X')\setminus\{\dia\})\subseteq L(\Cl X)$
  holds, and so does the inclusion $\ci C(\Sha_{\pv V}(\Cl X')\setminus\{\dia\})\subseteq \Sha_{\pv V}(\Cl X)$.
\end{Lemma}

\begin{proof}
  If $u\in L(\Cl X')$,
  then $u$ is a factor of $\ci E(v)$ for some $v\in L(\Cl X)$,
  whence $\ci C(u)$ is  a factor of $\ci C\circ\ci E(v)=v$,
  showing that $\ci C(L(\Cl X'))\subseteq L(\Cl X)\cup\{\varepsilon\}$.
  Moreover, if $u\in L(\Cl X')\setminus\{\dia\}$,
  then at least one letter appearing in $u$ is not $\dia$ (as $\dia\dia\not\in L(\Cl X')$), thus $\ci C(u)\neq \varepsilon$.
  Since $\ci C$ is a continuous
  homomorphism from $\Om BV$ to $\Om AV^I$,
  we immediately obtain
  $\ci C(\Sha_{\pv V}(\Cl X')\setminus\{\dia\})\subseteq \Sha_{\pv V}(\Cl X)$.
  \end{proof}

  We will need the following lemma.
  
\begin{Lemma}\label{l:possibilities}
  Every pseudoword $u$ in $\Mir_{\pv V,2}(\Cl X')$
  is of one, and only one, of the following four types:
  \begin{enumerate}
  \item $u\in\{\alpha,\dia\}$\label{item:possibilities-1}
  \item $u\in \ci E(\Om AV)$\label{item:possibilities-2}
  \item $u=\dia v$ for some $v\in \ci E(\Om AV)$\label{item:possibilities-3}
  \item $u=v\alpha$ for some $v\in \ci E(\Om AV)$\label{item:possibilities-4}
  \item $u=\dia v\alpha$ for some $v\in\ci E(\Om AV)\cup\{\varepsilon\}$\label{item:possibilities-5}
  \end{enumerate}
\end{Lemma}

\begin{proof}
  We assume $u\in B^+$ first. We prove, by induction
  on the length of the word~$u$,
  that if $u$ belongs to $\Mir_{\pv V,2}(\Cl X')\cap B^+$,
  then $u$ is of one of five types~\eqref{item:possibilities-1}-\eqref{item:possibilities-5}. The base step is immediate: if the length of $u$ is one, then $u$ is of
  type~\eqref{item:possibilities-1} or~\eqref{item:possibilities-2}.
  
  Suppose
  that $u$ is a word with length at least two and that the lemma holds
  for words of smaller length.
  Consider first the case in which $u$ starts with the letter $\dia$,
  and take a factorization $u=\dia w$.
  Since~$w$ also belongs to the factorial set $\Mir_{\pv V,2}(\Cl X')$,
  we may apply the induction hypothesis
  to $w$. Let us see what happens in each case:
  \begin{itemize}
  \item If $w$ is of type \eqref{item:possibilities-1},
    then from $u=\dia w\in L(\Cl X')$
    we get $w=\alpha$, and so $u$ is of type \eqref{item:possibilities-5}.
  \item If $w$ falls into type~\eqref{item:possibilities-2},
    then $u$ is of type~\eqref{item:possibilities-3}.
  \item It is impossible that $w$ falls into
    types~\eqref{item:possibilities-3} or~\eqref{item:possibilities-5},
    otherwise $\dia\,\dia$ would be a prefix of~$u$, contradicting
    that every factor of length two of $u$ is in $L(\Cl X')$.
  \item If $w$ falls into type~\eqref{item:possibilities-4},
    then $u$ is of type~\eqref{item:possibilities-5}.
  \end{itemize}
  Therefore, in all possible cases, $u$ is of one of the listed types, whenever $u\in\dia B^*$.

  Suppose now that $u$ starts with the letter $\alpha$.
  Since the factors of length two of $u$
  belong to $L(\Cl X')$,
  we must have $u=\alpha\dia w=\ci E(\alpha)\cdot w$ for some $w\in B^*\setminus \dia B^*$. Applying the induction hypothesis
  to $w$, one sees that $u$ must be either
  of type~\eqref{item:possibilities-2} or~\eqref{item:possibilities-4}.
  A similar reasoning is valid if $u$ starts with a letter
  $a\in A\setminus\{\alpha\}$, as we then have $u=aw=\ci E(a)\cdot w$ for
  some $w\in B^*\setminus \dia B^*$. We have thus concluded that the inductive step holds, and that the lemma is valid for every $u\in \Mir_{\pv V,2}(\Cl X')\cap B^+$.

  Now, let $u$ be a pseudoword
  belonging to $\Mir_{\pv V,2}(\Cl X')$.
  Since $\Mir_{\pv V,2}(\Cl X')$ is clopen,
  there is a sequence $(u_n)_n$
  of elements of $\Mir_{\pv V,2}(\Cl X')\cap B^+$
  converging to $u$.
  As the number of possible types is finite, taking subsequences,
  we may as well suppose that
  all elements of $(u_n)_n$
  are of the same type, among the five
  possible types~\eqref{item:possibilities-1}-\eqref{item:possibilities-5}.
  Since $\ci E(\Om AV)$ is a closed set and the multiplication is continuous,
  it follows that $u$ is of the same type as that of
  the terms $u_n$.

  We end by observing that no pseudoword can be of more than one
  of the five types~\eqref{item:possibilities-1}-\eqref{item:possibilities-5},
  since no element of $\ci E(\Om AV)$ starts with $\dia$ or ends with $\alpha$.
\end{proof}

Next is a sort of weak converse of Lemma~\ref{l:E-preserves-the-mirage}.
  
  \begin{Lemma}\label{l:xi-preserves-the-mirage}
    The inclusion
    \begin{equation*}
      \ci C (\Mir_{\pv V}(\Cl X')\setminus \{\dia \})\subseteq \Mir_{\pv V}(\Cl X)
    \end{equation*}
    holds. More precisely, one has
    \begin{equation*}
    \ci C (\Mir_{\pv V,2k}(\Cl X')\setminus \{\dia \})\subseteq \Mir_{\pv V,k}(\Cl X)  
    \end{equation*}
    for every positive integer~$k$.
  \end{Lemma}

  \begin{proof}
    Because
    $\Mir_{\pv V}(\Cl Z)=\bigcap_{k\geq 1}\Mir_{\pv V,2k}(\Cl Z)$
    for every subshift $\Cl Z$, we are reduced to showing the second inclusion.
    
    Let $u\in \Mir_{\pv V,2k}(\Cl X')\setminus \{\dia \}$.
    Let $w$ be a finite factor of $\ci C(u)$ of length at most~$k$.
    By Lemma~\ref{l:possibilities},
    there are $x\in\{\dia,\varepsilon\}$, $y\in\{\alpha,\varepsilon\}$ and $v\in \Om AV^I$ such
    that
    \begin{equation*}
    u=x\ci E(v)y.  
    \end{equation*}
    Since $\ci C\circ \ci E$ is the identity, we have
    \begin{equation*}
    \ci C (u)=vy.  
    \end{equation*}
    Hence $\ci E(w)$ is a finite factor of $\ci E(v)\cdot \ci E(y)$.
    Observe that $|\ci E(w)|\leq 2|w|\leq 2k$.

    Suppose that $y=\varepsilon$.
    Then $\ci E(w)$ is a factor of $\ci E(v)$,
    and so it is a factor of~$u$.
    Since $u\in \Mir_{\pv V,2k}(\Cl X')$, it follows that
    $\ci E(w)\in L(\Cl X')$. Applying Lemma~\ref{l:symbol-exp-3},
    we then get~$w\in L(\Cl X)$.

    Finally, suppose that $y=\alpha$.
    Then we have
    $u\,\dia=x\ci E(v)\alpha\,\dia=x\ci E(v)\ci E(\alpha)$,
    and $\ci E(w)$ is a factor of $u\,\dia$.
    As discussed in Section~\ref{sec:conn-with-symb},
    the set $\Mir_{\pv V,2k}(\Cl X')$
    is prolongable, whence $ub\in \Mir_{\pv V,2k}(\Cl X')$
    for some letter $b$.
    But $y=\alpha$ is a suffix of $u$,
    and so $\alpha b$ is a finite suffix of $u\dia$.
    In particular, $\alpha b\in L(\Cl X')$,
    implying $b=\dia$.
    Therefore, we have $u\dia\in\Mir_{\pv V,2k}(\Cl X')$.
    Since $\ci E(w)$ is a finite factor of $u\dia$
    of length at most $2k$,
    we must have $\ci E(w)\in L(\Cl X')$.
    Again by Lemma~\ref{l:symbol-exp-3}, we conclude that $w\in L(\Cl X)$.
  \end{proof}

  The following improvement of Lemma~\ref{l:possibilities}
  is not necessary for the sequel, but it may be worthwhile to have it in mind.
  
  \begin{Cor}\label{c:refined-possibilities}
    The equality 
      \begin{equation}\label{eq:refined-possibilities}
    \Mir_{\pv V}(\Cl X')\cap \ci E(\Om AV)=\ci E(\Mir_{\pv V}(\Cl X)).
  \end{equation}
  holds. Consequently, every pseudoword $u$ in $\Mir_{\pv V}(\Cl X')$
  is of one, and only one, of the following four types:
  \begin{enumerate}
  \item $u\in \{\alpha,\dia\}$\label{item:refined-possibilities-1}
  \item $u\in \ci E(\Mir_{\pv V}(\Cl X))$\label{item:refined-possibilities-2}
  \item $u=\dia v$ for some $v\in \ci E(\Mir_{\pv V}(\Cl X))$\label{item:refined-possibilities-3}
  \item $u=v\alpha$ for some $v\in \ci E(\Mir_{\pv V}(\Cl X))$\label{item:refined-possibilities-4}
  \item $u=\dia v\alpha$
    for some $v\in\ci E(\Mir_{\pv V}(\Cl X))\cup\{\varepsilon\}$\label{item:refined-possibilities-5}
  \end{enumerate}
\end{Cor}

\begin{proof}
  The inclusion
  $\ci E(\Mir_{\pv V}(\Cl X))\subseteq\Mir_{\pv V}(\Cl X')\cap \ci E(\Om AV)$
  is in Lemma~\ref{l:E-preserves-the-mirage}.
  Conversely, if $v\in\Mir_{\pv V}(\Cl X')\cap \ci E(\Om AV)$,
  then, by Lemma~\ref{l:xi-preserves-the-mirage},
  we have $\ci E^{-1}(v)=\ci C(v)\in \Mir_{\pv V}(\Cl X)$,
  and so~\eqref{eq:refined-possibilities} holds.

  Let $u\in\Mir_{\pv V}(\Cl X')$.
  Then $u$ is in one of the situations of Lemma~\ref{l:possibilities}.
  Since $\Mir_{\pv V}(\Cl X')$ is factorial and the equality~\eqref{eq:refined-possibilities} is valid, we conclude that in
  the list given for such~$u$ by Lemma~\ref{l:possibilities},
  we may replace $\ci E(\Om AV)$ by $\ci E(\Mir_{\pv V}(\Cl X))$.
\end{proof}

We adapt to compact categories
the notions of isomorphism of functors and of equivalence of categories.
For that purpose, the following simple fact is needed.

\begin{Lemma}\label{l:continuity-of-inversion}
  In a compact category $C$, the set of isomorphisms is a closed subspace of $\Mor (C)$, and the mapping $\varphi\mapsto \varphi^{-1}$
  is continuous on this subspace.
\end{Lemma}

\begin{proof}
  Observe first that the set of identities $\{1_c\mid c\in\Obj(C)\}$
  is a closed subspace of $\Mor(C)$,
  since the map $c\in\Obj(C)\mapsto 1_c$
  is continuous and $\Obj(C)$ is compact.
  Therefore, if the net $(\varphi_i)_{i\in I}$ of isomorphisms of $C$ converges
  to $\varphi$ then, by the continuity of the composition, every convergent subnet of $(\varphi_i^{-1})_{i\in I}$
  converges to an inverse of~$\varphi$. As $\Mor(C)$ is compact,
  we deduce that $(\varphi_i^{-1})_{i\in I}$ converges
  to $\varphi^{-1}$.
\end{proof}

Two continuous functors $F,G\colon C\to D$
between compact categories are \emph{continuously isomorphic},
written $F\cong G$,
when there is a \emph{continuous} natural isomorphism~$\eta\colon F\Rightarrow G$, which we define as natural isomorphism~$\eta\colon F\Rightarrow G$
such that the function $\Obj(C)\to \Mor(D)$ mapping each object $c$ of $C$
to the morphism $\eta_c\colon F(c)\to G(c)$ is continuous. By Lemma~\ref{l:continuity-of-inversion},
the inverse of a continuous natural isomorphism is a continuous
natural isomorphism, and so the relation $\cong$ is symmetric.
Moreover, it is straightforward that for all continuous functors $F,G\colon C\to D$ and $H,K\colon D\to E$ of compact categories,
if $F\cong G$ and $H\cong K$ then $H\circ F\cong K\circ G$.

A functor $F\colon C\to D$ between
compact categories $C$ and $D$ is a
\emph{continuous equivalence} if there is a continuous functor $G\colon D\to C$, such that $F\circ G \cong 1_D$
and $G\circ F \cong 1_C$.
Such $G$ is a \emph{continuous pseudo-inverse} of $F$.
We say that $C$ and $D$ are \emph{continuously equivalent} if there is a continuous equivalence $F\colon C\to D$. Note that the continuous equivalence of compact categories is an equivalence relation.

We are now ready to state the next theorem.
We mention that it applies in particular
when $\pv V=\overline{\pv{H}}$,
for a pseudovariety of groups $\pv H$,
as the equality $\overline{\pv{H}}=\overline{\pv{H}}*\pv D$ holds, and $\overline{\pv{H}}\supseteq\pv A\supseteq \Lo {Sl}$~\cite{Eilenberg:1976,Rhodes&Steinberg:2009qt}.

\begin{Thm}\label{t:splitting-category-free-profinite-is-flow-invariant}
  Let $\pv V$ be a monoidal pseudovariety
  of semigroups containing $\Lo {Sl}$
  and such that $\pv V=\pv V\ast\pv D$.
  With respect to the continuous equivalence of compact categories,
  the equivalence class of the compact category $\Kar(\Mir_{\pv V}(\Cl X))$ is a flow equivalence invariant.
\end{Thm}

\begin{proof}
  Thanks to Corollary~\ref{c:splitting-category-free-profinite-is-conj-invariant}, to show the flow invariance
  of the continuous equivalence class of $\Kar(\Mir_{\pv V}(\Cl X))$, it only remains
  to show that it is invariant under symbol expansion.
      
  Lemmas~\ref{l:E-preserves-the-mirage}
  and~\ref{l:xi-preserves-the-mirage}
  guarantee the correctness
  of the choice of the co-domains
  in the definition of both of the continuous
  functors $F\colon \Kar(\Mir_{\pv V}(\Cl X))\to \Kar(\Mir_{\pv V}(\Cl X'))$
  and $G\colon \Kar(\Mir_{\pv V}(\Cl X'))\to \Kar(\Mir_{\pv V}(\Cl X))$
  given by the rules
  \begin{equation*}
    F(e,u,f)=(\ci E(e),\ci E (u),\ci E(f))
    \quad
    \text{and}
    \quad
  G(e,u,f)=(\ci C (e),\ci C (u),\ci C (f)).  
  \end{equation*}  
  We prove the theorem
  by showing that $F$ and $G$ are continuous pseudo-inverses.
  Clearly, $1_{\Kar(\Mir_{\pv V}(\Cl X))}=G\circ F$.

  In the next lines, we use the notation
  $u'$ for the pseudoword $(\be 1(u)^{-1}u)\cdot\te 1(u)^{-1}$,
  where $u$ is an infinite pseudoword. Note that the map
  $u'\mapsto u$ is continuous, by Lemma~\ref{l:continuity-of-cancelation}.
  Suppose that $e$ is an idempotent
  of $\Mir_{\pv V}(\Cl X')$ not
  belonging to the image of $\ci E$.
  Then, by Lemma~\ref{l:possibilities},
  either the first letter of $e$ is $\dia$,
  or the last letter of $e$ is $\alpha$.
  Since every finite factor of~$e$
  belongs to $L(\Cl X')$ and $e=e\cdot e$,
  if follows that in fact both situations happen,
  entailing $e=\dia\, e'\alpha$.
  Note that $e'\in\mathrm{Im}{\ci E}$ according to Lemma~\ref{l:possibilities} and the definition of the pseudoword $e'$.
  Since $\ci C\circ \ci E$ is the identity,
  and $e'\in\mathrm{Im}{\ci E}$,
  we have $e'=\ci E\circ\ci C(e')$, and so
    \begin{equation}\label{eq:splitting-category-free-profinite-is-flow-invariant-1}
      F(G(e))=\ci E(\ci C (\dia\, e'\alpha))
      =\ci E(\ci C(e')\cdot \alpha)
      =\ci E(\ci C(e'))\cdot \ci E(\alpha)
      =e'\alpha\dia.
    \end{equation}
    
    If the idempotent  $e$
    of $\Mir_{\pv V}(\Cl X')$
    belongs to the image of $\ci E$,
    define $\eta_e=(e,e,e)$; if $e\notin \mathrm{Im}{\ci E}$,
    then we define $\eta_e=(e,e\dia,e'\alpha\dia)$.
    Note that in both cases $\eta_e$ is an isomorphism
    of $\Kar(\Mir_{\pv V}(\Cl X'))$, in the second case the inverse being
    $(e'\alpha\dia,e'\alpha e,e)$.

    Let $(e_n)_{n\in\NN}$ be a sequence of idempotents of
    $\Mir_{\pv V}(\Cl X)$, converging to the idempotent~$e$.
    Since $\ci E(A^+)$ is a locally testable set (cf.~Remark~\ref{r:image-of-E}),
    we know that $\ci E(\Om AV)=\overline{\ci E (A^+)}$
    is clopen.
    Therefore, there is $p\in\NN$
    such that either $e_n\in\mathrm{Im}\ci E$
    for all $n\geq p$,
    or $e_n\notin\mathrm{Im}\ci E$
    for all $n\geq p$.
    Since $\lim (e_n)'=e'$,
    by continuity of the operator $u\mapsto u'$,
    we conclude that the mapping $e\mapsto\eta_e$
    is continuous, viewing $\eta_e$ as an element of the space
    $\Om BV\times\Om BV\times \Om BV$.

    Let $(e,u,f)$ be a morphism of $\Kar_{\pv V}(\Mir(\Cl X))$.
    The proof of the theorem is now reduced to
    showing that Diagram~\ref{eq:natural-isomorphism}
    commutes.
    \begin{equation}\label{eq:natural-isomorphism}
    \begin{split}
      \xymatrix@C=1.2cm{
        e&F\circ G(e)\ar[l]_(0.55){\eta_e}\\
        \ar[u]^(0.45){(e,u,f)}f&F\circ G(f)\ar[l]^(0.55){\eta_f}\ar[u]_(0.45){F\circ G(e,u,f)}
      }
    \end{split}
  \end{equation}

    We have several cases to consider:
    \begin{enumerate}
      [label=(\roman*),series=axioms]
    \item Suppose first that $e\in \mathrm {Im} {\ci E}$.
      If $b$ is the first letter of $\ci E^{-1}(e)$,
      then the first letter of $\ci E(b)$
      is the first letter of $u=eu$. Hence,
      the first letter of $u$ is not~$\dia$.
      We have two subcases to consider:\label{item:splitting-category-free-profinite-is-flow-invariant-1}
      \begin{enumerate}
      \item If $f\in \mathrm {Im} {\ci E}$,
        then just as we reasoned for the first letter of $u$,
        we see that the last letter of $u=uf$ is not~$\alpha$.
            Since $u\in\Mir_{\pv V}(\Cl X')$,
      it follows from Lemma~\ref{l:possibilities}
      that $u\in \mathrm {Im} {\ci E}$.
      Therefore, as $\ci E\circ\ci C$ restricts to the identity
      on $\mathrm {Im} {\ci E}$, we have
      $F\circ G(e,u,f)=(\ci E(\ci C(e)),\ci E(\ci C(u)),\ci E(\ci C(f)))=(e,u,f)$. And since in this case $\eta_e=1_e$ and $\eta_f=1_f$,
      the commutativity of Diagram~\ref{eq:natural-isomorphism}
      is immediate.
      \item If $f\notin \mathrm {Im} {\ci E}$, then we have the factorization $f=\dia\, f'\alpha$,
        entailing  $F(G(f))=f'\alpha\dia$
        (cf.~\eqref{eq:splitting-category-free-profinite-is-flow-invariant-1}).
        The last letter of $u=uf$
        is $\alpha$, while first letter is not $\dia$,
        and so by Lemma~\ref{l:possibilities}
        we have $u=\ci E(w)\alpha$ for some $w\in\Om AV$.
        It follows that $\ci E\circ \ci C(u)=\ci E\circ \ci C(\ci E(w)\alpha)=\ci E(w)\alpha\dia=u\dia$ and $F\circ G(e,u,f)=(e,u\dia,f'\alpha\dia)$.
        Then, by the definition of $\eta_f$ when $f\notin \mathrm {Im} {\ci E}$,
        we have
      \begin{align*}
        \qquad\eta_e\circ (F\circ G)(e,u,f)
        &=1_e\circ (e,u\dia,f'\alpha\dia)=(e,uf\dia,f'\alpha\dia)\\
        &=(e,u,f)\circ(f,f\dia,f'\alpha\dia)=(e,u,f)\circ\eta_f,
      \end{align*}
      establishing that Diagram~\ref{eq:natural-isomorphism}
      is commutative in this case.
      \end{enumerate}
    \item Suppose now that $e\notin \mathrm {Im} {\ci E}$,
      so that $e=\dia\, e'\alpha$. 
      As $u=eu\in\Mir_{\pv V}(\Cl X)$, the first letter
      of $u$ must be $\dia$.
      Recall also that $F(G(e))=e'\alpha\dia$
      (cf.~\eqref{eq:splitting-category-free-profinite-is-flow-invariant-1}).
      Again, we have two subcases to consider:
      \begin{enumerate}
      \item If $f\in \mathrm {Im} {\ci E}$,
        then, as seen in case~\ref{item:splitting-category-free-profinite-is-flow-invariant-1}, the last letter of $u=uf$
        is not~$\alpha$. It follows from Lemma~\ref{l:possibilities}
        that $u=\dia\,\ci E(w)$ for some $w\in\Om AV$.
        Then we have
        $\ci E(\ci C(u))=\ci E(\ci C(\dia\,\ci E(w)))
        =\ci E(\ci C(\ci E(w)))=\ci E(w)$.
        On the other hand, $\dia\,\ci E(w)=u=eu=\dia\, e'\alpha u$,
        thus $\ci E(w)=e'\alpha u$ (cf.~Corollary~\ref{c:cancelation-property}). We conclude that $F\circ G(e,u,f)=(e'\alpha\dia,e'\alpha u,f)$,
        thus
        \begin{align*}
      \qquad\qquad\qquad\eta_e\circ (F\circ G)(e,u,f)&=(e,e\dia,e'\alpha\dia)
      (e'\alpha\dia,e'\alpha u,f)\\
      &= (e,e\cdot \underset{=e}{\underbrace{(\dia e'\alpha)}}\cdot u,f)=(e,u,f)=(e,u,f)\circ\eta_f,
    \end{align*}
    proving that Diagram~\ref{eq:natural-isomorphism}
    commutes in this case also.
      \item If $f\notin \mathrm {Im} {\ci E}$, then
        we have the factorization $f=\dia\, f'\alpha$.
        Since the first and last letters of $u=euf$
        are respectively $\dia$ and $\alpha$,
        applying Lemma~\ref{l:possibilities}
        we conclude that
        $u=\dia\,\ci E(w)\alpha$
        for some $w\in\Om AV$.
        Therefore,
        $\ci E\circ \ci C(u)=\ci E(w\alpha)=\ci E(w)\alpha\dia$.
        On the other hand, because $\dia\,\ci E(w)\alpha=u=eu=\dia\, e'\alpha u$,
        we have $\ci E(w)\alpha=e'\alpha u$ by
        Corollary~\ref{c:cancelation-property},
        and so we get $F\circ G(e,u,f)=(e'\alpha\dia,e'\alpha u\dia,f'\alpha\dia)$.
   Finally, we have     
         \begin{align*}
      \qquad&\eta_e\circ (F\circ G)(e,u,f)\\
      &=(e,e\dia,e'\alpha\dia)
      (e'\alpha\dia,e'\alpha u\dia,f'\alpha\dia)
      = (e,e(\dia e'\alpha)u\dia,f'\alpha\dia)\\
      &=(e,u\dia,f'\alpha\dia)=(e,u(f\dia),f'\alpha\dia)=(e,u,f)\circ(f,f\dia,f'\alpha\dia)\\
      &=(e,u,f)\circ\eta_f.
    \end{align*}
      \end{enumerate}
    \end{enumerate}
    With all cases having been exhausted, the proof is concluded.
  \end{proof}

  \begin{Rmk}
    By Lemmas~\ref{l:expansion-shadow-is-also-preserved}
    and~\ref{l:contraction-shadow-is-also-preserved},
    the functors $F$ and $G$
    in the proof of Theorem~\ref{t:splitting-category-free-profinite-is-flow-invariant}
    restrict to graph homomorphisms
    $\Kar(\Sha_{\pv V}(\Cl X))\to \Kar(\Sha_{\pv V}(\Cl X'))$
    and
    $\Kar(\Sha_{\pv V}(\Cl X'))\to \Kar(\Sha_{\pv V}(\Cl X))$,
    respectively.
  \end{Rmk}

  In the appendix section at the end of this paper
  we describe a labeled poset considered in~\cite{ACosta:2006},
  and check that it is encapsulated in $\Kar(\Mir_{\pv V}(\Cl X))$.
  The invariance under flow equivalence
  of such labeled poset then follows from
  Theorem~\ref{t:splitting-category-free-profinite-is-flow-invariant}.
  A direct proof of the conjugacy invariance was given in \cite{ACosta:2006}.
    The description of the labeled poset and the proof of its invariance
    are somewhat technical. The most interesting
  information associated to that labeled poset is
  the following more palatable result, which we next
  easily deduce directly from the proof of Theorem~\ref{t:splitting-category-free-profinite-is-flow-invariant}. 
  
    \begin{Cor}\label{c:splitting-category-free-profinite-is-flow-invariant}
  Suppose that $\Cl X$ is an irreducible subshift.
  Let $\pv V$ be a monoidal pseudovariety
  of semigroups containing $\Lo {Sl}$
  and such that $\pv V=\pv V\ast\pv D$.
  The profinite groups $G_{\pv V}(\Cl X)$ and
  $\widetilde G_{\pv V}(\Cl X)$ are flow equivalence invariants.
\end{Cor}

    \begin{proof}
    By Corollary~\ref{c:invariance-of-GX},
    to show the flow invariance of $\widetilde G_{\pv V}(\Cl X)$
    we only need to check that $\widetilde G_{\pv V}(\Cl X)$ and $\widetilde G_{\pv V}(\Cl X')$
    are isomorphic profinite groups.
    By Theorem~\ref{t:splitting-category-free-profinite-is-flow-invariant},
    there is a continuous equivalence $F\colon \Kar(\Mir_{\pv V}(\Cl X))\to \Kar(\Mir_{\pv V}(\Cl X'))$.
     In every category, the retraction order $\prec$
    is preserved by every equivalence functor,
    and so if $e$ is an idempotent
    in $\widetilde J_{\pv V}(\Cl X)$,
    then $F(e)$ is an idempotent in $\widetilde J_{\pv V}(\Cl X')$,
    by Proposition~\ref{p:retraction}.    
    Also, every continuous equivalence functor
    of compact categories
    preserves the compact group of automorphisms
    in each object, so that $G_e$ and $G_{F(e)}$
    are isomorphic compact groups,
    establishing the flow invariance of~$\widetilde G_{\pv V}(\Cl X)$.

    For what follows we use the specific functor $F\colon \Kar(\Mir_{\pv V}(\Cl X))\to \Kar(\Mir_{\pv V}(\Cl X'))$ given by $F(e,u,f)=(\ci E(e),\ci E (u),\ci E(f))$,
    already met in the proof of Theorem~\ref{t:splitting-category-free-profinite-is-flow-invariant}, where we saw that it is indeed a continuous equivalence.
    By Lemma~\ref{l:expansion-shadow-is-also-preserved},
    if $e$ is an idempotent in $J_{\pv V}(\Cl X)$,
    then $F(e)$ is an idempotent in $J_{\pv V}(\Cl X')$,
    also because of the preservation of the retraction order by equivalence
    functors.
    And as $G_e$ will then be isomorphic
    to $G_{F(e)}$,
    we get the flow invariance of~$G_{\pv V}(\Cl X)$.
  \end{proof}

     \begin{Rmk}
    In the paper~\cite{Almeida&ACosta:2016b}
a sort of geometric interpretation was
given to $G_{\pv S}(\Cl X)$
when~$\Cl X$ is minimal, in which case $G_{\pv S}(\Cl X)=\widetilde G_{\pv S}(\Cl X)$: there it  was shown that the profinite group~$G_{\pv S}(\Cl X)$ is an inverse limit of profinite completions of fundamental groups
in an inverse system of the so called \emph{Rauzy graphs} of $\Cl X$.
A geometric interpretation of
this sort is yet to be obtained in the general case in which~$\Cl X$ is irreducible but may be non-minimal. The approach followed
in~\cite{Almeida&ACosta:2016b} was
based on exploring a profinite semigroupoid (a semigroupoid is a ``category possibly without identities''), there denoted
$\widehat\Sigma_{\infty}(\Cl X)$, and already considered in~\cite{Almeida&ACosta:2007a}, which is determined by the infinite
paths in the free profinite semigroupoid generated
by the inverse limit of the Rauzy graphs of $\Cl X$.
The proof for the geometric interpretation
made in~\cite{Almeida&ACosta:2016b}
included the proof that if $\Cl X$ is minimal then $\Kar(\Mir_{\pv S}(\Cl X))$
and $\widehat\Sigma_{\infty}(\Cl X)$ are isomorphic compact categories. But
that no longer holds if $\Cl X$ is not minimal,
as then $\widehat\Sigma_{\infty}(\Cl X)$
is not a category.
\end{Rmk}
  
  \section{Relationship with the zeta function}
  \label{sec:relat-with-zeta}

  The \emph{orbit} of an element $x$  of $A^\ZZ$ is the set $\mathcal O(x)=\{\sigma^n(x)\mid n\in\ZZ\}$.
  An element $x$ of $A^\ZZ$ is said to be a \emph{periodic} point,
  if $\sigma^n(x)=x$ for some positive integer~$n$, equivalently, if
  $\mathcal O(x)$ is finite.
  A positive integer $n$ such that $\sigma^n(x)=x$ is a \emph{period} of~$x$.
  The \emph{least period} of a periodic point $x$ is the
  smallest positive integer $n$ such that $\sigma^n(x)=x$,
  that is, the least period of such $x$ is the cardinal of~$\mathcal O(x)$. A subshift $\Cl X$ of $A^{\ZZ}$ is said to be a \emph{periodic subshift} if $\Cl X=\mathcal O(x)$
  for some periodic point $x$ of~$A^\ZZ$.
  Every periodic subshift is both minimal and of finite type.

  Given a subshift  $\Cl X$ of $A^\ZZ$,
  we denote by $p_{\Cl X}(n)$ the number of periodic points with period $n$ (i.e., with least period dividing $n$),
  and by $q_{\Cl X}(n)$ the number of periodic points with least period $n$.
   The sequences $(p_{\Cl X}(n))_{n\geq 1}$ and $(q_{\Cl X}(n))_{n\geq 1}$
 determine each other~\cite[Exercise 6.3.1]{Lind&Marcus:1996}.
  The \emph{zeta function} of $\Cl X$,
 defined by
 \begin{equation*}
   \zeta_{\Cl X}(t)=\exp\Bigl(\sum_{n=1}^{+\infty}\frac{p_{\Cl X}(n)}{n}t^n\Bigl)
 \end{equation*}
 encodes the sequence $(p_{\Cl X}(n))_{n\geq 1}$ enumerating
 the number of periods,
 and so it also encodes the sequence $(q_{\Cl X}(n))_{n\geq 1}$ enumerating
 the number of least periods.
 The zeta function is an important
 conjugacy invariant, namely of sofic subshifts (cf.~\cite{Lind&Marcus:1996}).  In this section, we show that the zeta function of~$\Cl X$ is encoded in $\Kar(\Mir_{\pv V}(\Cl X))$ as an invariant of isomorphism of compact categories (Corollary~\ref{c:complexity-function-is-inside-karoubi-envelope}).
  
 Two elements $u$ and $v$ in a semigroup $S$ are said to be \emph{conjugate}, and we write $u\sim_c v$, if
  there are elements $x,y\in S^I$ such that
  $u=xy$ and $v=yx$. For each $u\in S$, the elements~$v$
  such that $v\sim_c u$ are the \emph{conjugates} of $u$.
  In the next few lines, we focus on $S=A^+$, in which case $\sim_c$ is an equivalence relation.
  Indeed, the words conjugate to $u\in A^+$
  are those of the form $v=p^{-1}up$, for some
  prefix $p$ of~$u$. A word $v\in A^+$ is \emph{primitive} if $v=w^k$ implies that $v=w$.  
  Every conjugate of a primitive word is primitive, and the number of conjugates
  of a primitive word~$v$ is the length of $v$. The latter fact may be seen
  as a consequence of one of the most basic properties of combinatorics
  of words (cf.~\cite[Proposition 1.3.2]{Lothaire:1983}):
  if $x,y\in A^*$ are words such that
  $xy=yx$, then there is $z\in A^*$
  such that $x,y\in z^*$. 

  Given a word $v$ of length $n$ of $A^+$, we denote by $v^\infty$
  the unique element $x$ of $A^\ZZ$ such that
  $x_{[0,n-1]}=v$ and $\sigma^n(x)=x$.
  Likewise, we shall also use the notation $v^{+\infty}$
  for the right infinite sequence $x\in A^{\NN}$
  such that $x_{[kn,(k+1)n-1]}=v$ for every $k\geq 0$,
  and $v^{-\infty}$
  will be the left infinite sequence $x\in A^{\ZZ^-}$
  such that $x_{[kn,(k+1)n-1]}=v$ for every $k\leq -1$.
  When $y\in A^{\NN}$ and $x\in A^{\ZZ^-}$,
  we use the notation $z=x.y$
  for $z\in A^{\ZZ}$ such that $z_i=x_i$
  and $z_j=y_j$ for every $i\in\ZZ^-$ and $j\in \NN$.
  Hence, $v^{\infty}=v^{-\infty}.v^{+\infty}$ if $v$ is a word.
  
  The notion of primitive word is useful for dealing with periodic points, because of the following simple fact.
  
   \begin{Fact}\label{fact:periodic-points}
     Let $x$ be a periodic element of $A^{\ZZ}$. Then,
     there is a unique primitive word $v\in A^+$ such that $x=v^\infty$.
     Moreover, we have the equality $\mathcal O(x)=\{u^\infty\mid u\sim_c v\}$,
     and the least period of $x$ is the length of $v$.
   \end{Fact}

   We collect some more properties of primitive words.
 
    \begin{Lemma}\label{l:v-plus-locally-testable}
   If $v$ is a primitive word of $A^+$, then the language $v^+$
   is locally testable.
 \end{Lemma}

 \begin{proof}
   Let $\Cl X=\mathcal O (v^\infty)$.   
   Then, denoting by $[v]_{\sim_c}$ the $\sim_c$-class of $v$, we have the equality
   \begin{equation*}
     v^+=(L(\Cl X)\setminus A^{(<|v|)})\setminus
     \bigcup_{u,w\in [v]_{\sim_c}\setminus\{v\}}(uA^*\cup A^*w).
   \end{equation*}
   Since $\Cl X$ is of finite type, the
   language $L(\Cl X)$ is locally testable.
   As $A^{(<|v|)}$, $uA^*$
   and $A^*w$ are also locally testable,
   we conclude that $v^+$ is locally testable.
 \end{proof}
 
 Lemma~\ref{l:v-plus-locally-testable}
 may be seen as an application of the main
 result of~\cite{Restivo:1974}, a more general result
 about very pure codes (see also~\cite[Proposition 7.1.1]{Berstel&Perrin&Reutenauer:2010}).
 
   \begin{Lemma}\label{l:a-technical-consequence-of-primitivity}
   If $v$ is a primitive word of $A^+$ with length $n$,
   then the inclusion
   $ v^*\cdot v^2\cdot A^{(<n)}\cap A^*\cdot v^2\subseteq v^+$
   holds.
 \end{Lemma}
 
 \begin{proof}
   Let $w\in v^*\cdot v^2\cdot A^{(<n)}\cap A^*\cdot v^2$.
   Then $w=v^kq$ for some $k\geq 2$ and some (possibly empty) word $q$
   of length at most $n-1$, and $v^2$ is a suffix of $v^2q$.
   We are reduced to showing that $q=\varepsilon$.
   Take the word $p$ such that $v^2q=pv^2$.
   Then we have $|p|=|q|<|v|$, and $v=pv'=v''q$ for some $v',v''\in A^+$
   such that $|v'|=|v''|$. As the following chain of equalities 
   \begin{equation*}
     p\cdot v''\cdot qv=p\cdot (v''q)\cdot v=pv^2=v^2q=p\cdot v'\cdot vq
   \end{equation*}
   holds, comparing the extremities of the chain, we deduce from $|v''|=|v'|$
   that $qv=vq$. Therefore, by the aforemention property
   of commuting words,  one concludes that $v,q\in w^*$ for some word $w$.
   But as $v$ is primitive and $|q|<|v|$, one must have $w=v$ and~$q=\varepsilon$.
 \end{proof}

 We remark, \emph{en passant}, that the word $v^2$ is really relevant
 in Lemma~\ref{l:a-technical-consequence-of-primitivity}.
 More precisely, the inclusion 
 $v^*\cdot v\cdot A^{(<|v|)}\cap A^*\cdot v\subseteq v^+$
 fails, for example, for $A=\{a,b\}$ and the primitive word $v=bab$,
 since $(bab)ab=ba(bab)$ is not a power of $bab$.
   
   We turn now our attention to pseudowords.
   Let $\pv V$ be a pseudovariety of semigroups
   containing $\Lo I$. Suppose that $u\in\Om AV\setminus A^+$.
   We denote by $\ori u$ the unique element $x=(x_i)_{i\in \NN}$ of $A^{\NN}$
   such that
   $x_{[0,n-1]}$ is the prefix of length $n$ of $u$, whenever~$n$
   is a positive integer.
   We say that $\ori u$ is the \emph{positive ray} of $u$.
   Symmetrically, the \emph{negative ray} of $u$,
   denoted $\ole u$, is the unique element $x=(x_i)_{i\in \ZZ^-}$ of $A^{\ZZ^-}$ such that
   $x_{[-n,-1]}$ is the suffix of length $n$ of $u$, whenever $n$ is a positive integer. Let $u$ and $v$ be elements of
   $\Om AV\setminus A^+$. Note that if $u=vw$ for some $w\in\Om AV^I$,
   then $\ori u=\ori v$, but the converse is not true: $u=a^\omega b$
   and $v=a^\omega c$ are such that $\ori u=\ori v$, but neither $u\leq_\R v$
   nor $v\leq_\R u$. In contrast, we have the following proposition.   

   \begin{Prop}[{\cite[Lemma~6.6]{Almeida&ACosta:2007a}
       and \cite[Lemma 5.3]{Almeida&ACosta:2012}}]\label{p:parametrization-of-JX}
     Consider a pseudovariety of semigroups~$\pv V$ containing
     $\Lo {Sl}$.
     Let $\Cl X$ be a minimal subshift. For every $u,v\in J_{\pv V}(\Cl X)$,
     the equivalences
     \begin{equation*}
       u\mathrel{\R}v\Leftrightarrow \ori u=\ori v\qquad
     \text{and}\qquad u\mathrel{\L}v\Leftrightarrow\ole u=\ole v
     \end{equation*}
     hold,    and therefore so does the equivalence
     \begin{equation*}
       u\mathrel{\H}v\Leftrightarrow\li u=\li v.
     \end{equation*}
     Moreover, the $\H$-class of $u\in J_{\pv V}(\Cl X)$ is
     a maximal subgroup of $J_{\pv V}(\Cl X)$
     if and only if $\li u\in \Cl X$.
   \end{Prop}

   In other words, Proposition~\ref{p:parametrization-of-JX} states
   in particular that if $\Cl X$ is minimal then
   the $\R$-classes and the $\L$-classes
   of $J_\pv V(\Cl X)$ are respectively parameterized
   by the positive rays of $\Cl X$ and the negative rays of $\Cl X$, provided
   $\pv V$ contains $\Lo {Sl}$.
   
   \begin{Cor}\label{c:number-of-RLH-classes-in-minimal-case}
     Consider a pseudovariety of semigroups $\pv V$ containing
     $\Lo {Sl}$.
     Let $\Cl X$ be a minimal subshift.
     If $\Cl X$ is not a periodic subshift, then $J_{\pv V}(\Cl X)$
     contains $2^{\aleph_0}$ many $\R$-classes
     and $2^{\aleph_0}$ many $\L$-classes.
     If $\Cl X$ is a periodic subshift of least period $n$,
     then $\Cl X$ contains precisely $n$ $\R$-classes, $n$ $\L$-classes,
     $n^2$ $\H$-classes and $n$ idempotents, and these
     idempotents are the pseudowords
     of the form $u^\omega$ with $u$ a conjugate word
     of~$v$, where $v$ is a primitive word
     of length $n$ such that $\Cl X=\mathcal O(v^\infty)$.
   \end{Cor}

   \begin{proof}
     It suffices to combine Proposition~\ref{p:parametrization-of-JX}
     with the following facts that we recall.
     First, it is known that
     a nonperiodic minimal subshift has $2^{\aleph_0}$ many negative rays,
     and $2^{\aleph_0}$ many positive rays (cf.~\cite[Chapter 2]{Lothaire:2001}).
     Second, if we assume that $\Cl X$ is a periodic subshift of period~$n$,
     with $\Cl X=\mathcal O(v^\infty)$ for some primitive word
     $v$ of length~$n$, then it is clear that $\Cl X$  has $n$ positive rays,
     namely those of the form $u^{-\infty}$
     with $u$ a conjugate of $v$.
     And whenever $u$ and $w$ are conjugates of $v$,
     one has $u^{-\infty}.w^{+\infty}\in\Cl X$
     if and only if $u=w$, since periodic shifts are minimal and hence Proposition~\ref{p:parametrization-of-JX} applies.
     Finally, if $u$ is conjugate with the primitive word $v$,
     then $u^\omega$ is an idempotent in $J_{\pv V}(\Cl X)$,
     the one in the unique maximal subgroup of
     $J_{\pv V}(\Cl X)$ whose elements have negative ray
     $u^{-\infty}$ and positive ray $u^{+\infty}$.
   \end{proof}

   For later reference, we state the
   next well known and easy to prove lemma.

   \begin{Lemma}\label{l:conjugation-of-idempotents}
     Suppose that $xy$ is an idempotent in a semigroup $S$,
     and consider the conjugate $yx$.
     Then $(yx)^2$ is an idempotent of $S$
     which is $\J$-equivalent to $xy$.
   \end{Lemma}
   
   Next is
   another well known fact
   (cf.~\cite[Propositions A.1.15 and 3.1.10]{Rhodes&Steinberg:2009qt})
   that we shall use.
   
      \begin{Lemma}\label{l:conjugation-in-a-same-J-class}
         In a compact semigroup, every two $\J$-equivalent
         idempotents are conjugate.
       \end{Lemma}

  In what follows, $J_e$ denotes the $\J$-class of $e$.

 \begin{Prop}\label{p:J-finite-number-idempotents}
   Let $\pv V$ be a pseudovariety of semigroups
   containing~$\Lo{Sl}$.
   Let $e$ be an idempotent of $\Om AV$.
   The following conditions are equivalent:
   \begin{enumerate}
   \item $e=u^\omega$ for some $u\in A^+$;\label{item:J-finite-number-idempotents-1}
   \item $J_e$ contains a finite number of $\H$-classes;\label{item:J-finite-number-idempotents-2}
   \item $J_e$ contains a finite number of $\R$-classes;\label{item:J-finite-number-idempotents-3}
   \item $J_e$ contains a finite number of $\L$-classes;\label{item:J-finite-number-idempotents-4}
   \item $J_e$ contains a finite number of idempotents.\label{item:J-finite-number-idempotents-5}
   \end{enumerate}
 \end{Prop}

 In the following proof of Proposition~\ref{p:J-finite-number-idempotents}
 we use profinite powers $u^\nu$, with $\nu$ belonging to the
 profinite completion $\widehat\NN$ of $\NN$ (details may be found in~\cite[Section 2]{Almeida&Volkov:2006}).
 The power $u^\omega$ is an example of such powers, with $\omega=\lim n!$ in $\widehat\NN$. What is most relevant for the proof is that, for every $u\in A^+$,
 the power $u^\nu$  belongs to the maximal subgroup of $\Om AV$ containing $u^\omega$ if and only if $\nu\in\widehat\NN\setminus\NN$.
 
 \begin{proof}[Proof of Proposition~\ref{p:J-finite-number-idempotents}]
   The implication
   (\ref{item:J-finite-number-idempotents-1})
   $\Rightarrow$
   (\ref{item:J-finite-number-idempotents-2})
   is encapsulated in Corollary~\ref{c:number-of-RLH-classes-in-minimal-case}.
   The implications 
   (\ref{item:J-finite-number-idempotents-2})
   $\Rightarrow$
   (\ref{item:J-finite-number-idempotents-3})
   and
   (\ref{item:J-finite-number-idempotents-2})
   $\Rightarrow$
   (\ref{item:J-finite-number-idempotents-4})
   follow immediately from each $\R$-class and each $\L$-class being
   a union of $\H$-classes.

   (\ref{item:J-finite-number-idempotents-3})
   $\Rightarrow$
   (\ref{item:J-finite-number-idempotents-1}):
   Suppose there is no $u\in A^+$
   such that $e=u^\omega$.
   Let $f$ be a $\J$-maximal idempotent such that
   $e\leq_\J f$. Such an idempotent $f$ exists,
   as mentioned in Remark~\ref{rmk:always-a-factor}.
   Let $x,y\in\Om AV$ be such that $e=xfy$.
   Since $fy\cdot xf$ is a conjugate of $xf\cdot fy=e$,
   the pseudoword $h=(fyxf)^2$ is an idempotent in $J_e$, by Lemma~\ref{l:conjugation-of-idempotents}.
   Let $f'$ be an idempotent in $J_f$.
   Then $f=zt$ and $f'=tz$ for some $z,t\in \Om AV$
   (Lemma~\ref{l:conjugation-in-a-same-J-class}).
   Since $e=xfzf'ty$, we know that
   $h'=(f'tyxfz)^2$ is an idempotent in $J_e$ (Lemma~\ref{l:conjugation-of-idempotents}).
   By Proposition~\ref{p:parametrization-of-JX}, if $f$ and $f'$ are not $\R$-equivalent  then $\ori f\neq\ori f$.
   Since $h=fh$ and $h'=f'h'$,
   the inequality
   $\ori f\neq\ori f$
   in turn implies the inequality
   $\ori h\neq\ori h'$.
   This shows that if~$f$ and $f'$ are not $\R$-equivalent,
   then $h$ and $h'$ are not $\R$-equivalent,
   and so $J_h$ has at least as many $\R$-classes
   as $J_f$ has.
   By Corollary~\ref{c:number-of-RLH-classes-in-minimal-case},
   if $f$ is not of the form $v^\omega$, then $J_f$
   has $2^{\aleph_0}$ $\R$-classes, and
   so $J_e=J_h$ has at least $2^{\aleph_0}$ $\R$-classes.  

   From hereon, we suppose that $f=v^\omega$ for some word $v\in A^+$,
   which we may as well assume to be primitive.
   Since $h=v^\omega h v^\omega$, one has
   \begin{equation}\label{eq:J-finite-number-idempotents-0}
     h\in\overline{v^2\cdot v^+\cdot A^*\cap A^*\cdot v^+\cdot v^2},
   \end{equation}
   a fact which is the base of the reasoning that follows.   
   Let $n=|v|$.  For each $z\in A^n$,
   consider the language $K_z=v^+\cdot z\cdot A^*$.
   Note that, 
   \begin{equation}\label{eq:J-finite-number-idempotents-1}
     v^2\cdot v^+\cdot A^*\cap A^*\cdot v^+\cdot v^2\subseteq \Biggl[\,\bigcup_{z\in A^n\setminus\{v\}} K_z\,\Biggr]\cup (v^+\cdot v^2\cdot A^{(<n)}\cap A^*\cdot v^2).
   \end{equation}
   Since $v$ is primitive,
   we know by Lemma~\ref{l:a-technical-consequence-of-primitivity}
   that the inclusion
   \begin{equation}\label{eq:J-finite-number-idempotents-2}
     v^+\cdot v^2\cdot A^{(<n)}\cap A^*\cdot v^2\subseteq v^+
   \end{equation}
   holds.
   Combining~\eqref{eq:J-finite-number-idempotents-0}, \eqref{eq:J-finite-number-idempotents-1}
   and \eqref{eq:J-finite-number-idempotents-2},
      and noticing that the family $(K_z)_z$ is finite,
   we conclude that
   \begin{equation*}
     h\in\Biggl[\,\bigcup_{z\in A^n\setminus\{v\}}\overline{K_z}\,\Biggr]\cup \overline{v^+}.
   \end{equation*}

   If $h\in\overline{v^+}$, then $h$ is the unique idempotent $v^\omega$
   in $\overline{v^+}$, thus $e\in J_{v^\omega}$.
   By Corollary~\ref{c:number-of-RLH-classes-in-minimal-case},
   this contradicts our assumption
   that~$e$ is not of the form $u^\omega$ with $u\in A^+$.

   Therefore, we may take $z\in A^n\setminus\{v\}$
   such that $h\in\overline{K_z}$.
   Take a sequence $(h_k)_k=(v^{r_k}zw_k)_k$  of words
   of $K_z$ converging to $h$, with $r_k\geq 1$.
   By taking subsequences, we may as well suppose that
   $(v^{r_k})_k$ and $(w_k)_k$ respectively converge to some pseudowords $v^\alpha$
   and $w$ of $\Om AV^I$, with $\alpha\in\widehat \NN$,
   thanks to the compactness of $\Om AV$ and $\widehat\NN$.
   Note that $h=v^\alpha\cdot z\cdot w$.
   If $\alpha\in\NN$, then
   $v^\alpha\cdot z$
   is the prefix of length $(\alpha+1)\cdot n$
   of $h$. But since $h=v^\omega h$,
   the prefix of length $(\alpha+1)\cdot n$
   of $h$ is actually $v^{\alpha+1}$, and so we reached a contradiction
   with $v\neq z$.  To avoid the contradiction, we must have
   $\alpha\in\widehat{\NN}\setminus\NN$.
   Therefore, for each positive integer $k$,
   we may consider the pseudoword
   $g_k=(v^k\cdot z\cdot w\cdot v^{\alpha-k})^2$,
   which is an idempotent $\J$-equivalent to $h$ (Lemma~\ref{l:conjugation-of-idempotents}).
   If $k<\ell$, then the prefix of length $(k+1)n$ of $g_\ell$ is $v^{k+1}$,
   while the prefix of the same length of  $g_k$ is $v^kz\neq v^{k+1}$.
   Hence, we conclude that $g_k$ and $g_\ell$ are not $\R$-equivalent
   whenever $k\neq\ell$,
   thus showing that $J_e=J_h$ has at least $\aleph_0$ $\R$-classes.
   
   (\ref{item:J-finite-number-idempotents-4})
   $\Rightarrow$
   (\ref{item:J-finite-number-idempotents-1}):
   This implication holds with a proof
   entirely symmetric to the proof
   of the implication
   (\ref{item:J-finite-number-idempotents-3})
   $\Rightarrow$
   (\ref{item:J-finite-number-idempotents-1}).

   At this point, we have established the equivalences
   (\ref{item:J-finite-number-idempotents-1})
   $\Leftrightarrow$
   (\ref{item:J-finite-number-idempotents-2})
   $\Leftrightarrow$
   (\ref{item:J-finite-number-idempotents-3})
   $\Leftrightarrow$
   (\ref{item:J-finite-number-idempotents-4}).
   The implication (\ref{item:J-finite-number-idempotents-1})
   $\Rightarrow$
   (\ref{item:J-finite-number-idempotents-5})
   is also encapsulated in Corollary~\ref{c:number-of-RLH-classes-in-minimal-case}. Finally, the implication
   (\ref{item:J-finite-number-idempotents-5})
   $\Rightarrow$
   (\ref{item:J-finite-number-idempotents-1})
   follows from the well-know fact that, in a stable semigroup,
   every $\R$-class contained in a regular $\J$-class
   contains at least one idempotent.   
 \end{proof}
 
 \begin{Cor}\label{c:complexity-function-is-inside-karoubi-envelope}
   Let $\Cl X$ be a subshift of $A^\ZZ$.
   Suppose that $\pv V$ is a pseudovariety of semigroups containing
   $\Lo {Sl}$. 
   Then $q_{\Cl X}(n)$ is the number of objects of the
   category~$\Kar(\Mir_{\pv V}(\Cl X))$
   whose isomorphism class is a set of cardinal $n$.
 \end{Cor}

 \begin{proof}
   Let $P$ be the set of primitive words with length $n$
   belonging to $L(\Cl X)$.
   Then the mapping $u\mapsto u^\infty$
   is a bijection between
   $P$ and the set of periodic points of~$\Cl X$ with least period $n$.
   Moreover, the mapping $\psi\colon u^\infty\mapsto u^\omega$, with $u\in P$,
   is injective, and
   for every $u\in P$ the idempotent $u^\omega$ is an object of $\Kar(\Mir_{\pv V}(\Cl X))$  whose isomorphism class is a set with $n$ elements
   (cf.~Proposition~\ref{p:parametrization-of-JX} and Corollary~\ref{c:number-of-RLH-classes-in-minimal-case}).

   On the other hand,
   by Proposition~\ref{p:J-finite-number-idempotents},
   if $e$ is an object of $\Kar(\Mir_{\pv V}(\Cl X))$
   whose isomorphism class has $n$ elements, then $e=u^\omega$
   for some primitive word $u\in L(\Cl X)$, which, by Corollary~\ref{c:number-of-RLH-classes-in-minimal-case}, has length $n$. Therefore, the image
   of the injective map $\psi$ is the set of objects of the category 
   $\Kar(\Mir_{\pv V}(\Cl X))$  whose isomorphism class is a set of cardinal $n$.
 \end{proof}

 The following
 perspective about zeta functions
 is immediate from
 Corollary~\ref{c:complexity-function-is-inside-karoubi-envelope}.
 
 \begin{Cor}\label{c:complexity-function-is-inside-karoubi-envelope}
   Let $\Cl X$
   and $\Cl Y$ be subshifts such that
   $\Kar(\Mir_{\pv V}(\Cl X))$
   and $\Kar(\Mir_{\pv V}(\Cl Y))$
   are isomorphic, where $\pv V$ is a pseudovariety of semigroups containing
   $\Lo {Sl}$. 
   Then we have $\zeta_{\Cl X}=\zeta_{\Cl Y}$.
 \end{Cor}

 \appendix

 \section{A labeled topological poset}
 \label{sec:append-label-poset}

 Here a \emph{topological poset} $T$ is a partially ordered set $T$
 such that $T$ is a topological space and
 the partial order $\leq$ of $T$ is a closed subset
 of $T\times T$.
 A \emph{labeled} topological poset
 is a topological poset $T$
 together with a labeling map $\lambda$,
 of domain $T$, assigning
 to each element $t$ of $T$ its \emph{label}, denoted $\lambda(t)$.

 Consider labeled topological posets $T$ and $R$,
 respectively with partial orders $\leq_T$
 and $\leq_R$, and labeling maps
 $\lambda_T$ and $\lambda_R$.
 An isomorphism of labeled topological posets between $T$
 and $R$ is a homeomorphism
 $\varphi\colon T\to R$
 that preserves orders
 (that is, $t_1\leq_T t_2\Leftrightarrow \varphi(t_1)\leq_R\varphi(t_2)$ for
 every $t_1,t_2\in T$)
 and labels (that is, $\lambda_R(\varphi(t))=\lambda_T(t)$ for
 every $t\in T$).
 Naturally, $T$ and $R$ are said to be isomorphic labeled topological posets
 when such an isomorphism exists.

 An element $s$ of a semigroup $S$ is said to have \emph{local units} in $S$ if $s=esf$ for some idempotents $e,f$ of $S$.

 Suppose that $s\in S$ has local units and let $t\in S$ be $\J$-equivalent
 to $s$. Since $s\in SsS$,
 there are $x,y\in S$ such that $t=xty$, whence $t=x^kty^k$
 for every $k\geq 1$. Since in a compact semigroup the closure
 of a monogenic semigroup contains an idempotent~\cite[Theorem 3.5]{Carruth&Hildebrant&Koch:1983},
 we conclude that $t$ also has local units.
 Therefore, in a compact semigroup, the set of local units is a union of $\J$-classes.
 
 For each subset $K$ of a semigroup $S$, we denote by $LU(K)$
 the set of elements of $K$ which have local units in $S$.
 Suppose that $S$ is a compact semigroup
 and that $K$ is a closed subset of $S$ which is factorial.
 We associate to $LU(K)$ a labeled topological poset, denoted by
 $LU(K)^\dagger$,
 as follows:
 \begin{enumerate}
 \item The underlying space of $LU(K)^\dagger$ is the quotient
   of the space $LU(K)$ by the restriction to $LU(K)$ of the relation $\J$.
   In other words, the underlying space is
   the space of $\J$-classes contained in $LU(K)$.
    \item One has $J_1\leq J_2$ in $LU(K)^\dagger$
      if and only if $u\leq_{\J} v$
      for some (equivalently, for all) elements $u\in J_1$
      and $v\in J_2$, 
      whenever $J_1$ and $J_2$ are $\J$-classes
      contained in $LU(K)$.
    \item The label of each regular $\J$-class
      contained in $K$ is the pair
      $(\varepsilon,\Gamma(J))$ such that
      $\varepsilon=1$ if $J$ is regular and
      $\varepsilon=0$ if $J$ is not regular,
      and $\Gamma(J)$ is the isomorphism class of the Sch\"utzenberger group
      of $J$, as a compact group. 
    \end{enumerate}

     \begin{Prop}\label{p:invariance-of-labeled-posets}
      Let $\Cl X,\Cl Y$ be subshifts for which there is
      a continuous equivalence functor
      $F\colon\Kar(\Mir_{\pv V}(\Cl X))\to \Kar(\Mir_{\pv V}(\Cl Y))$.
      Then the labeled topological posets
      $LU(\Mir_{\pv V}(\Cl X))^\dagger$
      and $LU(\Mir_{\pv V}(\Cl Y))^\dagger$
      are isomorphic.
      If, moreover, $F$
      is such that
      the inclusion 
      $F(\Kar(\Sha_{\pv V}(\Cl X)))\subseteq \Kar(\Sha_{\pv V}(\Cl Y))$
      holds, and for some continuous pseudo-inverse $G$ of $F$,
      the inclusion $G(\Kar(\Sha_{\pv V}(\Cl Y)))\subseteq \Kar(\Sha_{\pv V}(\Cl X))$ also holds, then the
      labeled topologically posets $LU(\Sha_{\pv V}(\Cl X))^\dagger$
      and $LU(\Sha_{\pv V}(\Cl Y))^\dagger$
      are isomorphic.
    \end{Prop}

    The proof of Proposition~\ref{p:invariance-of-labeled-posets}
    will be later deduced as a consequence
    of some intermediate technical results.

        \begin{Rmk}
      A reader familiar with the paper \cite{ACosta&Steinberg:2015}
      will note the similarity of Proposition~\ref{p:invariance-of-labeled-posets} with the Theorem 6.3 from~\cite{ACosta&Steinberg:2015}, which
      concerns labeled posets, without topology, with
      the Sch\"utzenberger groups in the labels not being viewed as
      topological groups.
      But the techniques of~\cite{ACosta&Steinberg:2015}
      are not suitable for the topological ingredients which we add here.
      Indeed, one crucial
      step of the approach made in~\cite{ACosta&Steinberg:2015} consists of the following:
      for each element $s$
      with local units in a semigroup~$S$, choose idempotents $e_s,f_s$
      such that $s=e_ssf_s$. There is no reason to expect
      continuous choices $s\mapsto e_s$ and $s\mapsto f_s$.
    \end{Rmk}

    Combining Proposition~\ref{p:invariance-of-labeled-posets}
    with Theorems~\ref{c:splitting-category-free-profinite-is-conj-invariant}
    and~\ref{t:splitting-category-free-profinite-is-flow-invariant},
    one immediately gets the following
    consequence, which is the reason for
    this appendix.
    
    \begin{Cor}\label{c:invariance-of-labeled-posets}
      Let $\pv V$ be a pseudovariety
  of semigroups containing $\Lo {Sl}$
  and such that $\pv V=\pv V\ast\pv D$.
  The labeled topological posets
  $LU(\Mir_{\pv V}(\Cl X))^\dagger$
  and
  $LU(\Sha_{\pv V}(\Cl X))^\dagger$
  are conjugacy invariants,
  and they are invariants of flow equivalence if $\pv V$ is monoidal.
    \end{Cor}

    \begin{Rmk}
      Note that, when $\Cl X$ is irreducible, the conjugacy/flow invariance
      of $\widetilde G_{\pv V}(\Cl X)$
      and of $G_{\pv V}(\Cl X)$,
      stated in Corollaries~\ref{c:invariance-of-GX}
      and~\ref{c:splitting-category-free-profinite-is-flow-invariant},
      may also be derived from  Corollary~\ref{c:invariance-of-labeled-posets}, because $\widetilde J(\Cl X)$ is the minimum element of $LU(\Mir_{\pv V}(\Cl X))^\dagger$
      and $J(\Cl X)$ is the minimum element
      of $LU(\Sha_{\pv V}(\Cl X))^\dagger$.
    \end{Rmk}

      In what follows, when we refer to a ``category'', we mean a ``small category''.    
The Green relations on the set of morphisms of a category $C$
may be defined by adapting in a direct and natural manner the usual definitions of the Green relations
on a monoid.
Alternatively,
one may use a classical construction, the semigroup~$C_{cd}$, which is called the \emph{consolidation}
of $C$. The elements of $C_{cd}$
are the morphisms of $C$ together with an extra element $0$, which
is as zero of~$C_{cd}$. For any morphisms $\varphi,\psi$ of $C$,
the product $\varphi\psi$ in $C_{cd}$
equals the composition $\varphi\circ\psi$ when $d(\varphi)=r(\psi)$,
and equals~$0$ when $d(\varphi)\neq r(\psi)$.
Then, for each $\mathcal K\in\{\R,\L,\D,\H,\J,\leq_{\R},\leq_{\L},\leq_{\J}\}$,
one has $\varphi\mathrel{\mathcal K}\psi$ in $C$
if and only if $\varphi\mathrel{\mathcal K}\psi$ in $C_{cd}$,
for all morphisms $\varphi,\psi$ of $C$.

\begin{Lemma}\label{l:green-relations-are-closed}
  In a compact category $C$, each relation $\R,\L,\D,\H,\J,\leq_{\R},\leq_{\L},\leq_{\J}$ is a closed set of $\Mor(C)\times \Mor(C)$. Moreover, the
  $\K$-classes of morphisms are closed sets,
  for each $\mathcal K\in\{\R,\L,\D,\H,\J\}$.
\end{Lemma}

\begin{proof}
  Consider a net $(\varphi_i,\psi_i)_{i\in I}$
  of morphisms of $C$, converging to $(\varphi,\psi)$,
  such that $\varphi_i\leq_{\J} \psi_i$ for all $i\in I$.
  Then we have factorizations $\varphi_i=\alpha_i\circ\psi_i\circ\beta_i$
  for some nets $(\alpha_i)_{i\in I}$ and $(\beta_i)_{i\in I}$
  of morphisms of $C$.
  As $C$ is compact, we may take a cluster point
  $(\alpha,\beta)$ of $(\alpha_i,\beta_i)_{i\in I}$. By continuity of the composition, we get $\varphi=\alpha\circ\psi\circ\beta$,
  thus $\varphi\leq_\J\psi$. This proves that $\leq_\J$ is closed.
  The proofs for the other relations are similar.
  Since each class of a closed equivalence
  relation in a compact space
  is a closed set (cf.~\cite[Exercise 3.17]{Rhodes&Steinberg:2009qt}), we are done.
\end{proof}

As another expression of the link between $C$
and $C_{cd}$, a morphism of $C$
is said to be regular when it is a regular element of $C_{cd}$,
and a $\J$-class of $C$ is regular when it is a regular
$\J$-class of $C_{cd}$.

Let $H$ be an $\H$-class of morphisms of the category $C$.
Note that $H\subseteq C(c,d)$ for some objects $c,d$.
The \emph{Sch\"utzenberger group of $H$ in $C$}, denoted $\Gamma(H)$,
is the quotient of
the monoid
\begin{equation*}
T(H)=\{\alpha\in C(c,c) \mid H\circ\alpha\subseteq H\}  
\end{equation*}
by the monoid congruence $\approx_H$ on $T(H)$ given,
whenever $\alpha,\beta\in T(H)$, by
\begin{equation*}
\alpha\approx_H\beta\Leftrightarrow
[\forall \varphi\in H:\varphi\circ\alpha=\varphi\circ\alpha],
\end{equation*}
or, in what is easily seen to be an equivalent formulation,
\begin{equation*}
\alpha\approx_H\beta\Leftrightarrow
[\exists \varphi\in H:\varphi\circ\alpha=\varphi\circ\alpha].
\end{equation*}
The useful equality 
$T(H)=\{\alpha\in C(c,c) \mid H\circ\alpha\cap H\neq \emptyset\}$ is also easy to check.
The monoid quotient $T(H)/{\approx_H}$ is indeed a group; clearly, it coincides with the classical Sch\"utzenberger group in $C_{cd}$ of $H$, if we view $H$ as an $\H$-class of $C_{cd}$.
Moreover, if~$C$ is a compact category, then
$T(H)$ is a closed submonoid of $C(c,c)$
and $\approx_H$ is a closed relation, by the same arguments
used in the proof of Lemma~\ref{l:green-relations-are-closed},
and then the quotient $\Gamma(H)=T(H)/{\approx_H}$
becomes a compact group (cf.~\cite[Theorem 1.54]{Carruth&Hildebrant&Koch:1983}).
It is this compact group that will be for us
the \emph{Sch\"utzenberger group of $H$ in $C$}, when $C$ is a compact category.

\begin{Rmk}
  If $H$ contains some idempotent (which implies $c=d$),
  then $\Gamma(H)$ is isomorphic to $H$, via the mapping
  $\varphi\in H\mapsto [\varphi]_{\approx_H}\in\Gamma(H)$,
  and this mapping is continuous if $C$ is a compact category.
\end{Rmk}

\begin{Lemma}
  If $C$ is a profinite category, then $\Gamma(H)$ is a profinite group.
\end{Lemma}

\begin{proof}
  Let $C=\varprojlim_{i\in I} C_i$ be an inverse limit of finite categories. Let
  $\varphi,\alpha,\beta$ be morphisms of $C$
  with $\varphi\circ\alpha\mathrel{\H}\varphi\circ\beta\mathrel{\H}\varphi$
  and $\varphi\circ\alpha\neq\varphi\circ\beta$.
  For $\psi\in C$, denote by $\psi_i$
  its projection on $C_i$, where $i\in I$.
  Let $H$ be the $\H$-class of $\varphi$ and $H_i$ be the $\H$-class of
  $\varphi_i$.
  Take $i_0\in I$ such that $\varphi_{i_0}\circ\alpha_{i_0}
  \neq
  \varphi_{i_0}\circ\beta_{i_0}$.
  We have a well defined continuous homomorphism
  $\Gamma(H)\to\Gamma(H_{i_0})$
  given by the assignment $[\gamma]_{\approx_H}\mapsto [\gamma_i]_{\approx_{H_i}}$,
  and we also know that $[\alpha_{i_0}]_{\approx_{H_{i_0}}}\neq [\beta_{i_0}]_{\approx_{H_{i_0}}}$.
  Therefore $\Gamma(H)$ is profinite: in fact $\Gamma(H)$
  embeds as a closed subgroup of
  the natural inverse limit $\varprojlim_{i\in I} \Gamma(H_i)$. 
\end{proof}

We may use for compact semigroups some of the notation employed for compact categories. For example, $\Gamma(H)$ is the Sch\"utzenberger group (as a compact group)
of an $\H$-class $H$ of a compact semigroup $S$.
In fact, when $S$ is a compact semigroup, we may view $S^I$ as compact category with a unique object and the elements of $S^I$ as the morphisms.

  For any compact category $C$,
  the compact  groups $\Gamma(H)$ and $\Gamma(K)$ are isomorphic
  when $H$ and $K$ are $\H$-classes of morphisms of $C$
  contained in the same $\J$-class~(see the proof of \cite[Theorem 3.61]{Carruth&Hildebrant&Koch:1983}).  
Hence, when in a compact category $C$, we may associate
to each $\J$-class $J$ its Sch\"utzenberger group $\Gamma(J)$, which is
(the isomorphism class of) the Sch\"utzenberger group of any of
the $\H$-classes contained in $J$.
We next consider the labeled topological poset~$C^\dagger$ defined by:
    \begin{enumerate}
    \item the underlying space is the quotient space $\Obj(C)/{\J}$;
    \item one has $J_1\leq J_2$ in $C^\dagger $
      if and only if $\varphi\leq_{\J}\psi$
      for some (equivalently, for all) morphisms $\varphi\in J_1$
      and $\psi\in J_2$;
    \item the label of each element $J$ of $C^\dagger$ is the pair
      $(\varepsilon,\Gamma(J))$ such that
      $\varepsilon=1$ if $J$ is regular and
      $\varepsilon=0$ if $J$ is not regular, where $\Gamma(J)$
      is taken as an isomorphism class of compact groups.
    \end{enumerate}
    
    Let $F\colon C\to D$ be a continuous functor between compact categories.
    We define a map $F^\dagger\colon C^\dagger\mapsto D^\dagger$
    by letting $F^\dagger([\varphi]_\J)=[F(\varphi)]_\J$.
    This map is well defined,
    indeed it is immediate that $\varphi\leq_\J\psi$
    in $C$ implies $F(\varphi)\leq_\J F(\psi)$ in $D$.
    Note also that $F^\dagger$ is continuous,
    because it is the map $\Obj(C)/{\J}\to \Obj(D)/{\J}$
    naturally induced by the continuous map $F\colon \Obj(C)\to \Obj(D)$,
    and we are dealing with compact quotients of compact
      spaces (cf.~\cite[Theorems 9.2 and 9.4]{Willard:1970}). 
    
    \begin{Prop}\label{p:functoriality-dagger}
      Let $C,D,E$ be compact categories. The following hold:
      \begin{enumerate}
      \item if the continuous functors $F,G\colon C\to D$ are isomorphic, then~$F^\dagger=G^\dagger$;\label{item:functoriality-dagger-1}
      \item if $F$ is the identity functor $C\to C$,
        then $F^\dagger$ is the identity $C^\dagger\to C^\dagger$;\label{item:functoriality-dagger-2}
      \item for any functors $F\colon C\to D$ and $G\colon D\to E$,
        we have $G^\dagger\circ F^\dagger=(G\circ F)^\dagger$;\label{item:functoriality-dagger-3}
      \item if $F:C\to D$ is a continuous equivalence,
        then $F^\dagger$ preserves labels.\label{item:functoriality-dagger-4}
      \end{enumerate}
    \end{Prop}

    \begin{proof}
      If $\eta\colon F\Rightarrow G$
      is a natural isomorphism,
      then, for every $\varphi\in C(c,d)$,
      one has $G(\varphi)=\eta_d\circ F(\varphi)\circ\eta_c^{-1}$
      and $F(\varphi)=\eta_d^{-1}\circ G(\varphi)\circ\eta_c$,
      thus $F(\varphi)\mathrel{\J}G(\varphi)$.
      This establishes the first item.
      Items~\ref{item:functoriality-dagger-2}
      and~\ref{item:functoriality-dagger-3}
      are immediate.

      Concerning the last item, it is immediate that if $\varphi$ is regular then $F(\varphi)$  is regular, where $F\colon C\to D$ is a functor.
      If $F$ is an equivalence with pseudo-inverse
      $G$, then $G(F(\varphi))\mathrel{\J}\varphi$,
      and so if $F(\varphi)$ is regular
      then so is $\varphi$.           
      Finally, for every morphism $\varphi$
      of $C$, if $H_\varphi$
      and $H_{F(\varphi)}$
      are respectively the $\H$-classes
      of $\varphi$
      and $F(\varphi)$,
      then
      one clearly has $F(T(H_\varphi))\subseteq T(H_{F(\varphi)})$,
      with equality if $F$ is an equivalence.
      This induces
      a well defined map $\Gamma(H_\varphi)\to \Gamma(H_{F(\varphi)})$
      assigning each class $[\alpha]_{\approx_{H_\varphi}}$
      to~$[F(\alpha)]_{\approx_{H_F(\varphi)}}$,
      such map being a bijection if $F$ is an equivalence.
      This map is continuous,
      as we are dealing with compact quotients of compact spaces.
      Therefore, the Sch\"utzenberger groups
      of $\varphi$ and $F(\varphi)$
      are indeed isomorphic compact groups.
    \end{proof}
    
    \begin{Cor}\label{c:functoriality-dagger}
      If $C$ and $D$ are equivalent compact categories,
      then $C^\dagger$ and $D^\dagger$
      are isomorphic labeled topological posets.
    \end{Cor}

    Next we show the last piece needed
    for the proof of Proposition~\ref{p:invariance-of-labeled-posets}.
    
    \begin{Prop}\label{p:k-lu}
      The mapping $P_{\Cl X}\colon\Kar(\Mir_{\pv V}(\Cl X))^\dagger
      \to LU(\Mir_{\pv V}(\Cl X))^\dagger$
      defined by $P_{\Cl X}([(e,u,f)]_\J)=[u]_\J$ is an isomorphism of labeled topological posets,  for every subshift $\Cl X$ of $A^{\ZZ}$.
    \end{Prop}

    \begin{proof}
      It is trivial that $P_{\Cl X}$ is surjective.
      Let us check that it satisfies the remaining conditions
      for being an isomorphism of posets.
      Take morphisms $(e,u,f)$  and $(e',v,f')$
      of $\Mir_{\pv V}(\Cl X)$.
      Suppose $u=xvy$ for some $x,y\in\Om AV$.
      As $u=euf$ and $v=e'vf'$, we may assume that
      $x=exe'$ and $y=f'yf$,
      yielding $(e,u,f)=(e,x,e')(e',v,f')(f',y,f)$
      and $(e,u,f)\leq_\J(e',v,f')$.
      Conversely, if $(e,u,f)=(e,x,e')(e',v,f')(f',y,f)$
      then $u=xvy$. This shows that $P_{\Cl X}$ is
      a well defined isomorphism of posets.
      
      Because the map $(e,u,f)\mapsto u$
      is continuous,
      and since we are dealing with compact
      spaces and their compact quotients, the map $P_{\Cl X}$ is continuous.

      Fix a morphism $(e,u,f)$ of
      $\Kar(\Mir_{\pv V}(\Cl X))$.  Suppose that $u$ is regular. Then $u=uxu$
      for some $x\in\Om AV$. Since $u=uf=eu$, we may suppose
      that $x=fxe$, thus $(e,u,f)=(e,u,f)(f,x,e)(e,u,f)$
      and so $(e,u,f)$ is regular. Conversely,
      if $(e,u,f)$ is regular then it is immediate that $u$ is regular.

      It remains to show that the Sch\"utzenberger groups
      of $(e,u,f)$ and $u$ are isomorphic compact groups.
      Let $H$ be the $\H$-class of $(e,u,f)$ in $\Kar(\Mir_{\pv V}(\Cl X))$,
      and let $K$ be the $\H$-class of $u$
      in $\Om AV$.
      Suppose that $(f,x,f)$ and $(f,y,f)$
      are elements of $T(H)$ such that $(f,x,f)\mathrel{\approx_H}(f,y,f)$.
      This means that we have
      $(e,u,f)(f,x,f)=(e,u,f)(f,y,f)\in H$.
    Hence $ux=uy=zu$ for some~$z$, and, according to what we already saw in the first paragraph  of the proof, we also know that $ux$ and $u$ are $\J$-equivalent. Since $\Om AV$ is a stable semigroup, we conclude that
      $u\mathrel{\H}ux$,
      and, similarly, $u\mathrel{\H}uy$,
      thus $x\mathrel{\approx_K}y$.
      Therefore, we have a well defined map $\varphi\colon\Gamma(H)\to\Gamma(K)$ given by $\varphi([(f,x,f)]_{\approx_H})=[x]_{\approx_{K}}$.
      This map is continuous, again because we are dealing with compact
      spaces and their compact quotients.
      Moreover, $\varphi$ is clearly a homomorphism.
      
      Suppose that $x\in T(K)$.
      Then $u=uf$ yields $f\in T(K)$ and $fxf\in T(K)$, thus
      $[x]_{\approx_{K}}=[fxf]_{\approx_{K}}$ as $f=f^2$ and $\Gamma(K)$ is a group.
      Consider the equality $(e,u,f)(f,fxf,f)=(e,ux,f)$,
      entailing $(e,u,f)\leq_{\R}(e,ux,f)$ in $\Kar(\Mir_{\pv V}(\Cl X))$.
      Since $x$ is an arbitrary element of $T(K)$,
      we know that $ux$ may be any element of $K$,
      and so, by the symmetry of the $\H$-relation, we conclude that $(e,u,f)\mathrel{\H}(e,ux,f)$.
      Therefore, $(e,u,f)(f,fxf,f)=(e,ux,f)$
      implies that $(f,fxf,f)\in T(H)$.
      We deduce that $[x]_{\approx_{K}}=\varphi([(f,fxf,f)]_{\approx_H})$,
      and so $\varphi$ is onto. 

      If $(f,x,f)$ is an element of $T(H)$ such that $[x]_{\approx_K}$
      is the identity of $\Gamma(H)$, then we have $ux=u$,
      implying $(e,u,f)(f,x,f)=(e,u,f)$.
      The latter equality entails that
      $[(f,x,f)]_{\approx_H}$ is the identity of $\Gamma(H)$.
      This shows that $\varphi$ is a continuous isomorphism of compact
      groups, concluding the proof.
    \end{proof}

    \begin{proof}[Proof of Proposition~\ref{p:invariance-of-labeled-posets}]
      According to Propositions~\ref{p:k-lu}
      and~\ref{p:functoriality-dagger},
      the mapping $P_{\Cl Y}\circ F^\dagger\circ P_{\Cl X}^{-1}$
      is an isomorphism
      between the labeled topological posets
      $LU(\Mir_{\pv V}(\Cl X))^\dagger$ and~$LU(\Mir_{\pv V}(\Cl Y))^\dagger$,
      which restricts to an isomorphism
      between
      $LU(\Sha_{\pv V}(\Cl X))^\dagger$ and $LU(\Sha_{\pv V}(\Cl Y))^\dagger$.
    \end{proof}

\bibliographystyle{amsalpha}


\end{document}